\documentclass[11pt]{amsart}

\usepackage{amsmath,amsthm,amssymb,mathtools}
\usepackage{color,tikz-cd,hyperref,enumitem}
\hypersetup{colorlinks=true,linkcolor=blue,citecolor=blue,urlcolor=blue,linktocpage}
\usetikzlibrary{patterns}
\numberwithin{equation}{section}
\newtheorem{thm}{Theorem}[section]
\newtheorem{prp}[thm]{Proposition}
\newtheorem{lem}[thm]{Lemma}
\newtheorem{cor}[thm]{Corollary}

\newtheorem{thm-intro}{Theorem}
\newtheorem{cor-intro}[thm-intro]{Corollary}

\theoremstyle{definition}
\newtheorem{dfn}[thm]{Definition}

\newtheorem{rmk}[thm]{Remark}
 
\theoremstyle{remark}

\newcommand{\mx}[1]{\begin{pmatrix}#1\end{pmatrix}}
\newcommand{\vb}{\,|\,}

\newcommand{\Z}{\mathbb{Z}}

\newcommand{\R}{\mathbb{R}}

\newcommand{\D}{\mathbb{D}}

\newcommand{\cA}{\mathcal{A}}
\newcommand{\cB}{\mathcal{B}}
\newcommand{\cH}{\mathcal{H}}
\newcommand{\cC}{\mathcal{C}}
\newcommand{\cD}{\mathcal{D}}
\newcommand{\cE}{\mathcal{E}}

\newcommand{\cG}{\mathcal{G}}

\newcommand{\cS}{\mathcal{S}}
\newcommand{\cW}{\mathcal{W}}
\newcommand{\cP}{\mathcal{P}}

\newcommand{\cM}{\mathcal{M}}

\newcommand{\colim}{\textup{colim}}

\newcommand{\hocolim}{\textup{hocolim}}
\newcommand{\Cone}{\textup{Cone}}

\newcommand{\Hom}{\textup{Hom}}

\newcommand{\Tw}{\textup{Tw}}

\newcommand{\dgCat}{\textup{dgCat}}

\newcommand{\Ho}{\textup{Ho}}
\newcommand{\intr}{\mathrm{int}}
\newcommand{\s}{\mathrm{sgn}}

\title{The wrapped Fukaya category of plumbings}
\author{Dogancan Karabas}
\address[Dogancan Karabas]{Kavli Institute for the Physics and Mathematics of the Universe (WPI), The University of Tokyo Institutes for Advanced Study, The University of Tokyo, Kashiwa, Chiba 277-8583, Japan}
\email{dogancan.karabas@ipmu.jp}

\author{Sangjin Lee}
\address[Sangjin Lee]{Korea Institute for Advanced Study, 85 Hoegiro Dongdaemun-gu, Seoul 02455, Republic of Korea}
\email{sangjinlee@kias.re.kr}

\begin{document}

\maketitle

\begin{abstract}
	Plumbing spaces have drawn significant attention among symplectic topologists due to their natural occurrence as examples of Weinstein manifolds. 
	In our paper, we provide an explicit general formula for the wrapped Fukaya category of plumbings (with arbitrary grading structure) of cotangent bundles along any quiver, in terms of homotopy types of cotangent bundles.
	The resulting category is freely generated by finitely many morphisms with differentials. 
	Our approach relies on ``local-to-global" computations. Especially, we give a specific presentation of the wrapped Fukaya category of ``plumbing sectors" that serve as local models for the singularities of Lagrangian skeletons of plumbing spaces.
	As corollaries, we explicitly describe the wrapped Fukaya category of plumbing spaces in dimension $4$ and plumbings of $T^*S^n$ for $n \geq 3$. 
	We show that any Ginzburg dg algebra/category of a graded quiver without potential is equivalent to the wrapped Fukaya category of a plumbing of $T^*S^n$ (with the corresponding grading structure).
\end{abstract}

\tableofcontents

\section{Introduction}
\label{section introduction}

Fukaya categories have had numerous significant impacts in symplectic topology and modern mathematics: It is not only a powerful invariant of symplectic manifolds explained in, for example, \cite{Seidel08, Fukaya-Oh-Ohta-Ono09a, Fukaya-Oh-Ohta-Ono09b}, but also related to homological mirror symmetry, Gromov-Witten theory, quantum cohomology, topological quantum field theory, etc. 

One of the obstructions to study Fukaya category is its computational difficulty. 
However, in the case of Weinstein manifolds and their ``wrapped" Fukaya categories, there is a well-established strategy for computation. 
More precisely, Kontsevich conjectured that the wrapped Fukaya category of a Weinstein manifold is equivalent to the microlocal sheaf category of its Lagrangian skeleton.
Since a microlocal sheaf category can be computed via a ``local-to-global" approach, it would be natural to expect that the wrapped Fukaya categories also can be computed in the same manner.
The expectation is proven to be true by Ganatra, Pardon and Shende \cite{gps1, gps2}. 

The aim of the current paper is to compute wrapped Fukaya categories of plumbing spaces via the local-to-global approach, using the techniques established in \cite{hocolim, Karabas-Lee24}.
We note that plumbing spaces admit ``simple" Lagrangian skeletons, where ``simple" refers to the fact that these skeletons are unions of smooth Lagrangian submanifolds intersecting transversally.
This simplicity makes plumbing spaces ideal examples for demonstrating ``local-to-global" computations. 

Plumbing spaces are interesting spaces to study, even without mentioning their simple Lagrangian skeletons. 
In the literature, numerous known results shed light on these spaces. 
For instance, Abouzaid \cite{Abouzaid11} computed the Lagrangian Floer homology of two zero sections in a plumbing space, where the plumbing space is obtained by plumbing two cotangent bundles whose zero sections intersect clearly. 
Abouzaid and Smith \cite{Abouzaid-Smith12} investigated generation results for the compact Fukaya categories of plumbing spaces satisfying some conditions.
Etgu and Lekili \cite{Etgu-Lekili17, Etgu-Lekili19} and Ekholm and Lekili \cite{Ekholm-Lekili17} explored Koszul duality patterns within the wrapped and compact Fukaya categories of plumbing spaces, again satisfying some conditions.
Additionally, Lekili and Ueda \cite{lekili-ueda} studied the homological mirror symmetry of Milnor fibers with simple singularities, which can be seen as a particular class of plumbing spaces.
Recently, Hu, Lau, and Tan \cite{Hu-Lau-Tan24} studied localized mirrors of (framed) Lagrangian immersions in certain plumbing spaces of real dimension four, and Nakajima quiver varieties are constructed as the localized mirrors.  

To be more precise, let us briefly recall the construction of plumbing spaces. 
One can find more detailed explanations on the construction in Section \ref{section plumbing space}. 
Plumbing spaces are constructed by doing surgeries of cotangent bundles. 
The surgery information can be encoded in a quiver with extra data, which we call the {\em plumbing quiver} and {\em plumbing data}.  
More precisely, one can assign a cotangent bundle to each vertex of the given plumbing quiver. 
If there exists an arrow connecting two vertices, then one ``combines" two cotangent bundles assigned to those vertices.
The given plumbing data and the direction of the arrow determine the way of combining two cotangent bundles.

To compute the wrapped Fukaya category of plumbing spaces, we approach plumbing spaces from a ``local-to-global" perspective.
More precisely, each singularity of their Lagrangian skeletons can be viewed as two transversally intersecting half-dimensional disks.
We consider a local model around such singularities, which is referred to as the ``plumbing sector" in \cite[Section 6.2]{gps3}.
Note that, in \cite[Lemma 6.2]{gps3}, Ganatra, Pardon, and Shende provided a homotopy colimit diagram for the wrapped Fukaya category of the plumbing sector of any dimension $2n$.
Then, plumbing spaces can be constructed by gluing plumbing sectors with cotangent bundles.
Therefore, the wrapped Fukaya category of plumbing spaces can be obtained by ``gluing" the wrapped Fukaya category of plumbing sectors and that of cotangent bundles via a homotopy colimit diagram as outlined in \cite[Corollary 6.3]{gps3}.

The wrapped Fukaya category of cotangent bundles can be understood through the work of \cite{Abouzaid12, gps2} and is explained explicitly in Section \ref{section wrapped Fukaya categories of Weinstein sectors} for cotangent bundles of punctured $n$-spheres and surfaces.
We determine the wrapped Fukaya category of the plumbing sector in Section \ref{sec:plumbing-sectors} (Theorem \ref{thm:wfuk-plumbing-3} for $n\geq 3$ and Theorem \ref{thm:wfuk-plumbing-2} for $n=2$) together with the inclusions of the wrapped Fukaya categories of its two boundary sectors $T^*S^{n-1}$, by explicitly computing the homotopy colimit and choosing ``nice'' generators.
Here, by ``nice" generators, we mean those that are the images of the natural generators, i.e., cotangent fibers of $T^*S^{n-1}$, under the inclusion functors above.
We note that the finite-dimensional representations resulting from these computations correspond to constructible sheaves on $\R^n$ with a stratification defined by the origin and its complement.
When $n=2$, these constructible sheaves are also described by \cite[Proposition 4.10]{combinatorics} and studied within the context of perverse sheaves in \cite{verdier-perverse-sheaf}.

In Sections \ref{sec:plumbing-sectors} and \ref{section wrapped Fukaya categories of Weinstein sectors}, all wrapped Fukaya categories, including the inclusion functors for wrapped Fukaya categories of their boundary sectors, are presented in a manner that facilitates explicit computation and simple presentation of the homotopy colimit determining the wrapped Fukaya category of plumbing spaces. Consequently, we prove the following theorem by computing the homotopy colimit in Section \ref{section wrapped Fukaya categories of plumbings}:
\begin{thm}[Theorem \ref{thm:wfuk-plumbing-general}]
	\label{thm main intro}
	Fix $n\geq 2$. Let $P$ be a plumbing of cotangent bundles of connected, oriented $n$-manifolds (with or without boundary) along any quiver with or without negative intersections.
	In other words, there is an arbitrary plumbing data $(Q,M,\s)$ such that $P = P(Q,M,\s)$. 
	Then, the wrapped Fukaya category of $P$ is, up to pretriangulated equivalence, given by
	\[\cW(P)\simeq\begin{cases}
		\cP_2[\{1+y_ex_e\vb e\in E(Q)\}^{-1}] & \text{if $n=2$,}\\
		\cP_n & \text{if $n\geq 3$,}
	\end{cases}\]
	where $\cP_n$ is a semifree dg category given as follows:
	\begin{enumerate}[label = (\roman*)]
		\item {\em Objects:} $L_v$ (representing a Lagrangian cocore dual to $M(v)$) for any $v \in V(Q)$.
		\item {\em Generating morphisms:} There are three types of generating morphisms: 
		\begin{itemize}
			\item For any $v \in V(Q)$, $h_v\colon L_v\to L_v$. 
			\item For any $v\in V(Q)$, the generating morphisms of $ C_{-*}(\Omega_p (M(v)\setminus\text{pt}))$, where \\ $C_{-*}(\Omega_p (M(v)\setminus\text{pt}))$ is considered as a semifree dg algebra, see Remark \ref{rmk:wfuk-plumbing}\eqref{item:based-loop-1} and \eqref{item:based-loop-2}.
			\item For any $e=v\to w\in E(Q)$, $x_e\colon L_v \to L_w,\quad y_e\colon L_w \to L_v$.
		\end{itemize}
		\item {\em Degrees:} $|h_v|=1-n,\quad |x_e|=0,\quad |y_e|=2-n$. 
		\item {\em Differentials:} $d x_e = d y_e= 0$, and
		\[dh_v=\begin{cases}
			\displaystyle \eta_v(z)\prod_{\substack{e= \bullet \to v\\ \s(e)=-1}} (1+x_ey_e)\prod_{e= v\to \bullet} (1+y_ex_e) - \prod_{\substack{e= \bullet \to v\\ \s(e)=1}}(1+x_ey_e) &\text{if $n=2$,}\\
			\displaystyle \eta_v(z)+\sum_{e= v\to \bullet} y_ex_e + \sum_{e= \bullet \to v}(-1)^{n(n-1)/2}\s(e)\, x_ey_e &\text{if $n\geq 3$,}
		\end{cases}\]
	\end{enumerate}
	where $\eta_v(z)\in C_{-*}(\Omega_p (M(v)\setminus\text{pt}))$ is as in \eqref{eq:eta-inclusion}.
\end{thm}
\begin{rmk}
	We would like to note that Theorem \ref{thm:wfuk-plumbing-grading} is a generalized version of Theorem \ref{thm main intro}, since Theorem \ref{thm main intro} deals with wrapped Fukaya categories equipped with the ``standard'' grading structure but Theorem \ref{thm:wfuk-plumbing-grading} considers an arbitrary grading structure.
\end{rmk}

It is worth noting that cotangent bundles can be seen as special plumbing spaces whose plumbing quiver is a point without arrows. 
Abouzaid \cite{Abouzaid12} described the wrapped Fukaya category of cotangent bundles in terms of their homotopy type.
Theorem \ref{thm main intro} can be seen as a generalization of the result in \cite{Abouzaid12} to plumbing spaces.

As corollaries of Theorem \ref{thm main intro}, we provide an explicit computation for plumbing spaces of $T^*S^n$ for $n\geq 3$ (Corollary \ref{cor homotopy colimit for sphere plumbing}) and cotangent bundles of oriented surfaces (Corollary \ref{cor:wfuk-plumbing-surface}), with arbitrary grading structures.
Then, we can see that the resulting wrapped Fukaya categories are equivalent to Ginzburg dg categories (Corollary \ref{cor:ginzburg-equal-wfuk}) and derived multiplicative preprojective algebras (Corollary \ref{cor:multiplicative-preprojective-algebra}), respectively.
As for the other direction, any Ginzburg dg category of a graded quiver without potential or any derived multiplicative preprojective algebra can be realized as such a wrapped Fukaya category.
Consequently, plumbing spaces can be seen as geometric models to investigate Ginzburg dg categories and derived multiplicative preprojective algebras.
For example, we can reprove a result of \cite{keller-calabi-yau} that Ginzburg dg categories are smooth Calabi-Yau categories, see Corollary \ref{cor:ginzburg-smooth-cy}.

From the viewpoint of Homological Mirror Symmetry (HMS) conjecture, Theorem \ref{thm main intro} provides the $A$-side of HMS.
Recall that HMS suggests an equivalence between the (wrapped) Fukaya category of $X$ and the category of coherent sheaves on a mirror variety $X^\vee$ (or LG model, etc.)
A future direction involves finding a variety $X^\vee$ on $B$-side such that the category of coherent sheaves on $X^\vee$ is equivalent to the computed wrapped Fukaya category of a plumbing $X$. 
Especially, if one considers a plumbing of $T^*S^n$, then the category on $B$-side should be equivalent to the Ginzburg dg category of a graded quiver without potential, which can be seen as a generalization of the result in \cite{lekili-ueda}.
%

Before concluding the introduction, it is worth noting that the wrapped Fukaya categories of ``some" plumbing spaces are known by \cite{Etgu-Lekili17, Etgu-Lekili19, Asplund21}. 
More specifically, Etgu and Lekili \cite{Etgu-Lekili17, Etgu-Lekili19} computed wrapped Fukaya categories of plumbing spaces of {\em dimension $4$}, and Asplund \cite{Asplund21} computed wrapped Fukaya categories of plumbing spaces of $T^*S^n$ with $n \geq 3$ without considering intersection signs of plumbing points. 

In the current paper, we are employing a different method from \cite{Etgu-Lekili17, Etgu-Lekili19, Asplund21}. 
The previous works used Chekanov-Eliashberg dg algebras of Legendrians for the computation, but we follow a local-to-global approach. 
As a consequence, we can take care of broader cases; for example, Theorem \ref{thm main intro} computes wrapped Fukaya categories with nonstandard grading structures or arbitrary commutative ring coefficients, rather than being limited to particular grading structures or field coefficients.

\subsection{Acknowledgment}
\label{subsection acknowledgment}
We extend our deepest gratitude to the anonymous referee and the editor for their invaluable suggestions and insightful comments.

The first-named author is supported by World Premier International Research Center Initiative (WPI), MEXT, Japan.
The second-named author is supported by a KIAS Individual Grant (MG094401) at Korea Institute for Advanced Study.

\section{Preliminaries}

We review categorical preliminaries in Section \ref{sec:preliminaries-dg-categories} and symplectic preliminaries in Section \ref{sec:preliminaries-symplectic}.
We point out that \cite[Section 2]{Karabas-Lee24} covers similar material to Section \ref{sec:preliminaries-dg-categories} with some overlaps, but in greater detail. 

\subsection{Dg categories}\label{sec:preliminaries-dg-categories}
A {\em differential graded (dg) category} is a category enriched over the symmetric monoidal category of complexes over a fixed commutative ring $k$.
It can also be viewed as an $A_{\infty}$-category in which compositions of order greater than 2 are set to vanish.
For further details on dg categories, readers may refer to \cite{dgcat}, and for a review of $A_{\infty}$-categories, one can consult \cite{seidel}.
We will use $\hom^*$ for the morphism (cochain) complex, $d$ for differential, $\circ$ for compositions of morphisms, and we will omit it whenever it is convenient.
When introducing a dg category, we will follow the convention of providing the following five items:
\begin{enumerate}[label = (\roman*)]
	\item {\em Objects:} We list the objects in the category.
	\item {\em Generating morphisms:} We give a set of generating morphisms. They generate all the morphisms as an algebra, not as a module. We will not explicitly mention the existence of identity morphisms, but it should be understood that every object has the identity endomorphism.
	\item {\em Degrees:} For each generating morphism, we specify its degree.
	\item {\em Differentials:} For each generating morphism, we specify its differential.
	\item {\em Relations:} We specify the relations between generating morphisms. This item will be omitted if the generating morphisms freely generate all other morphisms.
\end{enumerate}

\subsubsection{Weak equivalence}

Let $\dgCat$ denote the category of small dg categories, where morphisms are dg functors. We aim to invert certain dg functors, referred to as {\em weak equivalences}, in $\dgCat$. 
The resulting categories can be studied by introducing {\em model structures} on $\dgCat$, making $\dgCat$ a {\em model category}. See \cite{tabuada-model} for a review of model structures for $\dgCat$. 

First, we introduce the following notation and definition:
\begin{itemize}
	\item For a given dg category $\cC$, $\Tw\,\cC$ is the dg category of twisted complexes in $\cC$, which is a pretriangulated envelope of $\cC$.
	\item A {\em pretriangulated equivalence} is a dg functor from $\cC$ to $\cD$, which induces a quasi-equivalence from $\Tw \, \cC$ to $\Tw \, \cD$.
\end{itemize}

Then, we choose weak equivalences as either {\em quasi-equivalences} or {\em pretriangulated equivalences}.
By inverting quasi-equivalences (resp.\ pretriangulated equivalence), we obtain the homotopy category associated to the Dwyer--Kan model structure (resp.\ quasi-equiconic model structure) $\Ho_{qe}(\dgCat)$ (resp.\ $\Ho_{tr}(\dgCat)$).

We say two dg categories are equivalent in the following sense: 
\begin{dfn}
	\label{dfn:equivalence}
	Let $\cC$ and $\cD$ be dg categories.
	\begin{enumerate}
		\item $\cC$ and $\cD$ are {\em quasi-equivalent} if there is a chain of dg categories and dg functors
		\[\cC\xleftarrow{\sim}\cC'\xrightarrow{\sim}\cD\]
		for some dg category $\cC'$, where each dg functor in the chain is a quasi-equivalence.
		
		\item $\cC$ and $\cD$ are {\em pretriangulated equivalent} if $\Tw\,\cC$ and $\Tw\,\cD$ are quasi-equivalent.
	\end{enumerate}
\end{dfn}

Using Definition \ref{dfn:equivalence}, one can define the notion of {\em generation}.
\begin{dfn} Let $\cC$ be a dg category. Let $\{L_i\}$ be a collection of objects in $\cC$, and let $\cD$ be the full dg subcategory of $\cC$ whose object set is $\{L_i\}$. Then, $\{L_i\}$ {\em generates} $\cC$ if $\cD$ is pretriangulated equivalent to $\cC$.
\end{dfn}

\subsubsection{Semifree dg category and homotopy colimit}
\label{subsubsection semifree dg category and homotopy colimit}
We introduce a particular class of dg categories and dg functors that will play a crucial role throughout the paper.
For more details, see \cite{hocolim, Karabas-Lee24}.
\begin{dfn}\label{dfn:semifree}\mbox{}
	\begin{enumerate}
		\item A dg category $\cC$ is called {\em semifree} if its morphisms, treated as a non-commutative algebra, are freely generated (as a non-commutative algebra) by a set of morphisms $\{f_i\}$ (indexed by an ordinal), with the condition that $df_i$ is generated by the set $\{f_j\vb j<i\}$.
		
		\item\label{item:semifree-extension} A dg functor $F\colon \cC\to\cD$ is called a {\em semifree extension} by a set of objects $R$ and a set of morphisms $S=\{f_i\}$ if it satisfies the following conditions:
		\begin{itemize}			
			\item $F$ is an inclusion.						
			\item The objects of $F(\cC)$, along with $R$, form the objects of $\cD$.						
			\item The morphisms of $\cD$, treated as an algebra, can be expressed as a free non-commutative extension of the morphisms of $F(\cC)$ by $\{f_i\}$ (indexed by an ordinal), with the condition that $df_i$ is generated by the morphisms of $F(\cC)$ and $\{f_j\vb j<i\}$.		
		\end{itemize}			
		\item A dg category $\cD$ is called a {\em semifree extension} of a dg category $\cC$ by a set of objects $R$ and a set of morphisms $S$ if there exists a semifree extension $F\colon\cC\to\cD$ as in Definition \ref{dfn:semifree}\eqref{item:semifree-extension}.
	\end{enumerate}
\end{dfn}
It is worth noting that a semifree dg category has no relations under composition.
But relations arise on the level of homology, since a semifree dg category can have a non-trivial differential.
We also note that every dg category has a semifree resolution (see \cite{drinfeld}).
In the most of the paper, dg categories are semifree.
In the rest of Section \ref{subsubsection semifree dg category and homotopy colimit}, we discuss the advantages of considering semifree dg categories. 
Finally, we would like to point out that semifree dg categories are fibrant/cofibrant objects in Dwyer--Kan model structure and quasi-equiconic model structure.
See, for example, \cite{drinfeld, hocolim}.

When $\cC$ is a dg category, and $S$ is a subset of closed degree zero morphisms in $\cC$, there exists a dg category $\cC[S^{-1}]$, known as the {\em dg localization} of $\cC$ at the morphisms in $S$. This localization is essentially obtained from $\cC$ by inverting morphisms in $S$. For a precise definition, one can refer to sources such as \cite{toen} or \cite{hocolim}. The dg localization is unique up to quasi-equivalence, and its existence is established in \cite{toen-morita}.
One advantage of considering semifree dg category $\cC$ is that we can explicitly describe $\cC[S^{-1}]$.
\begin{prp}[\cite{hocolim}]\label{prp:localise-semifree}
	When $\cC$ is a semifree dg category, and $S = \{g_i \colon A_i \to B_i\}$ is a subset of closed degree zero morphisms in $\cC$, the dg localization $\cC[S^{-1}]$ can be viewed as the semifree extension of $\cC$ by the morphisms $g'_i, \hat{g}_i, \check{g}_i, \bar{g}_i$
		\[\begin{tikzcd}
			A_i\ar[loop left,"\hat g_i"]\ar[r,"g_i"]\ar[r,"\bar g_i",bend left=60] & B_i\ar[l,"g'_i",bend left=30]\ar[loop right,"\check g_i"]
		\end{tikzcd}\]
		for each $i$, whose gradings and differentials are 
		\begin{gather*}
			|g_i'|=0, |\hat g_i|=|\check g_i|=-1, |\bar g_i|=-2, \text{  and  }
			dg_i'=0, d\hat g_i=1_{A_i}-g_i'g_i, d\check g_i=1_{B_i}-g_ig_i', d\bar g_i=g_i\hat g_i - \check g_ig_i .
		\end{gather*}
\end{prp}

Similar to localizations, semifreeness allows us to have an explicit formula for colimits.
Moreover, because of the semifreeness, the colimits are homotopy colimits that we would like to compute. 
See Remark \ref{rmk colimit}.
\begin{prp}[\cite{hocolim}]\label{prp:colimit-dg}
	For dg categories $\cA, \cB, \cC$, with a dg functor $\alpha\colon\cC\to\cA$ and a semifree extension $\beta\colon\cC\to\cB$ by a set of objects $R$ and a set of morphisms $S$, the colimit
	\[\cD:=\colim\left(
	\begin{tikzcd}
		\cA&\cC\ar[l,"\alpha"]\ar[r,"\beta"'] & \cB	 
	\end{tikzcd}\right)\]
	is the semifree extension of $\cA$ by the same sets $R$ and $S$. In $\cD$, the domains, codomains, gradings, and differentials of the morphisms in $S$ are determined in a straightforward manner.
\end{prp}

\begin{rmk}
	\mbox{}
	\begin{enumerate}
		\item In a more casual sense, the colimit in Proposition \ref{prp:colimit-dg} can be thought as the disjoint union $\cA\amalg\cB$ with the identification of the images of $\alpha$ with the images of $\beta$.
		\item In the context of a semifree dg category $\cC$ and a generating morphism $f\colon A\to B$ in $\cC$, where $A\neq B$, the dg localization $\cC[f^{-1}]$ can be understood by considering $\cC$ with the identifications $A=B$ and $f=1_{A=B}$. This description relies on a description of dg localization through a colimit diagram, as presented in \cite{toen-morita}, and Proposition \ref{prp:colimit-dg}.
	\end{enumerate}
\end{rmk}

We also note that there exists a specific formula computing {\em homotopy colimit functor for semifree dg categories}, which is proven in \cite{hocolim, Karabas-Lee24}.
We will state the formula in below for reader's convenience, but we need preparations.

\begin{dfn}
	Let $\dgCat$ be the category of dg categories with weak equivalences chosen as quasi-equivalences or pretriangulated equivalences, and let $\dgCat^I$ be the category of functors $I\to\dgCat$, i.e., $I$-diagrams, where $I$ is an index category, such that weak equivalences in $\dgCat^I$ are the objectwise weak equivalences. The {\em homotopy colimit functor}
	\[\hocolim\colon\Ho_{qe}(\dgCat^I)\to\Ho_{qe}(\dgCat) \text{  (resp.\  } \hocolim\colon\Ho_{tr}(\dgCat^I)\to\Ho_{tr}(\dgCat))\]
	is defined (up to natural equivalence) as the total left derived functor of the colimit functor
	\[\colim\colon\dgCat^I\to\dgCat.\]
\end{dfn}

\begin{thm}
	\label{thm:hocolim-functor-dg}
	Let $\cA, \cB, \cC \in \dgCat$ be semifree dg categories, and let $\alpha, \beta$ be dg functors. 
	\begin{enumerate}
		\item (\cite{hocolim}) The homotopy colimit
			\[\hocolim(\begin{tikzcd}
				\cA & \cC \ar[l, "\alpha"'] \ar[r, "\beta"] & \cB.
			\end{tikzcd})\] 
			is given as a semifree extension of $\cA \amalg \cB$ by a set of morphisms of the following two types: 
			\begin{itemize}
				\item For any object $X \in \cC$, we add a morphism $t_X: \alpha(X) \to \beta(X)$, and then we take a dg localization of $t_X$. 
				\item For any generating morphism $f: X \to Y$ of $\cC$, we add $t_f:\alpha(X) \to \beta(Y)$ such that $|t_f| = |f|-1$ and
				\begin{gather}
					\label{eqn correction term}
					dt_f = (-1)^{|f|} \left(\beta(f) \circ t_X - t_Y \circ \alpha(f)\right) + \text{  correction term}.
				\end{gather}
			\end{itemize}
			We note that the correction term in Equation \eqref{eqn correction term} is determined by $df$ in $\cC$, as explained in \cite{hocolim}. 
			Especially, if $df=0$, then the correction term is $0$. 
		\item Let $S$ be a collection of degree $0$ closed morphisms in $\cC$. 
			Then, 
			\[\hocolim\left(\cA \leftarrow \cC[S^{-1}] \rightarrow \cB\right) \simeq \hocolim\left(\cA \leftarrow \cC \rightarrow \cB\right).\]  
		\item (\cite{Karabas-Lee24}) Let $F$ be a morphism in $\dgCat^I$, i.e., $F$ is given as the following commuting diagram 
		\[\begin{tikzcd}
			\cA \ar[d, "F_\cA"] & \cC \ar[l, "\alpha"'] \ar[r, "\beta"] \ar[d, "F_{\cC}"] & \cB \ar[d, "F_\cB"] \\
			\cA' & \cC' \ar[l, "\alpha'"'] \ar[r, "\beta'"]  & \cB' 
		\end{tikzcd}.\] 
		Then, the induced functor between homotopy colimits 
			\[\hocolim(F): \hocolim\left(\cA \xleftarrow{\alpha} \cC \xrightarrow{\beta} \cB\right) \to \hocolim\left(\cA' \xleftarrow{\alpha'} \cC' \xrightarrow{\beta'} \cB'\right),\] 
		is given by 
		\[\hocolim(F)\vert_{\cA \sqcup \cB} = F_\cA \sqcup F_\cB, \quad \hocolim(F)(t_X) = t_{F(X)},\quad \text{and} \quad \hocolim(F)(t_f)= t_{F(f)},\]
		where $t_{F(f)}$ is defined in \cite{Karabas-Lee24}.
	\end{enumerate}
\end{thm}

\begin{rmk}
	\label{rmk colimit}
	We note that if $\cA, \cB, \cC$ are semifree dg categories and either $\alpha: \cC\to \cA$ or $\beta: \cC \to \cB$ is a semifree extension, then 
	\[\hocolim\left(\cA \xleftarrow{\alpha} \cC \xrightarrow{\beta} \cB \right) \simeq \colim\left(\cA \xleftarrow{\alpha} \cC \xrightarrow{\beta} \cB \right).\]
\end{rmk}

Finally, we introduce two useful propositions from \cite{chekanov, subcritical}.
We will use them to simplify semifree dg categories, because they can be thought as ``basis change'' and ``cancellation of generators'' for the morphisms of semifree dg categories, respectively.

\begin{prp}\label{prp:variable-change}
	Let $\cC$ be a semifree dg category with a set of generating morphisms $\{f_i\}$ (indexed by an ordinal).
	For a unit $u_i \in k$ and a morphism $g_i$ in $\cC$ generated by $\{f_j\vb j<i\}$, let 
	\[\tilde f_i:=u_i f_i + g_i.\] 
	Then, the set $\{\tilde f_i\}$ also generates the morphisms in $\cC$ semifreely.
\end{prp}

\begin{prp}\label{prp:destabilisation}
	Let $\cC$ be a semifree dg category, and $\cD$ be the semifree extension of $\cC$ by the morphisms $\{a_i,b_i\}$ such that $da_i=b_i$ for all $i$. Then, $\cC$ and $\cD$ are quasi-equivalent. 
\end{prp}

\subsection{Symplectic topology}\label{sec:preliminaries-symplectic}
Now, we recall some basic notions in symplectic topology. 
First of all, we recall that a {\em Liouville domain} means a pair $(W, \lambda)$ such that 
\begin{itemize}
	\item $W$ is a $2n$-dimensional compact exact symplectic manifold with boundary equipped with the symplectic form $\omega := d \lambda$, 
	\item on the boundary of $W$, $\lambda|_{\partial W}$ is a contact form, and 
	\item the orientation of $\partial W$ by the form $\lambda \wedge (d\lambda)^{n-1}$ coincides with its orientation as the boundary of symplectic manifold $(W,\omega)$. 
\end{itemize}
One can always {\em complete} a Liouville domain $W$ by attaching a cylindrical end. The completion $\hat{W}$ is given as 
\[\hat{W}:= W \cup \left(\partial W \times [0,\infty)\right) .\]
On the cylindrical end $\partial W \times [0,\infty)$, the Liouville form is given as $e^r \lambda$ where $r$ denotes the second factor of the product space. We note that the completion $\hat{W}$ is a {\em Liouville manifold}, and $\partial_{\infty}\hat W:=\partial W$ is called its {\em ideal contact boundary}.

A Liouville domain/manifold admits the {\em skeleton}, or the {\em core} as follows: 
Let $Z$ be the Liouville vector field and $Z^t$ be the time $t$-flow of $Z$. 
In other words, $Z$ is the dual vector field of $\lambda$ with respect to the symplectic 2-form $d \lambda$. 
Then, the skeleton of a Liouville domain $W$ is given as 
\[\mathrm{Skel}(W, \lambda) := \bigcap_{t >0} Z^{-t}(W).\]
If $W$ is a Liouville manifold, then its skeleton is defined as that of a Liouville domain $W_0$ whose completion is $W$.

A {\em Weinstein domain/manifold} means a triple $(W, \lambda, \phi)$ such that $(W, \lambda)$ is a Liouville domain/manifold equipped with a Lyapunov function $\phi: W \to \mathbb{R}$ for the Liouville vector field dual to $\lambda$.
For convenience, we simply say that $W$ is a Weinstein manifold without mentioning a Liouville form $\lambda$ on it and a corresponding Lyapunov function. 

There exists a slightly generalized notion of Weinstein manifold, called {\em Weinstein sector}, which can be loosely understood as a Weinstein manifold with a boundary. 
Or equivalently, a Weinstein sector is a pair of a Weinstein manifold $W$ and a (possibly singular) isotropic subset $\Lambda \subset \partial_\infty W$, called {\em stop}. 
Further details will be given in Section \ref{subsection Weinstein pair}, and we also refer to \cite{gps1}.

Let $W$ (resp.\ $(W, \Lambda)$) be a Weinstein manifold (resp.\ Weinstein sector).
The main focus of this paper is a crucial invariant known as the {\em wrapped Fukaya category} $\cW(W)$ (resp.\ {\em partially wrapped Fukaya category} $\cW(W,\Lambda)$). This is an $A_{\infty}$-category with objects being certain exact Lagrangians with cylindrical ends equipped with additional data. Morphisms are generated by the intersections of Lagrangians after perturbing them through a process known as {\em wrapping}, and $A_{\infty}$-operations arise from counting pseudoholomorphic polygons bounded by Lagrangians (edges) and their intersections (corners). For a rigorous definition, see \cite{gps1, seidel}.

We note that any $A_{\infty}$-category $\cC$ can be regarded as a dg category up to quasi-equivalence by replacing $\cC$ with its image under the $A_{\infty}$-Yoneda embedding. Hence, wrapped Fukaya categories can be regarded as dg categories up to quasi-equivalence.

\begin{rmk}\label{rmk:grading-pin}
	The wrapped Fukaya category $\cW(W)$ (or $\cW(W,\Lambda)$) can be given $\Z$-grading if $2c_1(W)=0\in H^2(W;\Z)$. Also, the definition of $\cW(W)$ (or $\cW(W,\Lambda))$ depends on the classes $\eta\in H^1(W;\Z)$ (grading structure) and $b\in H^2(W;\Z/2)$ (background class), which are used to give gradings on the Lagrangian intersections and orientations of moduli spaces of pseudoholomorphic disks, respectively. See \cite{seidel} for more details, or \cite{cat-entropy} for a quick overview. We will use the notation $\cW(W;\eta)$ to denote the wrapped Fukaya category of $W$ with a grading structure $\eta$, specifically when we wish to emphasize the chosen grading structure. This notation will only appear in Section \ref{section wrapped Fukaya categories of plumbings}.
\end{rmk}

In the literature, several results regarding the wrapped Fukaya categories have been established. See Theorems \ref{thm:cocore-generate}, \ref{thm:gps}, and \ref{thm:cotangent-generation} below:

\begin{thm}[\cite{CDRGG,gps2}]\label{thm:cocore-generate}
	Let $W$ be a Weinstein manifold (or sector) of dimension $2n$. Consider a mostly Legendrian set $\Lambda\subset\partial_{\infty}W$ (see \cite{gps2} for the definition). Then
	\begin{enumerate}
		\item $\cW(W)$ is generated by the {\em Lagrangian cocores}, which are Lagrangian disks dual to $n$ dimensional strata of the skeleton of $W$,
		
		\item $\cW(W,\Lambda)$ is generated by the Lagrangian cocores, and linking disks of $\Lambda$.
	\end{enumerate}
\end{thm}

Given an inclusion of Liouville/Weinstein sectors $F\hookrightarrow W$, there is an induced $A_{\infty}$-functor $\cW(F)\to\cW(W)$ as described in \cite{gps1}. Then, we can state the following theorem:

\begin{thm}[\cite{gps2}]\label{thm:gps}
	Let $W$ be a Weinstein manifold (or sector). Suppose $W=W_1\cup W_2$ for some Weinstein sectors $W_1$ and $W_2$, and $W_1\cap W_2$ is a hypersurface in $W$ whose neighborhood can be written as $F\times T^*[0,1]$, where $F$ is a Weinstein manifold/sector of codimension 2 (up to deformation). Then there is a pretriangulated equivalence
	\[\cW(W)\simeq\hocolim\left(
	\begin{tikzcd}
		\cW(W_1)  & \cW(F) \ar[l]\ar[r] & \cW(W_2)
	\end{tikzcd}
	\right) ,\]
	where the arrows are induced by the inclusion of Weinstein sectors $F\hookrightarrow W_i$ for $i=1,2$.
\end{thm}

We recall some known facts about cotangent bundles. The following definition can be found in \cite{Abouzaid12}:

\begin{dfn}\label{dfn:pontryagin-category}
	Let $M$ be a topological space. The {\em Pontryagin category} $\cP(M)$ of $M$ is a dg category (over a fixed commutative ring $k$) whose objects are the points of $M$, and the cochain complex of morphisms from $p$ to $q$ are given by the normalized cubical chains on $\Omega(p,q)$, i.e.,
	\[\hom^*_{\cP(M)}(p,q) :=C_{-*}(\Omega(p,q);k) ,\]
	where $\Omega(p,q)$ is the Moore path space, i.e.,
	\[\Omega(p,q):=\{\gamma\colon [0,R] \to M\vb \gamma(0) = p,\gamma(R) = q \text{ with } R \in (0, \infty]\}.\]
	The product on $\cP(M)$ is induced by the concatenation of Moore paths, which is strictly associative.
\end{dfn}

\begin{rmk}\label{rmk:loop-space}
	If $p,q\in\cP(M)$ correspond to two points in the same path component of $M$, then they are homotopy equivalent in $\cP(M)$. If $p$ and $q$ are in different path components of $M$, then $\hom^*(p,q)$ is the zero module. As a consequence, there is a quasi-equivalence
	\[\cP(M)\simeq\coprod_{[p]\in \pi_0(M)} C_{-*}(\Omega_p M)\]
	where $\Omega_p M$ is the based loop space of $M$ at $p$, $C_{-*}(\Omega_p M)$ is seen as a dg category with a single object whose endomorphism (dg) algebra is $C_{-*}(\Omega_p M)$, and we choose a representative point $p$ for each path component of $M$. In particular, if $M$ is path connected, we have a quasi-equivalence
	\[\cP(M)\simeq C_{-*}(\Omega_p M), \text{ for any  } p \in M.\]
\end{rmk}

\cite[Example 1.36]{gps2} implies the following generalization of \cite{abouzaid-wrapped-generation, Abouzaid12}:

\begin{thm}\label{thm:cotangent-generation}
	Let $M$ be a smooth manifold with or without a boundary. 
	Let $\cW(T^*M)$ is equipped with the standard grading structure and the background class as described in \cite{nadler-zaslow}.
	Then, there is a pretriangulated equivalence
	\begin{align*}
		\cP(M) \xrightarrow{\sim} \cW(T^*M), \qquad p \mapsto L_p:= T^*_pM.
	\end{align*}
\end{thm}

Using the ``higher'' Seifert-Van Kampen theorem in \cite[A.3]{lurie-ha}, we get the following theorem as described in the footnote for \cite[Example 1.36]{gps2}:

\begin{thm}\label{thm:seifert-van-kampen}
	Let $M$ be a smooth manifold (possibly with a boundary) or a sufficiently nice topological space as described in \cite{lurie-ha}. Let $\{M_1,M_2\}$ be an open covering of $M$. Then we have a quasi-equivalence
	\[\cP(M)\simeq\hocolim\left(\cP(M_1)\leftarrow \cP(M_1\cap M_2)\rightarrow \cP(M_2)\right)\]
	where the arrows are the obvious maps.
\end{thm}

We end this subsection with a proposition whose proof can be found in \cite{Karabas-Lee24}. 

\begin{prp}\label{prp:wfuk-sphere}
	Let $n\geq 1$. The wrapped Fukaya category of $T^*S^n$ is given, up to pretriangulated equivalence, by
	\[\cW(T^*S^n)\simeq 
	\begin{cases}
		\cC_1[z^{-1}] &\text{if }n=1\\
		\cC_n &\text{if }n\geq 2
	\end{cases}\]
	where $\cC_n$ is the semifree dg category given as follows:
	\begin{enumerate}[label = (\roman*)]
		\item {\em Objects:} $L$ (representing a cotangent fiber of $T^*S^n$).
		\item {\em Generating morphisms:} $z\in\hom^*(L,L)$.
		\item {\em Degrees:} $|z|=1-n$.
		\item {\em Differentials:} $dz=0$.
	\end{enumerate}
\end{prp}

\section{Plumbing spaces}
\label{section plumbing space}
In the main part of the paper, we discuss about the wrapped Fukaya categories of plumbing spaces of dimension $\geq 4$. 
Before starting the main discussion, we define the notion of plumbing spaces as gluings of Weinstein sectors in Section \ref{section plumbing space}.
We would like to point out that the plumbing space construction is equivalent to the conventional plumbing procedure that is explained in \cite[Chapter 7.6]{Geiges08} and \cite[Section 2.3]{Abouzaid11}, for example. 
In Remark \ref{rmk equivalence with the conventional plumbing}, we briefly discuss why our construction and the conventional construction match each other.

In the first two subsections, we give some preliminary knowledge. 
In Section \ref{subsection notation}, we define the notions of {\em plumbing data} and {\em plumbing sector}. 
In Section \ref{subsection construction of plumbing spaces}, we will describe how to construct a {\em plumbing space} from a plumbing data. 

\subsection{Weinstein pair and sector} 
\label{subsection Weinstein pair}
The main goal of Section \ref{subsection Weinstein pair} is to review the notion of {\em Weinstein pair} and to show the equivalence of the notions of Weinstein pair and {\em Weinstein sector}.
We note that one can find (original statements of) the most of this section, for example, Definitions \ref{dfn Weinstein hypersurface}, \ref{dfn Weinstein pair}, \ref{dfn adjusted}, and Proposition \ref{prp modified Liouville form}, in \cite[Chapters 2 and 3]{Eliashberg18} and the references therein.
We also refer the reader to \cite{Avdek21, Sylvan19, gps1, Ekholm-Lekili17}.
Lastly, see Section \ref{sec:preliminaries-symplectic} to review some definitions and facts about {\em Weinstein manifolds} and {\em Weinstein domains}.

We first define the notion of Weinstein hypersurfaces.
\begin{dfn}
	\label{dfn Weinstein hypersurface}
	Let $(Y, \xi)$ be a contact manifold. A codimension 1 submanifold $\Sigma \subset Y$ with boundary is called {\em Weinstein hypersurface} if there exists a contact one form $\lambda$ such that $(\Sigma, \lambda|_\Sigma)$ is a Weinstein domain.  
\end{dfn} 

\begin{rmk}
	\label{rmk skeleton of stops}
	It is known that the induced Weinstein structure on a Weinstein hypersurface $\Sigma$ depends on the choice of contact one form, but the {\em skeleton} of $\Sigma$ is independent of the choice. 
	For the proof, see \cite[Chapter 2]{Eliashberg18}.
	Moreover, the skeleton should be a stratified subset that consists of isotopic strata.
\end{rmk}

\begin{dfn}
	\label{dfn Weinstein pair}
	A {\em Weinstein pair} consists of a Weinstein domain $(W, \lambda)$ together with a Weinstein hypersurface $(\Sigma, \lambda|_\Sigma) \subset \partial W$. 
	Equivalently, one can define a Weinstein pair as a pair $(W, \Sigma)$ such that $W$ is a Weinstein manifold with cylindrical end and $\Sigma$ is a	Weinstein hypersurface in its ideal contact boundary.
\end{dfn}

We define the {\em core}, or the {\em skeleton}, of a Weinstein pair as follows: 
Let $(W, \lambda, \phi)$ be a Weinstein domain together with a Weinstein hypersurface $\Sigma \subset \partial W$. 
Let $\Lambda$ denote the skeleton of $\Sigma$ and let 
\[\hat{\Lambda}:= \cup_{t\geq 0} Z^{-t}(\Lambda),\]
where $Z$ is the Liouville vector field dual to $\lambda$. 
Then, the {\em core}, or {\em skeleton}, of a Weinstein pair $(W, \Sigma)$ is defined as 
$\mathrm{Skel}(W, \Sigma) := \mathrm{Skel}(W) \cup \hat{\Lambda}$.
Then, it is known that one can modify a Liouville form on $W$ so that 
\begin{gather}
	\label{eqn modified Liouville form}
	\mathrm{Skel}(W, \Sigma) = \cap_{t > 0} Z_0^{-t}(W) =: \mathrm{Skel}(W, \lambda_0),
\end{gather}
where $\lambda_0$ is the modified Liouville form and $Z_0$ is a Liouville vector field with respect to $\lambda_0$. 
To state the known fact more rigorously, let us present Definition \ref{dfn adjusted}.
\begin{dfn}
	\label{dfn adjusted}
	Let a Weinstein pair $(W, \Sigma)$ consist of a Weinstein domain $W = (W, \lambda, \phi)$ and a Weinstein hypersurface $(\Sigma, \lambda|_\Sigma, \phi_\Sigma) \subset \partial W$.
	Let $U$ be a neighborhood of $\Sigma$ in $\partial W$ such that 
	\[\Sigma \times [-\epsilon, \epsilon] \simeq U \subset \partial W, \text{  and  } \phi_\Sigma (\partial U) = 0. \]
	Let us denote by $Z_\Sigma$ the Liouville vector field on $\Sigma$ dual to $\lambda|_\Sigma$. 
	A Liouville 1-form $\lambda_0$, the Liouville vector field $Z_0$ dual to $\lambda_0$, and a smooth function $\phi_0$ are {\em adjusted} to the structure of Weinstein pair $(W, \Sigma)$ if the following hold:
	\begin{itemize}
		\item $Z_0$ is tangent to $\partial W$ on $U$ and transversal to $\partial W$ elsewhere, 
		\item $Z_0|_U = Z_\Sigma + u \tfrac{\partial}{\partial u}$, where $u$ is the interval coordinate of $U \simeq \Sigma \times [-\epsilon, \epsilon]$, 
		\item $\mathrm{Skel}(W, \lambda_0) = \mathrm{Skel}(W, \Sigma)$,
		\item the function $\phi_0: W \to \mathbb{R}$ is Lyapunov for $Z_0$ such that $\phi_0|_U= \phi_\Sigma + u^2$ and $\phi_0$ has no critical values bigger than or equal to the constant $\epsilon^2 = \phi_0(\partial U)$.
	\end{itemize}
\end{dfn}

Proposition \ref{prp modified Liouville form} claims the existence of modified Liouville form $\lambda_0$ satisfying Equation \eqref{eqn modified Liouville form}.
\begin{prp}[Proposition 2.9 of \cite{Eliashberg18}]\label{prp:pair-sector-correspondence}
	\label{prp modified Liouville form}
	Let $(W, \lambda, \phi)$ be a Weinstein domain and let $(W, \Sigma)$ be a Weinstein pair.
	There exist a Liouville form $\lambda_0$ for $W$ and a function $\phi_0: W \to \mathbb{R}$ such that 
	\begin{itemize}
		\item $\lambda_0, \phi_0$ are adjusted to $(W, \Sigma)$, and
		\item $\lambda_0$ coincides with $\lambda$ outside a neighborhood of $\Sigma$.
	\end{itemize} 
\end{prp} 

We note that $\phi_0(\partial U) = \epsilon^2$. 
If we set $W_0 := \phi_0^{-1}(\leq \epsilon^2)$, 
then there is no critical point of $\phi_0$ in $W \setminus W_0$. 
Thus, $W_0$ is a manifold with corners whose corners are along $\partial U$.
Moreover, $W_0$ is homeomorphic to $W$. 
Then, it is easy to check that one can see $W_0$ as a {\em Weinstein sector} whose convex completion (for the definition of convex completion, see \cite{gps2} and references therein) is the Weinstein pair $(W, \Sigma)$.
For the formal definition of Weinstein sector and more details, we refer the reader to \cite{gps1} and we omit the details. 
It implies that a Weinstein pair $(W,\Sigma)$ and a Weinstein sector $W_0$ are equivalent in the sense that one can recover from one the other. 

In this preliminary subsection, we reviewed various (equivalent) versions of a Weinstein pair/sector. 
In the rest of the paper, we will use the term Weinstein pair/sector for any of those versions.

\subsection{Gluing of Weinstein pairs/sectors}
\label{subsection gluing of Weinstein pairs/sectors}
Section \ref{subsection gluing of Weinstein pairs/sectors} reviews the gluing operation of Weinstein pairs/sectors.
For more details, we refer the reader to \cite{Avdek21, Eliashberg18}. 
We will use the same notation that we used in the previous section (Section \ref{subsection Weinstein pair}) without mentioning/defining.

Let $(W, \Sigma)$ and $(W', \Sigma')$ be two Weinstein pairs such that there exists a Weinstein isomorphism between $\Sigma$ and $\Sigma'$, 
\[F: (\Sigma, \lambda|_\Sigma, \phi|_\Sigma) \stackrel{\sim}{\to} (\Sigma', \lambda'|_{\Sigma'}, \phi'|_{\Sigma'}).\]
We consider the cornered version of the pairs $(W, \Sigma), (W', \Sigma')$. 
In other words, we consider a manifold with corners $W_0 \subset W$ (resp.\ $W_0' \subset W'$) equipped with a Liouville one form $\lambda_0$ (resp.\ $\lambda_0'$) and a function $\phi_0: W \to \mathbb{R}$ (resp.\ $\phi_0': W' \to \mathbb{R}$) such that $\lambda_0$ and $\phi_0$ (resp.\ $\lambda_0'$ and $\phi_0'$) are adjusted to $(W, \Sigma)$ (resp.\ $(W', \Sigma')$). 
We note that $W_0$ (resp.\ $W_0'$) has boundary of two types, one is a contact neighborhood of $\Sigma$ (resp.\ $\Sigma'$), denoted by $U$ (resp.\ $U'$), and the other is transversal to the Liouville vector field $Z_0$ (resp.\ $Z_0'$) dual to $\lambda_0$ (resp.\ $\lambda_0'$). 

Since $F$ is a Weinstein isomorphism, we extend $F$ to a contactomorphism between $U$ and $U'$. 
Let $F$ still denote the extended contactomorphism. 
We set 
\[W \sqcup_F W' := W_0 \sqcup W_0' / \{ \left(x \in U\right) \sim \left(F(x) \in U'\right)\}.\]
Then the Liouville forms $\lambda_0$ and $\lambda_0'$ and the Lyapunov functions $\phi_0$ and $\phi_0'$ can be glued together and define a Weinstein structure on $W \sqcup_F W'$.

\begin{dfn}
	\label{dfn gluing}
	With the above notation, the {\em gluing of Weinstein pairs} $(W, \Sigma)$ and $(W', \Sigma')$ is the Weinstein domain $(W \sqcup_F W', \lambda_F, \phi_F)$ where the gluing of $\lambda_0$ and $\lambda_0'$ (resp.\ $\phi_0$ and $\phi_0'$) is denoted by $\lambda_F$ (resp.\ $\phi_F$). 
\end{dfn}

\subsection{Plumbing data and plumbing sector}
\label{subsection notation}
Definition \ref{dfn plumbing data} defines the notion of pluming data. 

\begin{dfn}
	\label{dfn plumbing data}
	A {\em plumbing data} is a triple $(Q, M, \s)$ such that 
	\begin{itemize}
		\item $Q$ is a (finite) quiver, i.e., a directed graph, 
		\item $M$ is a map from the set of vertices of $Q$, denoted by $V(Q)$, to the collection of $n$-dimensional connected, oriented, smooth manifolds (with or without boundary), denoted by $\mathcal{O}_n$, i.e., 
		\[M: V(Q) \to \mathcal{O}_n,\]
		\item $\s$ is a map from the set of arrows, denoted by $E(Q)$, to $\{1, -1\}$, i.e., 
		\[\s: E(Q) \to \{1, -1\}. \]
	\end{itemize}
\end{dfn}
\begin{rmk}
	We note that in Definition \ref{dfn plumbing data}, $Q$ can be any graph, with or without loops, multiple edges between two vertices. 
	Moreover, it is also worthy note that by Definition \ref{dfn plumbing data}, plumbing spaces in the current paper mean plumbings of cotangent bundles of {\em connected, oriented, smooth} manifolds. 
	Also, manifolds can be compact or non-compact, and the cotangent bundles of non-compact manifolds would be seen as {\em open Liouville sectors} defined and studied in \cite[Remark 2.8]{gps1}.
	Additionally, we expect that we can compute the wrapped Fukaya category of plumbings of cotangent bundles of {\em nonorientable} manifolds, by employing the same techniques given in the paper.  
\end{rmk}

We would like to define the notion of plumbing sector. 
To do that, we use the following notation: for a fixed $n \in \mathbb{N}$, 
\begin{itemize}
	\item $\mathbb{R}^{2n}$ is a Weinstein manifold equipped with the standard Liouville one form
	\[\lambda_n:= \sum_i \tfrac{1}{2} \left(x_i d y_i - y_i d x_i\right),\]
	where $x_1, \dots, x_n$ (resp.\ $y_1, \dots, y_n$) coordinate the first (resp.\ last) $n$-factors of $\mathbb{R}^{2n}$.
	\item $\mathbb{D}^{2n}$ is a subset of $\mathbb{R}^{2n}$ such that 
	\[\mathbb{D}^{2n} := \{(x_1, \dots, x_n, y_1, \dots, y_n) \in \mathbb{R}^{2n} | \sum_i(x_i^2 + y_i^2) \leq 1\}.\]
	We note that $\left(\mathbb{D}^{2n}, \lambda_n\right)$ is a Weinstein domain. 
	\item Let $S^{n-1}$ be the unit sphere in $\mathbb{R}^n$, and let $\mathbf{0}_n$ denote the origin point $(0, \dots, 0)$ in $\mathbb{R}^n$. 
\end{itemize}

With the above notations, Definition \ref{dfn plumbing sector} defines the notion of plumbing sector, see also \cite[Section 6.2]{gps3}.
\begin{dfn}
	\label{dfn plumbing sector}
	A {\em plumbing sector of dimension $\mathit{2n}$} is a Weinstein sector 
	\[\Pi_n := \left(\mathbb{D}^{2n}, \lambda_n, (S^{n-1} \times \mathbf{0}_n) \sqcup (\mathbf{0}_n \times S^{n-1})\right). \]
	For convenience, let $\mathit{\Lambda_1}$ (resp.\ $\mathit{\Lambda_2}$) denote $S^{n-1} \times \mathbf{0}_n$ (resp.\ $\mathbf{0}_n \times S^{n-1}$).
\end{dfn}
In Section \ref{subsection construction of plumbing spaces}, the plumbing spaces will be obtained by gluing Weinstein sectors. 
The following two inclusion maps will be used for representing the gluing information.

\begin{gather}
	\label{eqn inclusion map alpha}
	\Phi: S^{n-1} \stackrel{\sim}{\to} \Lambda_1 \hookrightarrow \Pi_n,\quad p \mapsto \left(p, \mathbf{0}_n \right), \\
	\label{eqn inclusion map beta}
	\Psi: S^{n-1} \stackrel{\sim}{\to} \Lambda_2 \hookrightarrow \Pi_n,\quad p \mapsto \left(\mathbf{0}_n, p \right).
\end{gather}

It is easy to check that one can extend the above inclusion map $\Phi$ in \eqref{eqn inclusion map alpha} (resp.\ $\Psi$ in \eqref{eqn inclusion map beta}) to a contactomorphism between a small neighborhood of $S^{n-1}$ in $1$-jet space $T^*S^{n-1} \times \mathbb{R}$ and a neighborhood of $\Lambda_1$ (resp.\ $\Lambda_2$) in $\partial \Pi_n$. 
We note that the second mentioned space, i.e., a neighborhood of $\Lambda_1$ (resp.\ $\Lambda_2$) in $\partial \Pi_n$, is a codimension $1$ (contact-)submanifold of $\Pi_n$, whose contact one form is the restriction of the Liouville one form of $\Pi_n$, i.e., $\lambda_n$. Note also that by considering the pullbacks of \eqref{eqn inclusion map alpha} and \eqref{eqn inclusion map beta}, we have the induced inclusions (called {\em inclusions of Liouville sectors} in \cite{gps1})
\begin{gather}
	\label{eq:plumbing-inclusion-1}
	\Phi: T^*S^{n-1} \hookrightarrow \Pi_n ,
	\Psi: T^*S^{n-1} \hookrightarrow \Pi_n .
\end{gather}
Here, $\Pi_n$ is regarded as a Weinstein manifold with stops, instead of a Weinstein domain with stops. For the equivalence of these two notions, see Definition \ref{dfn Weinstein pair}.

Finally, it is clear that the skeleton $\text{Skel}(\Pi_n)$ of $\Pi_n$ is given by $(\mathbb{D}^n\times\mathbf{0}_n)\cup(\mathbf{0}_n\times\mathbb{D}^n)$ as shown in Figure \ref{fig:skel-plumb}.

\begin{figure}[ht]
	\centering	
	\begin{tikzpicture}	
		\draw[domain=0:1.5,smooth,variable=\x] plot ({\x^2},{\x});
		\draw[domain=0:1.5,smooth,variable=\x] plot ({\x^2},{-\x});
		\draw[domain=0:1.5,smooth,variable=\x] plot ({-\x^2},{\x});
		\draw[domain=0:1.5,smooth,variable=\x] plot ({-\x^2},{-\x});
		\draw[blue] (1.5^2,-1.5) to[bend left] (1.5^2,1.5);
		\draw[blue] (1.5^2,-1.5) to[bend right] (1.5^2,1.5);
		\draw[blue] (-1.5^2,-1.5) to[bend left] (-1.5^2,1.5);
		\draw[blue,dashed] (-1.5^2,-1.5) to[bend right] (-1.5^2,1.5);
		
		\node[blue] at (-4,0) (X) {$\Phi(S^{n-1})= \Lambda_1$};
		\node[blue] at (4,0) (Y) {$\Lambda_2=\Psi(S^{n-1})$};
	\end{tikzpicture}
	
	\caption{The skeleton of $\Pi_n$}
	\label{fig:skel-plumb}
\end{figure}
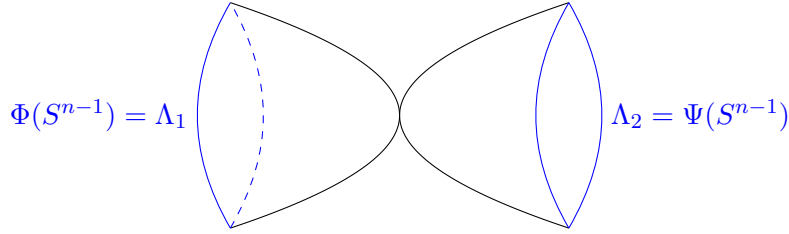

\begin{rmk}\label{rmk orientations of stops}
	In the rest of the paper, we consider $\Lambda_1$ and $\Lambda_2$ as oriented manifolds. 
	Their orientations are the induced orientations by $\Phi, \Psi$ from the standard orientation of $S^{n-1}$. 
	Or equivalently, $\Lambda_1 = \partial \left(\mathbb{D}^n \times \mathbf{0}_n \right)$ (resp.\ $\Lambda_2 = \partial \left(\mathbf{0}_n \times \mathbb{D}^n\right)$) is oriented by the standard orientation of $\mathbb{D}^n \times \mathbf{0}_n$ (resp.\ $\mathbf{0}_n \times \mathbb{D}^n$) when $\mathbb{D}^n \times \mathbf{0}_n$ (resp.\ $\mathbf{0}_n \times \mathbb{D}^n$) is oriented by the the volume form $d x_1 \wedge \dots \wedge d x_n$ (resp.\ $d y_1 \wedge \dots \wedge d y_n$).
		
	On the other hand, the symplectic form $d \lambda_n = \sum_i d x_i \wedge d y_i$ on $\Pi_n=(\mathbb{D}^{2n},\lambda_n,\Lambda_1\sqcup\Lambda_2)$ implies that $\mathbb{D}^{2n}$ admits an orientation induced from $\left(d \lambda_n\right)^{\wedge n}$.
	Equivalently, the orientation is same as that induced by the volume form
	\[d x_1 \wedge d y_1 \wedge d x_2 \wedge d y_2 \wedge \dots \wedge d x_n \wedge d y_n = (-1)^{\tfrac{1}{2}(n-1)n} d x_1 \wedge \dots \wedge d x_n \wedge d y_1 \wedge \dots \wedge d y_n .\]
	Hence, if $\tfrac{1}{2}(n-1)n$ is an even (resp.\ odd) integer, then two disks $\mathbb{D}^n \times \mathbf{0}_n$ and $\mathbf{0}_n \times \mathbb{D}^n$ intersect positively (resp.\ negatively) in $\mathbb{D}^{2n}$, or equivalently, in $\Pi_n$. This observation will be important in Section \ref{subsubsection gluing information}.
\end{rmk}

\subsection{Construction of plumbing spaces}
\label{subsection construction of plumbing spaces}
Now, we construct a plumbing space from a plumbing data. 
We first consider a collection of Weinstein sectors from a given plumbing data, then glue the collection. 

\subsubsection{Collection of Weinstein sectors}
\label{subsubsection collection of Weinstein sectors}
Let 
\[(Q, M: V(Q) \to \mathcal{O}_n, \s: E(Q) \to \{\pm 1\})\] 
be an arbitrary plumbing data. 
The collection of Weinstein sectors corresponding to $(Q, M, \s)$ consists of Weinstein sectors of two types. 
The first type is a cotangent bundle, and the second type is the plumbing sector $\Pi_n$. 
More precisely, the collection consists of
\begin{itemize}
	\item a cotangent bundle for each vertex $v \in V(Q)$, and 
	\item a plumbing sector $\Pi_n$ for each arrow $e \in E(Q)$.
\end{itemize}

For each vertex $v \in V(Q)$, we add a cotangent bundle as follows: $M(v)$ is an oriented manifold of dimension $n \geq 2$ (with or without boundary). 
For every edge $e = v \to w \in E(Q)$, we choose two {\em plumbing points} $p(e)$ in the interior of $M(v)$ and $q(e)$ in the interior of $M(w)$ so that, for all $v \in V(Q)$, the points of the subset of $M(v)$ given by
\begin{gather}
	\label{eqn choice of points}
	\{p(e) \in M(v) |  e= v \to \bullet \text{  for a  } e \in E(Q)\}
	\sqcup
	\{q(e) \in M(v) |  e= \bullet \to  v \text{  for a  } e \in E(Q)\}
\end{gather}
are all distinct. Then, one can choose small, disjoint, closed neighborhoods $U_{p(e)}, U_{q(e)} \subset M(v)$ of $p(e)$ and $q(e)$ such that $U_{p(e)}, U_{q(e)} \simeq \mathbb{D}^n$.
We note that the choice of plumbing points $p(e)$ and $q(e)$ does not affect the symplectomorphism type of resulting plumbing space, if $n \geq 2$. 
See Remark \ref{rmk choice of plumbing points} for more details. 

\begin{dfn}
	\label{dfn M_v}
	For a given plumbing data $(Q, M,\s)$, we set {\em $\mathit{M_v}$} as 
	\[M_v :=\text{  the closure of  } \left(M(v) \setminus \left(\bigcup_{e= v \to \bullet} U_{p(e)} \cup \bigcup_{e =\bullet \to v} U_{q(e)}\right)\right).\]
\end{dfn}

Let us assume that $T^*M_v$ is equipped with the standard Liouville one form, $\sum_{i=1}^n -p_i dq_i$ where $(q_1, \dots, q_n)$ (resp.\ $(p_1, \dots, p_n)$) coordinates the base (resp.\ fiber) direction of the cotangent bundle $T^*M$. 
Then, $T^*M_v$ is a Weinstein sector (or equivalently, Weinstein pair) as described in Section \ref{subsection Weinstein pair}. 
We note that the symplectomorphism classes of resulting plumbing spaces do not depend on the specific choice of points $p(e)$ and $q(e)$ in \eqref{eqn choice of points}, since $n \geq 2$. 

Now, for a given plumbing data, we consider the following collection of Weinstein sectors: 
\[\left\{T^*M_v, \Pi_e = \Pi_n | v \in V(Q), e \in E(Q)\right\}.\]

\subsubsection{Gluing information}
\label{subsubsection gluing information}
In Section \ref{subsubsection collection of Weinstein sectors}, we obtained a collection of Weinstein sectors from an arbitrary plumbing data $(Q, M, \s)$. 
To finish our construction of plumbing spaces, we need to glue the Weinstein sectors in the collection. 

We note that the relevant stops of Weinstein sectors in the given collections are 
\[\partial U_{p(e)}, \partial U_{q(e)} \text{  (in  } T^*M_v), \Lambda_1 = S^{n-1} \times \mathbf{0}_n, \text{  and  } \Lambda_2=\mathbf{0}_n \times S^{n-1} \text{  (in  } \Pi_e = \Pi_n).\]
Thus, all these Legendrian stops are diffeomorphic to $S^{n-1}$. 
Moreover, the corresponding Weinstein hypersurfaces are $T^*S^{n-1}$ since their skeletons are $S^{n-1}$. 

In order to glue, it is enough to give the following information:
\begin{enumerate}
	\item[(A)] We need to {\em pair} Legendrian stops so that two Legendrians in a pair will be identified/glued to each other, and
	\item[(B)] We need to choose a diffeomorphism to $S^{n-1}$ for each of Legendrian stops. 
\end{enumerate}
We note that the diffeomorphisms in (B) can induce a diffeomorphism between two Legendrian stops that are paired in (A), and the pullback of the induced diffeomorphism becomes a Weinstein isomorphism between two Weinstein hypersurfaces. 
As explained in Section \ref{subsection gluing of Weinstein pairs/sectors}, the induced Weinstein isomorphism allows us to glue along the paired/identified Weinstein hypersurfaces of the Weinstein sectors. 

{\em Information for (A)} :
Let an arrow $e \in E(Q)$ start (resp.\ end) at a vertex $v$ (resp.\ $w$), i.e.,
\[e = v \to w.\] 
Then, there are four Legendrian spheres related to $e$, $\Lambda_1, \Lambda_2$ in $\Pi_e$ (which is $\Pi_n$ corresponding to $e$), $\partial U_{p(e)}$ in $T^*M_v$, and $\partial U_{q(e)}$ in $T^*M_w$. 
The stop $\Lambda_1$ (resp.\ $\Lambda_2$) in the corresponding $\Pi_e$ is paired with $\partial U_{p(e)} \subset M_v$ (resp.\ $\partial U_{q(e)} \subset M_w$). 

{\em Information for (B)}: 
Again, let an arrow $e$ be the arrow starting at $v$ and ending at $w$. 
As mentioned above, there exist two pairs of Legendrian spheres related to $e$, $\left(\Lambda_1, \partial U_{p(e)}\right)$ and $\left(\Lambda_2, \partial U_{q(e)}\right)$. 

For the first two spheres $\Lambda_1$ and $\Lambda_2$, we choose diffeomorphisms $\Phi, \Psi$ (onto their images) defined in \eqref{eqn inclusion map alpha} and \eqref{eqn inclusion map beta}.
To indicate that those maps are related to the plumbing sector corresponding to a specific arrow $e$, we put $e$ as a subscript, and we have 
\[\Phi_e : S^{n-1}\stackrel{\sim}{\to}\Lambda_1 ,\quad \Psi_e : S^{n-1}\stackrel{\sim}{\to} \Lambda_2 .\]
We note that $\Phi_e$ and $\Psi_e$ are orientation preserving diffeomorphisms. 

For the third sphere $\partial U_{p(e)}$, we remark that $U_{p(e)}$ is an oriented disk since $U_{p(e)}$ is a submanifold of an oriented manifold $M(v)$. 
We assume that $\partial U_{p(e)}$ is equipped with the natural boundary orientation inherited from $U_{p(e)}$. 
Then, we choose any orientation preserving diffeomorphism between $S^{n-1}$ and $\partial U_{p(e)}$. 
Let $F_e: S^{n-1} \stackrel{\sim}{\to}\partial U_{p(e)}$ be the chosen diffeomorphism. 
Then, we can identify $\partial U_{p(e)}$ and $\Lambda_1$ via the following:
\begin{equation}\label{eq:weinstein-isomorphism-1}
	\partial U_{p(e)}\xrightarrow{F_e^{-1}} S^{n-1}\xrightarrow{\Phi_e}\Lambda_1.
\end{equation}

\begin{rmk}
	\label{rmk identifying stops 1}
	We glue/identify $\Lambda_1$ in $\Pi_n$ and $\partial U_{p(e)}$ via $\Phi_e$ and $F_e$. 
	Instead of this, one can identify $U_{p(e)}$ and $\mathbb{D}^n \times \mathbf{0}_n \subset \Pi_n$, as `oriented' manifolds. 
	Then, this identification induces the boundary identification between $\partial U_{p(e)}$ and $\partial \left(\mathbb{D}^n \times \mathbf{0}_n\right) = \Lambda_1$. 
\end{rmk}

For the last sphere $\partial U_{q(e)}$, we also remark that $U_{q(e)}$ is an oriented disk and $\partial U_{q(e)}$ admits a natural boundary orientation. 
Then, we choose $G_e: S^{n-1}\stackrel{\sim}{\to} \partial U_{q(e)} $ by the following rule:
\begin{itemize}
	\item If $\tfrac{1}{2}(n-1)n$ is an even integer, then 
	\begin{gather*}
		G_e = \begin{cases}
			\text{  any orientation preserving diffeomorphism, if  } \s(e)=1 ,\\
			\text{  any orientation reversing diffeomorphism, if  } \s(e)=-1 .
		\end{cases}
	\end{gather*}
	\item If $\tfrac{1}{2}(n-1)n$ is an odd integer, then 
	\begin{gather*}
		G_e = \begin{cases}
			\text{  any orientation reversing diffeomorphism, if  } \s(e)=1 ,\\
			\text{  any orientation preserving diffeomorphism, if  } \s(e)=-1 .
		\end{cases}
	\end{gather*}
\end{itemize}
More concisely,
\begin{gather*}
	G_e = \begin{cases}
		\text{  any orientation preserving diffeomorphism, if  } (-1)^{*_e}=1 ,\\
		\text{  any orientation reversing diffeomorphism, if  } (-1)^{*_e}=-1 ,
	\end{cases}
\end{gather*}
where $(-1)^{*_e}:=(-1)^{n(n-1)/2}\s(e)$.
Then, as in the previous case, we can identify $\partial U_{q(e)}$ and $\Lambda_2$ via the following:
\begin{equation}\label{eq:weinstein-isomorphism-2}
	\partial U_{q(e)}\xrightarrow{G_e^{-1}} S^{n-1}\xrightarrow{\Psi_e}\Lambda_2.
\end{equation}

\begin{rmk}
	\label{rmk identifying stops 2}
	As in Remark \ref{rmk identifying stops 1}, one can identify $U_{q(e)}$ and $\mathbf{0}_n \times \mathbb{D}^n$ instead of identifying their boundaries. 
	However, we should take care of orientations according to the value of $\s(e)$, as opposed to Remark \ref{rmk identifying stops 1},. 
\end{rmk} 

Finally, note that by considering the pullbacks of $F_e$ and $G_e$, we have the induced inclusions (called inclusions of Liouville sectors in \cite{gps1})
\begin{gather}
	\label{eq:F-inclusion} F_e : T^*S^{n-1}\hookrightarrow T^*M_v,\qquad G_e : T^*S^{n-1}\hookrightarrow T^*M_w .
\end{gather}

Now, we are done with the specification of the information (A) and (B). 
Thus, one can glue the Weinstein sectors in the collection corresponding to a given plumbing data.
By gluing the Weinstein sectors in a collection, we can construct a plumbing space.
\begin{dfn}
	\label{dfn plumbing space}
	\mbox{}
	\begin{enumerate}
		\item Let $(Q, M, \s)$ be a plumbing data.
		Let {\em $\mathit{P(Q, M, \s)}$} denote the Weinstein manifold constructed from $(Q, M, \s)$ by the above procedure, i.e.,
		\[P(Q,M,\s) = \left(\bigcup_{v \in V(Q)} T^*M_v \cup \bigcup_{e \in E(Q)} \Pi_n \right)/\sim,\]
		where the gluing occurs via the Weinstein isomorphism $\Phi_e\circ F_e^{-1}$ in \eqref{eq:weinstein-isomorphism-1} and $\Psi_e\circ G_e^{-1}$ in \eqref{eq:weinstein-isomorphism-2} for all $e\in E(Q)$.
		\item A Weinstein manifold $P$ is a {\em plumbing space} if there is a plumbing data $(Q, M, \s)$ such that 
		$P = P(Q, M, \s)$.
		We sometimes say that $P$ is the {\em plumbing of $\{T^*M(v)\vb v\in V(Q)\}$ along the quiver $Q$} (where $Q$ is equipped with $\s$ function).
		\item A plumbing space $P=P(Q, M, \s)$ is a {\em plumbing space with negative intersection} if there exists an arrow $e$ such that $\s(e) = -1$. 
	\end{enumerate}
\end{dfn}
\begin{rmk}
	\label{rmk choice of plumbing points} 
	We note that in the construction, we fixed specific plumbing points $p(e)$ and $q(e)$ for each $e \in E(Q)$, in \eqref{eqn choice of points}. 
	However, the specific choice does not affect on the symplectomorphism class of the resulting plumbing space if $n \geq 2$, since it does not change the collection of (symplectomorphism classes of) Weinstein sectors and the gluing information.
	The equivalence of symplectomorphism classes of Weinstein sectors can be seen by choosing isotopies of neighborhoods of plumbing points, whose existence is guaranteed by connectedness of $M(v)$ and the dimension $n \geq 2$ of $M(v)$.  
\end{rmk}

Before ending Section \ref{subsection construction of plumbing spaces}, let us explain the choice of $G_e$ and the relation between our and the conventional plumbing procedures.  
In the constructed plumbing space $P=P(Q, M, \s)$, we note that $M(v)$ can be seen as an embedded Lagrangian submanifold of $P=P(Q,M,\s)$ as follows: We have 
\[M(v) = M_v \cup \bigcup_{e = v \to \bullet} U_{p(e)} \cup \bigcup_{e = \bullet \to v} U_{q(e)}.\]
We also note that $M_v$, the zero section of $T^*M_v$, is embedded in the plumbing space. 
Moreover, $U_{p(e)}$ and $U_{q(e)}$ could be seen as $\mathbb{D}^n \times \mathbf{0}_n$ and $\mathbf{0}_n \times \mathbb{D}^n$ embedded in $\Pi_n$ corresponding to $e$, as mentioned in Remark \ref{rmk identifying stops 1} and \ref{rmk identifying stops 2}, respectively. 
Thus, $M(v)$ is an embedded Lagrangian manifold of the plumbing space. 
Moreover, given $e=v\to w$, $M(v)$ and $M(w)$ intersect at the identified plumbing point $p(e) = q(e)$ as Lagrangian submanifolds. 
By choosing $G_e$ as described above, we can achieve that $M(v)$ and $M(w)$ intersect positively (resp.\ negatively) if $\s(e)$ is $1$ (resp.\ $-1$) because of Remark \ref{rmk orientations of stops}.

\begin{rmk}
	\label{rmk equivalence with the conventional plumbing}
	Remarks \ref{rmk identifying stops 1}--\ref{rmk identifying stops 2} and the orientation argument right above imply that the conventional plumbing construction is equivalent to that given in Section \ref{subsection construction of plumbing spaces}.
\end{rmk}

\subsection{Equivalent plumbing spaces}
\label{subsection a property of plumbing spaces}
Let $(Q, M, \s)$ be a plumbing data. 
One can obtain another plumbing data by changing directions and signs of some arrows of $Q$.
Let $(Q', M', \s')$ be the new plumbing data obtained from $(Q, M, \s)$. 
Then, one can ask how different $P(Q, M, \s)$ and $P(Q', M', \s')$ are. 
In this subsection, we answer the question. 

At first glance, it seems that the changes of directions of arrows in a plumbing quiver interchange the role of $\Lambda_1$ and $\Lambda_2$ in the plumbing procedure.
Motivated from this, we will consider a Hamiltonian diffeomorphism on $\Pi_n$ which interchanges $\Lambda_1$ and $\Lambda_2$. 

To be more precise, let us recall that $\Pi_n$ is a unit disk of $\mathbb{R}^{2n}$.
In Section \ref{subsection notation}, $(x_1, \dots, x_n)$ (resp.\ $(y_1, \dots, y_n)$) coordinates the first (resp.\ last) $n$-factors of $\mathbb{R}^{2n}$. 
With the coordinates, the Liouville one form is $\lambda_n := \sum_i \tfrac{1}{2} \left(x_i dy_i - y_i dx_i\right)$. 

Let $\varphi: \mathbb{R}^{2n} \to \mathbb{R}^{2n}$ be the map defined as 
\[\varphi(x_1, \dots, x_n, y_1, \dots, y_n) = (- y_1, \dots, - y_n, x_1, \dots, x_n).\]
Since it is easy to check that $\varphi$ is a Hamiltonian diffeomorphism on $\mathbb{R}^{2n}$ such that $\varphi^* \lambda_n = \lambda_n$, even if one replaces every $\Pi_n$ in Section \ref{subsection construction of plumbing spaces} with $\varphi(\Pi_n)$, one should have the same plumbing space. 

Moreover, the following are also easy to observe:
\begin{enumerate}
	\item[(i)] As mentioned above, $\varphi(\mathbb{D}^{2n}, \lambda_n) = (\mathbb{D}^{2n}, \lambda_n)$.
	\item[(ii)] $\varphi(\Lambda_1) = \Lambda_2$.
	\item[(iii)] $\varphi|_{\Lambda_1}$ preserves the orientation. 
	\item[(iv)] $\varphi(\Lambda_2) = \Lambda_1$.
	\item[(v)] if $n$ is an odd integer, $\varphi|_{\Lambda_2}$ reverses the orientation, i.e., as oriented manifolds $\varphi(\Lambda_2) = -\Lambda_1$.
	\item[(vi)] if $n$ is an even integer, $\varphi|_{\Lambda_2}$ preserves the orientation, i.e., as oriented manifolds $\varphi(\Lambda_2) = \Lambda_1$.
\end{enumerate}
Then, the above arguments and the observations (i)--(vi) prove Propositions \ref{prp property 1} and \ref{prp property 2}.

\begin{prp}
	\label{prp property 1}
	Let $n \in \mathbb{N}$ be an odd integer.
	Let $(Q, M, \s)$ and $(Q', M', \s')$ be two sets of plumbing data satisfying the following:  
	\begin{itemize}
		\item Two quivers $Q$ and $Q'$ have the same base graph. In other words, $V(Q) = V(Q')$ and $E(Q) = E(Q')$. (This is by abuse of notation. Here we see $E(Q)$ and $E(Q')$ as sets of `edges', not `arrows'.)
		\item $M : V(Q) \to \mathcal{O}_n$ and $M': V(Q') \to \mathcal{O}_n$ are the same maps. 
		\item For any arrow $e \in E(Q) = E(Q')$, if $e$ has the same (resp.\ opposite) directions on $Q$ and $Q'$, then $\s(e) = \s'(e)$ (resp.\ $\s(e)= - \s'(e)$).
	\end{itemize}
	Then $P(Q, M, \s) = P(Q', M', \s')$.
\end{prp}

One can observe from Proposition \ref{prp property 1} that if $n$ is odd and an arrow $e \in E(Q)$ is a loop, then $\s(e)$ does not affect the Hamiltonian isotopy class of $P(Q,M,\s)$.  

\begin{prp}
	\label{prp property 2}
	Let $n \in \mathbb{N}$ be an even integer.
	Let $(Q, M, \s)$ and $(Q', M', \s')$ be two sets of plumbing data satisfying the following:  
	\begin{itemize}
		\item Two quivers $Q$ and $Q'$ have the same base graph. In other words, $V(Q) = V(Q')$ and $E(Q) = E(Q')$.
		\item $M : V(Q) \to \mathcal{O}_n$ and $M': V(Q') \to \mathcal{O}_n$ are the same maps. 
		\item For every arrow $e \in E(Q) = E(Q')$, $\s(e) = \s'(e)$.
	\end{itemize}
	Then $P(Q, M, \s) = P(Q', M', \s')$.
\end{prp}

Proposition \ref{prp property 2} means that if $n$ is an even integer, the directions of arrows of a plumbing quiver does not affect on the resulting plumbing space. 
We note that as one can see in Proposition \ref{prp property 1}, if $n$ is an odd integer, then the directions of arrows of a plumbing quiver affects on the resulting plumbing space. 
Thus, one can replace a plumbing quiver in a plumbing data with its base graph when $n$ is even. 
We note that since \cite{Etgu-Lekili17, Etgu-Lekili19} study plumbing spaces of dimension $4$, i.e., the case of $n=2$, \cite{Etgu-Lekili17, Etgu-Lekili19} can use plumbing {\em graphs} rather than plumbing quivers. 

We conclude this subsection with another comparison of plumbing data, where we maintain the directions of the arrows and only change the signs. First, observe that the plumbing space $P(Q,M,\s)$ does not depend on the specific orientation chosen for $M(v)$ at each $v\in V(Q)$. Changing the orientation of $M(v)$ at a given vertex $v$ results in the reversal of signs for all arrows in $Q$ starting or ending at $v$. Therefore, we obtain the following proposition:

\begin{prp}\label{prp property 3}
	Let $(Q, M, \s)$ and $(Q', M', \s')$ be two sets of plumbing data satisfying the following:  
	\begin{itemize}
		\item $Q$ and $Q'$ are the same quivers, with a chosen subset of vertices $I\subset V(Q)=V(Q')$.
		\item $M : V(Q) \to \mathcal{O}_n$ and $M': V(Q') \to \mathcal{O}_n$ are the same maps. 
		\item For every arrow $e \in E(Q) = E(Q')$,
		\[\s'(e)=\begin{cases}
			\s(e) & \text{if both the source and the target of $e$ is in $I$,}\\
			\s(e) & \text{if neither the source nor the target of $e$ is in $I$,}\\
			-\s(e) & \text{otherwise.}
		\end{cases}\]
	\end{itemize}
	Then $P(Q, M, \s) = P(Q', M', \s')$.
\end{prp}

An immediate corollary of Proposition \ref{prp property 3} is as follows:

\begin{cor}\label{cor property 4}
	A plumbing space $P(Q,M,\s)$ does not depend on the sign of its arrow $e\in E(Q)$ if removing $e$ from $Q$ results in a disconnected quiver. In particular, if $Q$ is a tree, then $P(Q,M,\s)$ is independent of the function $\s$ altogether.
\end{cor}

\begin{rmk}
	From Proposition \ref{prp property 3}, it can be observed that for a general quiver $Q$, there are effectively $H^1(Q;\Z/2)$ many choices of the $\s$ function when defining the plumbing space $P(Q,M,\s)$.
\end{rmk}

\subsection{Homotopy colimit diagrams for the wrapped Fukaya category of plumbing spaces}
\label{subsection homotopy colimit formula for a plumbing space}

The final goal of the current paper is to compute the wrapped Fukaya categories of plumbing spaces.
In order to achieve the goal, we will apply Theorem \ref{thm:gps}. 
By applying Theorem \ref{thm:gps}, one obtains a homotopy colimit formula computing the wrapped Fukaya category of a Weinstein manifold $W$ through three steps. 
The first step is to find a Weinstein sectorial covering of $W$, then one has a homotopy colimit diagram from the sectorial covering. 
We note that the homotopy colimit diagram consists of wrapped Fukaya categories of each of Weinstein sectors.
Thus, the second step is to compute each of wrapped Fukaya categories. 
The last step is to compute the homotopy colimit, which can be achieved via \cite{hocolim, Karabas-Lee24} (or Theorem \ref{thm:hocolim-functor-dg}).

Since a plumbing space $P(Q, M, \s)$ is obtained by gluing Weinstein sectors in our construction, there is a natural Weinstein sectorial covering of $P(Q, M, \s)$ from the construction. 
In Section \ref{subsubsection hocolim diagram A}, we will give a homotopy colimit diagram for the wrapped Fukaya category of $P(Q, M, \s)$, which is induced from the natural Weinstein sectorial covering. In Section \ref{subsubsection hocolim diagram B}, we will give another homotopy colimit diagram for the wrapped Fukaya category of $P(Q, M, \s)$ which uses a different Weinstein sectorial covering.

\subsubsection{Homotopy colimit diagram from the natural Weinstein sectorial covering}\label{subsubsection hocolim diagram A}

Let $(Q, M, \s)$ be a plumbing data, and let $P$ denote the plumbing space $P = P(Q, M, \s)$.
We recall that for each vertex $v \in V(Q)$, we add a Weinstein sector $T^*M_v$ in the natural Weinstein sectorial covering, and for each arrow $e \in E(Q)$, we add a plumbing sector $\Pi_n$ in the covering.
The plumbing space $P$ is the gluing of Weinstein sectors in the covering of $P$. 

The natural Weinstein sectorial covering of $P$ induces a homotopy colimit diagram computing the wrapped Fukaya category of $P$, denoted by $\cW(P)$. 
Before discussing the induced homotopy colimit for a general plumbing space $P = P(Q, M, \s)$, let us consider the simplest case such that the quiver $Q$ consists of two vertices and only one arrow. 
We denote by $e$ the unique arrow, and by $v$ and $w$ the starting and ending vertex of $e$, i.e., 
$ Q = v \stackrel{e}{\to} w.$

The natural Weinstein sectorial covering consists of three Weinstein sectors, $T^*M_v$ corresponding to the vertex $v$, $T^*M_w$ corresponding to the vertex $w$, and $\Pi_n$ corresponding to the arrow $e$. 
Then, $T^*M_v$ (resp.\ $T^*M_w$) is connected/glued to $\Pi_n$ along $\partial U_{p(e)}$ and $\Lambda_1$ (resp.\ $\partial U_{q(e)}$ and $\Lambda_2$).
We also note that the gluing information is given by the maps $\Phi_e$, $\Psi_e$, $F_e$, and $G_e$ defined in Section \ref{subsubsection gluing information}.

By applying Theorem \ref{thm:gps}, one has the following homotopy colimit diagram up to pretriangulated equivalence:
\begin{equation}
	\label{eqn homotopy colimit example}
	\cW(P) \simeq \textup{hocolim}\left(
	\begin{tikzcd}[column sep=0em]
		\cW(T^*M_v) & & \cW(\Pi_n) & & \cW(T^*M_w)\\
		& \cW(T^*S^{n-1}) \ar[lu,"F_e"]\ar[ru,"\Phi_e"'] & &  \cW(T^*S^{n-1})  \ar[lu,"\Psi_e"]\ar[ru,"G_e"']
	\end{tikzcd}\right),
\end{equation}
where $\Phi_e$, $\Psi_e$, $F_e$, and $G_e$ in Equation \eqref{eqn homotopy colimit example} are induced functors (by \cite{gps1}) from the inclusions with the same names, i.e., $\Phi_e$, $\Psi_e$, $F_e$, and $G_e$ given in \eqref{eq:plumbing-inclusion-1} and \eqref{eq:F-inclusion}.

The homotopy colimit in Equation \eqref{eqn homotopy colimit example} is for the special case such that the quiver is $Q = v \stackrel{e}{\to} w$. 
By generalizing the above argument, for any general plumbing space $P=P(Q, M, \s)$, we have the pretriangulated equivalence (compare with \cite[Corollary 6.3]{gps3}):

\begin{equation}
	\label{eqn homotopy colimit formula}
	\cW(P) \simeq \textup{hocolim}\left(
	\begin{tikzcd}[column sep=0em]
		\coprod_{v \in V(Q)} \cW(T^*M_v) & & \coprod_{e \in E(Q)} \cW(\Pi_n)\\
		& \hspace*{-5em} \coprod_{e \in E(Q)} \left(\cW(T^*S^{n-1}) \amalg \cW(T^*S^{n-1})\right) \hspace*{-5em} \ar[lu,"\coprod_{e \in E(Q)} \left(F_e \amalg G_e\right)"]\ar[ru,"\coprod_{e \in E(Q)} \left( \Phi_e \amalg \Psi_e \right)"']
	\end{tikzcd}\right).
\end{equation}

\subsubsection{Another homotopy colimit diagram using plumbings of cotangent bundles of disks}\label{subsubsection hocolim diagram B}

Let $(Q,M,\s)$ be an arbitrary plumbing data. Then, for each $v\in V(Q)$, $M(v)$ is an $n$-dimensional connected oriented manifold (with or without a boundary) for a fixed $n\geq 2$. We can choose an arbitrary point in the interior of $M(v)$ and consider its sufficiently small closed neighborhood $U_v$, which is homeomorphic to the closed disk $\mathbb{D}^n$. We define
\[M(v)^*:=\text{ the closure of }\left(M(v)\setminus U_v\right) ,\]
which implies
$M(v)=U_v\cup M(v)^*$.
We orient $U_v$ with the induced orientation from $M(v)$, and $\partial U_v$ with the boundary orientation (as a boundary of $U_v$). Since $\partial U_v$ is also (a part of) the boundary of $M(v)^*$, we have the maps
\begin{gather*}
	\mu_v\colon S^{n-1}\xrightarrow{\sim}\partial U_v\hookrightarrow U_v , \qquad
	\eta_v\colon S^{n-1}\xrightarrow{\sim}\partial U_v\hookrightarrow M(v)^*  , 
\end{gather*}
where the diffeomorphisms ``$\xrightarrow{\sim}$" are chosen to be orientation reversing when $S^{n-1}$ is given the boundary orientation coming from the standard orientation of $\mathbb{D}^n$.
They induce the inclusions
\begin{gather}
	\label{eq:mu-temp} \mu_v\colon T^*S^{n-1}\hookrightarrow T^*U_v , \qquad \eta_v\colon T^*S^{n-1}\hookrightarrow T^*M(v)^* .
\end{gather}
As a result, we get the gluing
\[T^*M(v)=\left(T^*U_v\cup T^*M(v)^*\right)/\sim\]
via the Weinstein isomorphism $\eta_v\circ\mu_v^{-1}\colon \mu_v(T^*S^{n-1})\to\eta_v(T^*S^{n-1})$, where $\mu_v,\eta_v$ are as in \eqref{eq:mu-temp}. We want to extend this gluing to give a gluing for $P(Q,M,\s)$.

We note that to construct $P(Q,M,\s)$, we should choose plumbing points on $M(v)$ for $v \in V(Q)$. 
As mentioned in Remark \ref{rmk choice of plumbing points}, we can assume that all plumbing points on $M(v)$ are contained in $U_v$. 
Under the assumption, $P(Q,M,\s)$ could be obtained by gluing $\bigcup_{v \in V(Q)} T^*M(v)^*$ and $P(Q,M',\s)$ that is defined below.

Let $(Q,M',\s)$ be another plumbing data with the same quiver $Q$ and the same map $\s$
but $M'$ is different from $M$. We choose $M'(v)=\mathbb{D}^n$ for each $v\in V(Q)$. The map $\mu_v$ in \eqref{eq:mu-temp} can be extended to the inclusion
\begin{equation}\label{eq:mu}
	\mu_v\colon T^*S^{n-1}\hookrightarrow T^*U_v \xrightarrow{\sim} T^*\mathbb{D}^n=T^*M'(v)\hookrightarrow P(Q,M',\s) ,
\end{equation}
where the symplectomorphism ``$\xrightarrow{\sim}$" comes from an orientation preserving identification of $U_v$ and $\mathbb{D}^n$.
Then, it is clear that
\[P(Q,M,\s) = \left(P(Q,M',\s)\cup\bigcup_{v \in V(Q)} T^*M(v)^* \right)/\sim,\]
where the gluing in the right-hand side occurs via the Weinstein isomorphism $\eta_v\circ\mu_v^{-1}$ for all $v\in V(Q)$, and $\mu_v,\nu_v$ are as in  \eqref{eq:mu}. Thus, by Theorem \ref{thm:gps}, we have the pretriangulated equivalence
\begin{equation}
	\label{eqn homotopy colimit formula alternative}
	\cW(P(Q,M,\s)) \simeq \textup{hocolim}\left(
	\begin{tikzcd}[column sep=-3.5em]
		\cW(P(Q,M',\s)) & & \coprod_{v \in V(Q)} \cW(T^*M(v)^*)\\
		& \coprod_{v \in V(Q)} \cW(T^*S^{n-1}) \ar[lu,"\coprod_{v \in V(Q)} \mu_v"]\ar[ru,"\coprod_{v \in V(Q)} \eta_v"']
	\end{tikzcd}\right) ,
\end{equation}
where $\mu_v$ and $\eta_v$ in Equation \eqref{eqn homotopy colimit formula alternative} are induced functors (by \cite{gps1}) from the inclusions with the same names, i.e., $\mu_v$ and $\eta_v$ given in \eqref{eq:mu}.

\begin{rmk}\label{rmk:puncture}
	$M_v$ and $M(v)^*$ are obtained by removing small disks from $M(v)$. However, for the sake of convenience, we will refer to these removed disks as {\em punctures}. It's worth noting that the usage of the term ``puncture'' in this context differs from its conventional meaning.
\end{rmk}

\subsubsection{Our strategy for computing the wrapped Fukaya category of plumbings}\label{subsubsection strategy}

Having the formulas \eqref{eqn homotopy colimit formula} and \eqref{eqn homotopy colimit formula alternative}, our strategy for computing the wrapped Fukaya category of plumbings is as follows:
\begin{enumerate}
	\item We will compute the wrapped Fukaya category $\cW(\Pi_n)$ of the plumbing sector $\Pi_n$ in Section \ref{sec:plumbing-sectors}, and the wrapped Fukaya category of cotangent bundles of punctured spheres in Section \ref{subsection the cotangent bundles of spheres with punctures}.
		
	\item Then, using Equation \eqref{eqn homotopy colimit formula} and the previous step, we will compute the wrapped Fukaya category of plumbings of cotangent bundles of {\em disks} along any quiver in Section \ref{subsection plumbings disks}.
	
	\item Finally, using Equation \eqref{eqn homotopy colimit formula alternative} and the previous step, we will give a formula for the wrapped Fukaya category of plumbings of cotangent bundles of any collection of manifolds $M(v)$ along any quiver in Section \ref{subsection plumbings disks} and \ref{subsection plumbing grading}, provided that we know the wrapped Fukaya category $\cW(T^*M(v)^*)$ of cotangent bundles of $M(v)^*=M(v)\setminus\text{an open disk}$.
\end{enumerate}

Note that by Theorem \ref{thm:cotangent-generation} and Remark \ref{rmk:loop-space}, $\cW(T^*M(v)^*)$ is equivalently given by (up to pretriangulated equivalence) $C_{-*}(\Omega_pM(v)^*)$, i.e., the chains on the based loop space of $M(v)^*$, which only depends on the homotopy type of $M(v)^*$. In short, we will be able to formulate the wrapped Fukaya category of any plumbing space in terms of the topological data $C_{-*}(\Omega_pM(v)^*)$ at each vertex of $Q$. For an application, we will explicitly explain $C_{-*}(\Omega_pM(v)^*)$ (or equivalently, $\cW(T^*M(v)^*)$) in the following cases:
\begin{itemize}
	\item When $M(v)$ is an $n$-sphere, $\cW(T^*M(v)^*)$ is trivial as $M(v)^*$ is a disk.
	\item When $M(v)$ is an oriented closed surface, $\cW(T^*M(v)^*)$ will be explained in Section \ref{subsection the cotangent bundles of oriented surfaces with punctures}.
\end{itemize}
Using this information, we will explicitly compute the wrapped Fukaya category of
\begin{itemize}
	\item plumbings of $T^*S^n$'s along any quiver, with or without negative intersections, in Section \ref{subsection plumbings in (a)}, and
	\item plumbings of cotangent bundles of closed, oriented surfaces along any quiver, with or without negative intersections, in Section \ref{subsection plumbings in (b)}.
\end{itemize}

\section{Wrapped Fukaya category of plumbing sectors}
\label{sec:plumbing-sectors}
As mentioned in Section \ref{subsubsection strategy}, in order to compute the wrapped Fukaya category of plumbings, we first need to compute the wrapped Fukaya category $\cW(\Pi_n)$ of the plumbing sector $\Pi_n$. 
The main goal of the present section is to compute $\cW(\Pi_n)$.
Moreover, we also compute the functors $\Phi,\Psi\colon \cW(T^*S^{n-1})\to\cW(\Pi_n)$ induced by the inclusions of the boundaries of the plumbing sector as in \eqref{eq:plumbing-inclusion-1}.

In Section \ref{sec:statement-plumbing-sector}, we set our notation, recall some basic facts about the wrapped Fukaya categories of the cotangent bundles of disks, and state Theorem \ref{thm:wfuk-plumbing-3} and \ref{thm:wfuk-plumbing-2}. 
In Section \ref{sec:proof-plumbing-sector}, we present some technical tools and prove Theorem \ref{thm:wfuk-plumbing-3} and \ref{thm:wfuk-plumbing-2}.

\subsection{The plumbing sector $\Pi_n$ and its wrapped Fukaya category $\cW(\Pi_n)$}\label{sec:statement-plumbing-sector}

We will mostly work with codisk bundles instead of the cotangent bundles, where the former is a Weinstein domain of the latter, hence their wrapped Fukaya categories are the same. We will use Weinstein domains and their completions, i.e.,  Weinstein manifolds interchangeably.
Let $D^n$ (resp.\ $S^n$) be $n$-dimensional (closed) disk (resp.\ sphere) with radius 1, unless stated otherwise. The codisk bundle $D^*D^n$ can be modeled as
\[D^*D^n=\{(x_1,\ldots,x_n,y_1,\ldots,y_n) \vb x_1^2+\ldots+x_n^2\leq 1\text{ and }y_1^2+\ldots+y_n^2\leq 1\}\]
and the cosphere bundle $S^*D^n$ as
\[S^*D^n=\{(x_1,\ldots,x_n,y_1,\ldots,y_n) \vb x_1^2+\ldots+x_n^2\leq 1\text{ and }y_1^2+\ldots+y_n^2 = 1\} .\]
Note that at the boundary points, covectors are defined in the ambient space $\R^n\supset D^n$.

We start with the well-known computations of the wrapped Fukaya categories of $D^*D^n$ for the purpose of computing the wrapped Fukaya category of the plumbing sector later. Note that the grading structure and the background class needed when defining the wrapped Fukaya categories of $D^*D^n$ are uniquely determined by Remark \ref{rmk:grading-pin}.

\begin{prp}\label{prp:msh-a2}\mbox{}
	\begin{enumerate}
		\item We have the following pretriangulated equivalences for the wrapped Fukaya categories:
		\begin{itemize}
			\item $\cW(D^*D^n)\simeq\cA_1$ for any $n\geq 1$, where $\cA_1$ is the semifree dg category defined as follows:
			\begin{enumerate}[label = (\roman*)]
				\item {\em Objects:} $K$ (corresponding to a cotangent fiber of $D^*D^n$).
				\item {\em Generating morphisms:} No generating morphisms. (It means that the morphisms are generated by only identity morphisms.)
			\end{enumerate}
			
			\item $\cW(D^*D^1,\{(0,1)\})\simeq\cA_2$, where $\cA_2$ is the semifree dg category defined as follows:
			\begin{enumerate}[label = (\roman*)]
				\item {\em Objects:} $K_0,K_1$ (corresponding to two cotangent fibers of $D^*D^1$ separated by the stop $\{(0,1)\}\subset S^*D^1$).
				\item {\em Generating morphisms:} $f\in\hom^*(K_0,K_1)$.
				\item {\em Degrees:} $|f|=0$.
				\item {\em Differentials:} $df=0$.
			\end{enumerate}
		\end{itemize}
		
		\item The skeleton of $(D^*D^1,\{(0,1)\})$ is given by the $A_2$-arboreal singularity, which is the conic subset of $D^*D^1$ defined as
		\[A_2:=D^1 \cup \Cone(\{(0,1)\in S^*D^1\}) = \{(x,0)\vb x\in [-1,1]\} \cup \{(0,y) \vb y\in (0,1]\}\]
		which is depicted in black in Figure \ref{fig:a2-singularity}.
		
		\begin{figure}[ht]
			\centering	
			\begin{tikzpicture}
				\draw[thick] (0,0) to (-2,0);
				\draw[thick] (0,0) to (2,0);
				\draw[thick] (0,0) to (0,2);
				
				\node[blue] at (-2.3,0) (K0) {$K_0$};
				\node[blue] at (2.3,0) (K1) {$K_1$};
				\node[blue] at (0,2.2) (K2) {$K_2:=\Cone(f)$};
				\draw[blue,->] (K0) to[bend right] node[above]{$f$} (K1);
			\end{tikzpicture}
			
			\caption{$A_2$-arboreal singularity and the dg category $\cA_2$}
			\label{fig:a2-singularity}
		\end{figure}
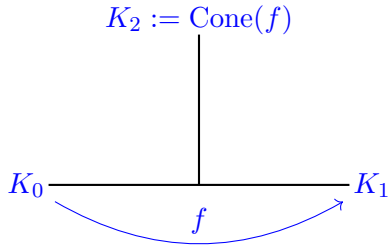
		
		\item The inclusions of the edges
		\begin{align*}
			j_0&\colon D^1\simeq\{(x,0) \vb x\in [-1,-1/2]\}\hookrightarrow A_2,\\
			j_1&\colon D^1\simeq\{(x,0) \vb x\in [1/2,1]\}\hookrightarrow A_2,\\
			j_2&\colon D^1\simeq\{(0,y) \vb y\in [1/2,1]\}\hookrightarrow A_2
		\end{align*}
		induce (up to shift) the dg functors
		\begin{align*}
			j_i\colon\cW(D^*D^1) \to\cW(D^*D^1,\{(0,1)\}), \qquad
			K \mapsto K_i
		\end{align*}
		for $i=0,1,2$, where $K_2:=\Cone(f)$, which corresponds to a linking disk for the stop $\{(0,1)\}$. See Figure \ref{fig:a2-singularity}.
	\end{enumerate}
\end{prp}

\begin{proof}
	$\cW(D^*D^n)$ (resp.\ $\cW(D^*D^1,\{(0,1)\})$) is generated by the cotangent fiber $K$ by \cite{abouzaid-wrapped-generation} (resp.\ the cotangent fibers $K_0,K_1$ by \cite{CDRGG} or \cite{gps2}). The rest of the proposition is a standard computation.
\end{proof}

Recall that by Definition \ref{dfn plumbing sector}, for a fixed $n\geq 1$, the plumbing sector $\Pi_n$ is given by a Weinstein pair (or equivalently, a Weinstein sector) 
\[\Pi_n := \left(D^{2n}, \lambda_n, \Lambda_1 \sqcup \Lambda_2\right)\]
where the Liouville 1-form is $\lambda_n:= \sum_{i=1}^n \tfrac{1}{2} \left(x_i d x_{n+i} - x_{n+i} d x_i\right)$, and the stops are $\Lambda_1:=S^{n-1} \times \mathbf{0}_n$ and $\Lambda_2:=\mathbf{0}_n \times S^{n-1}$.

To calculate the wrapped Fukaya category $\cW(\Pi_n)$, we need to deal with $n\geq 3$ (Theorem \ref{thm:wfuk-plumbing-3}) and $n=2$ (Theorem \ref{thm:wfuk-plumbing-2}) cases separately. We note that the grading structure and the background class needed when defining $\cW(\Pi_n)$ are uniquely determined by Remark \ref{rmk:grading-pin}.

Now, we present Theorem \ref{thm:wfuk-plumbing-3} and \ref{thm:wfuk-plumbing-2}. We postpone their proofs to Section \ref{sec:proof-plumbing-sector}. 

\begin{thm}\label{thm:wfuk-plumbing-3}
	Fix a natural number $n\geq 3$.
	\begin{enumerate}
		\item\label{item:plumbing-3} The wrapped Fukaya category of the plumbing sector $\Pi_n$ is given, up to pretriangulated equivalence, by
		\[\cW(\Pi_n)\simeq\cD_n^{12}\]
		where $\cD_n^{12}$ is the semifree dg category given as follows:
		\begin{enumerate}[label = (\roman*)]
			\item {\em Objects:} $L_1,L_2$ (representing a linking disk of the stop $\Lambda_1$ and $\Lambda_2$, respectively).
			\item {\em Generating morphisms:}
			\[\begin{tikzcd}
				L_1\ar[r,"x", bend left]& L_2\ar[l,"y",bend left]
			\end{tikzcd}\]
			\item {\em Degrees:} $|x|=0,\quad |y|=2-n$.
			\item {\em Differentials:} $dx=dy=0$.
		\end{enumerate}
		
		\item \label{item:inclusion-plumbing-3} The inclusions $\Phi\colon T^*S^{n-1}\hookrightarrow\Pi_n$ and $\Psi\colon T^*S^{n-1}\hookrightarrow\Pi_n$ in \eqref{eq:plumbing-inclusion-1} (when equipped with appropriate grading data) induce the dg functors
		\begin{gather*}
			\Phi\colon \cW(T^*S^{n-1}) \to\cW(\Pi_n)\\
			L \mapsto L_1, \qquad
			z \mapsto yx\\
			\Psi\colon \cW(T^*S^{n-1}) \to\cW(\Pi_n)\\
			L \mapsto L_2, \qquad
			z \mapsto xy
		\end{gather*}
		where $\cW(T^*S^{n-1})$ is given in Proposition \ref{prp:wfuk-sphere}. See Figure \ref{fig:wfuk-plumb-3}.
		
		\begin{figure}[ht]
			\centering	
			\begin{tikzpicture}	
				\draw[domain=0:1.5,smooth,variable=\x] plot ({\x^2},{\x});
				\draw[domain=0:1.5,smooth,variable=\x] plot ({\x^2},{-\x});
				\draw[domain=0:1.5,smooth,variable=\x] plot ({-\x^2},{\x});
				\draw[domain=0:1.5,smooth,variable=\x] plot ({-\x^2},{-\x});
				\draw (1.5^2,-1.5) to[bend left] (1.5^2,1.5);
				\draw (1.5^2,-1.5) to[bend right] (1.5^2,1.5);
				\draw (-1.5^2,-1.5) to[bend left] (-1.5^2,1.5);
				\draw[dashed] (-1.5^2,-1.5) to[bend right] (-1.5^2,1.5);
				
				\node[blue] at (-1.6,0) (X) {$L_1$};
				\node[blue] at (1.6,0) (Y) {$L_2$};
				\draw[blue,->] (X) to[bend left] node[above]{$x$} (Y);
				\draw[blue,->] (Y) to[bend left] node[below]{$y$} (X);
				\draw[blue,dashed,<-] (-1.6,0.3) arc (10:355:0.7cm and 2cm) node[midway,left]{$yx$};
				\draw[blue,dashed,<-] (1.6,0.3) arc (170:-175:0.7cm and 2cm) node[midway,right]{$xy$};
			\end{tikzpicture}
			
			\caption{The wrapped Fukaya category $\cW(\Pi_n)$ for $n\geq 3$}
			\label{fig:wfuk-plumb-3}
		\end{figure}
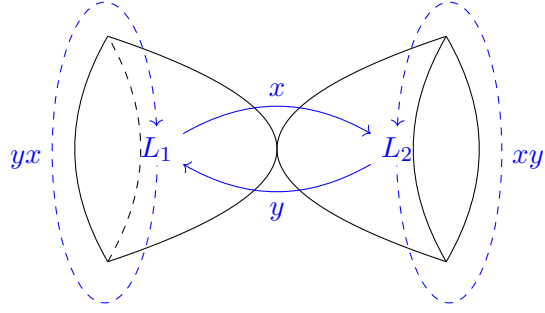
	\end{enumerate}
\end{thm}

\begin{rmk}
	In this remark, we visualize why $\Phi$ sends $z$ to $yx$, without rigorous proof. 
	From the construction of plumbing spaces given in Section \ref{section plumbing space}, $\Phi$ is induced from the geometric inclusion that identifies the zero section of $T^*S^{n-1}$ to $\Lambda_1$. 
	Since the generating morphism $z$ can be seen as a geodesic on the zero section $S^{n-1}$, $\Phi$ would send $z$ to a morphism corresponding to a geodesic on the Lagrangian skeleton of the plumbing sector $\Pi_n$. 
	Similarly, one can identify $x$ and $y$ to geodesics on the Lagrangian skeletons, which connect two points corresponding to the cocores $L_1, L_2$. 
	Thus, the fact that $\Phi$ sending $z$ to $yx$ would imply that the existence of holomorphic triangle bounded by three geodesics corresponding to $x, y$ and $z$. 
	The triangle is visualized in Figure \ref{figure triangle}.
	\begin{figure}[ht]
		\centering
		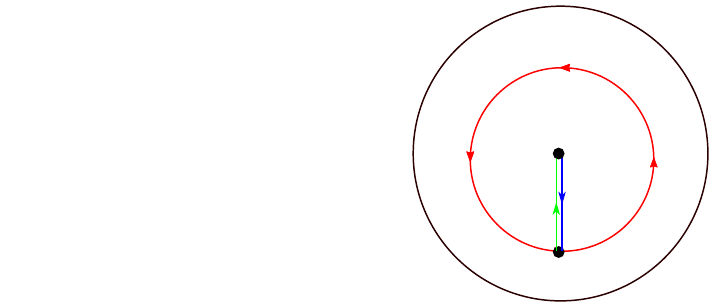		
		\caption{The left is the Lagrangian skeleton of $\Pi_n$, i.e., two disks transversely intersecting at one point. 
		The red, green, blue geodesics correspond to $z, x, y$ respectively.
		The right figure is one of two disks of the Lagrangian skeleton, and the geodesics on the disk are drawn in the same colors. 
		One can see that there exists a triangle bounded by red, green, blue geodesics on the right figure.}
		\label{figure triangle}
	\end{figure}
	
	We also note that $\Phi$ is fully faithful.
	It can be proven by simple computation. 
	Moreover, also by simple computation, one can easily check that $\Hom(L_1,L_1)=\hom(L_1,L_1)=k[yx]$ as an algebra, $\Hom(L_1,L_2) = \hom(L_1,L_2) = k \langle x(yx)^n | n \in \mathbb{N} \rangle$ as a $k$-module, and so on. 
\end{rmk}

\begin{rmk}\label{rmk:shifting-generators}
	The functor $\Phi$ (resp.\ $\Psi$) in Theorem \ref{thm:wfuk-plumbing-3}\eqref{item:inclusion-plumbing-3} is determined from \eqref{eq:plumbing-inclusion-1} up to a choice of grading on the linking disk $L_1$ (resp.\ $L_2$). In other words, we could have considered a different induced functor $\tilde\Phi$ satisfying $\tilde \Phi(L)=L_1[m_1]$ for some $m_1\in\Z$ (resp.\ $\tilde\Psi$ satisfying $\tilde \Psi(L)=L_2[m_2]$ for some $m_2\in\Z$). To make this compatible with the equivalence given in Theorem \ref{thm:wfuk-plumbing-3}\eqref{item:plumbing-3}, we can use a pretriangulated equivalence $\cW(\Pi_n)\simeq \tilde\cD_n^{12}$ (for $n\geq 3$) instead, where the semifree dg category $\tilde\cD_n^{12}$ is given as follows:
	\begin{enumerate}[label = (\roman*)]
		\item {\em Objects:} $L_1[m_1],\quad L_2[m_2]$.
		\item {\em Generating morphisms:} $\tilde x\colon L_1[m_1]\to L_2[m_2],\quad \tilde y\colon L_2[m_2]\to L_1[m_1]$.
		\item {\em Degrees:} $|\tilde x|=m_1-m_2,\quad |\tilde y|=2-n-(m_1-m_2)$.
		\item {\em Differentials:} $d\tilde x=d\tilde y=0$.
	\end{enumerate}
	Then, the dg functors $\tilde\Phi$ and $\tilde\Psi$ are described as follows:
	\begin{gather*}
		\tilde \Phi\colon \cW(T^*S^{n-1}) \to\cW(\Pi_n)\hspace{3cm} \tilde \Psi\colon \cW(T^*S^{n-1}) \to\cW(\Pi_n)\\
		L \mapsto L_1[m_1],\quad
		z \mapsto (-1)^{n(m_1-m_2)}\tilde y \tilde x,\hspace{3cm}
		L \mapsto L_2[m_2],\quad
		z \mapsto \tilde x\tilde y .
	\end{gather*}
\end{rmk}

\begin{proof}[Proof of Remark \ref{rmk:shifting-generators}]
	It is easy to see that the dg category $\tilde\cD_n^{12}$ is the full dg subcategory of $\Tw(\cD_n^{12})$ with the objects $L_1[m_1]$ and $L_2[m_2]$. Indeed, it has the generating morphisms
	\[\begin{tikzcd}
		L_1[m_1]\ar[rrr,"\tilde x:=1_{m_2,m_1}\otimes x", bend left]& & & L_2[m_2]\ar[lll,"\tilde y:=1_{m_1,m_2}\otimes y",bend left]
	\end{tikzcd}\]
	where we consider $L_i[m_i]$ as $k[m_i]\otimes L_i$ for $i=1,2$, and $1_{p,q}\colon k[q]\to  k[p]$ is the canonical isomorphism of degree $q-p$ (recall that $k$ is the coefficient ring). Note that $d\tilde x=d\tilde y=0$, $|\tilde x|=m_1-m_2$, and $|\tilde y|=2-n-(m_1-m_2)$. Therefore, $\tilde\cD_n^{12}$ is pretriangulated equivalent to $\cD_n^{12}$, and hence, to $\cW(\Pi_n)$.
	
	To determine the functor $\tilde \Phi$, first consider the $m_1$-shift functor
	\begin{align*}
		k[m_1]\otimes -\colon\cW(\Pi_n)&\to\cW(\Pi_n)\\
		A&\mapsto k[m_1]\otimes A=:A[m_1] &&\text{for every object $A$ in $\cW(\Pi_n)$},\\
		\theta &\mapsto 1_{m_1,m_1}\otimes\theta &&\text{for every morphism $\theta$ in $\cW(\Pi_n)$} .
	\end{align*}
	The functor $\tilde \Phi$ is the composition of the $m_1$-shift functor with the functor $\Phi$ in Theorem \ref{thm:wfuk-plumbing-3}. Hence, we have
	\begin{align*}
		\tilde \Phi\colon \cW(T^*S^{n-1}) &\to\cW(\Pi_n)\\
		L \mapsto L_1[m_1],\quad z&\mapsto 1_{m_1,m_1}\otimes yx .
	\end{align*}
	We want to determine $1_{m_1,m_1}\otimes yx$ in terms of $\tilde x$ and $\tilde y$. For that, we have the equalities
	\begin{align*}
		1_{m_1,m_1}\otimes yx&=(1_{m_1,m_2}\circ 1_{m_2,m_1})\otimes (y\circ x)\\
		&=(-1)^{n(m_1-m_2)}(1_{m_1,m_2}\otimes y)\circ(1_{m_2,m_1}\otimes x)\\
		&=(-1)^{n(m_1-m_2)}\tilde y \tilde x
	\end{align*}
	using the Koszul sign rule. Hence, we indeed have $\tilde \Phi(z)=(-1)^{n(m_1-m_2)}\tilde y\tilde x$. A similar calculation determines the functor $\tilde \Psi$.
\end{proof}

\begin{thm}\label{thm:wfuk-plumbing-2}\mbox{}
	\begin{enumerate}
		\item\label{item:plumbing-2} The wrapped Fukaya category of the plumbing sector $\Pi_2$ is given, up to pretriangulated equivalence, by
		\[\cW(\Pi_2)\simeq\cD_2^{12}[(1_{L_2}+xy)^{-1}]\]
		where $\cD_2^{12}$ is the semifree dg category given as follows:
		\begin{enumerate}[label = (\roman*)]
			\item {\em Objects:} $L_1,L_2$ (representing a linking disk of the stop $\Lambda_1$ and $\Lambda_2$, respectively).
			\item {\em Generating morphisms:}
			\[\begin{tikzcd}
				L_1\ar[r,"x", bend left]& L_2\ar[l,"y",bend left]
			\end{tikzcd}\]
			\item {\em Degrees:} $|x|=0,\quad |y|=0$.
			\item {\em Differentials:} $dx=dy=0$.
		\end{enumerate}
		
		\item \label{item:inclusion-plumbing-2} The inclusions $\Phi\colon T^*S^{1}\hookrightarrow\Pi_2$ and $\Psi\colon T^*S^{1}\hookrightarrow\Pi_2$ in \eqref{eq:plumbing-inclusion-1} (when equipped with appropriate grading data) induce the dg functors
		\begin{gather*}
			\Phi\colon \cW(T^*S^1) \to\cW(\Pi_2) \hspace{3cm}
			\Psi\colon \cW(T^*S^1) \to\cW(\Pi_2)\\
			L \mapsto L_1, 
			z \mapsto 1_{L_1}+yx  \hspace{3cm}
			L \mapsto L_2, 
			z \mapsto 1_{L_2}+xy
		\end{gather*}
		where $\cW(T^*S^1)$ is given in Proposition \ref{prp:wfuk-sphere}. See Figure \ref{fig:wfuk-plumb-2}.
		
		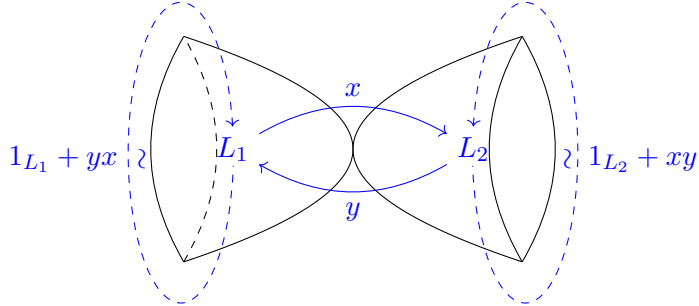
\begin{figure}[ht]
			\centering	
			\begin{tikzpicture}	
				\draw[domain=0:1.5,smooth,variable=\x] plot ({\x^2},{\x});
				\draw[domain=0:1.5,smooth,variable=\x] plot ({\x^2},{-\x});
				\draw[domain=0:1.5,smooth,variable=\x] plot ({-\x^2},{\x});
				\draw[domain=0:1.5,smooth,variable=\x] plot ({-\x^2},{-\x});
				\draw (1.5^2,-1.5) to[bend left] (1.5^2,1.5);
				\draw (1.5^2,-1.5) to[bend right] (1.5^2,1.5);
				\draw (-1.5^2,-1.5) to[bend left] (-1.5^2,1.5);
				\draw[dashed] (-1.5^2,-1.5) to[bend right] (-1.5^2,1.5);
				
				\node[blue] at (-1.6,0) (X) {$L_1$};
				\node[blue] at (1.6,0) (Y) {$L_2$};
				\draw[blue,->] (X) to[bend left] node[above]{$x$} (Y);
				\draw[blue,->] (Y) to[bend left] node[below]{$y$} (X);
				\draw[blue,dashed,<-] (-1.6,0.3) arc (10:355:0.7cm and 2cm) node[midway,left]{$1_{L_1}+yx$} node[midway,rotate=90,below]{$\sim$};
				\draw[blue,dashed,<-] (1.6,0.3) arc (170:-175:0.7cm and 2cm) node[midway,right]{$1_{L_2}+xy$} node[midway,rotate=90,above,yshift=-0.1cm]{$\sim$};
			\end{tikzpicture}
			
			\caption{The wrapped Fukaya category $\cW(\Pi_2)$}
			\label{fig:wfuk-plumb-2}
		\end{figure}
	\end{enumerate}
\end{thm}

\begin{rmk}\label{rmk:wfuk-plumb-2-inv}
	We could also write $\cW(\Pi_2)\simeq\cD_2^{12}[(1_{L_1}+yx)^{-1}]$. Having either $1_{L_1}+yx$ or $1_{L_2}+xy$ invertible up to homotopy implies both of $1_{L_1}+yx$ and $1_{L_2}+xy$ are invertible up to homotopy.
\end{rmk}

\begin{rmk}\label{rmk:shifting-generators-2}
	As in Remark \ref{rmk:shifting-generators}, we could have considered a different functor $\tilde\Phi$ induced from \eqref{eq:plumbing-inclusion-1} satisfying $\tilde \Phi(L)=L_1[m_1]$ for some $m_1\in\Z$ (resp.\ $\tilde\Psi$ induced from \eqref{eq:plumbing-inclusion-1} satisfying $\tilde \Psi(L)=L_2[m_2]$ for some $m_2\in\Z$) in Theorem \ref{thm:wfuk-plumbing-2}\eqref{item:inclusion-plumbing-2}. To make this compatible with the equivalence given in Theorem \ref{thm:wfuk-plumbing-2}\eqref{item:plumbing-2}, we can use a pretriangulated equivalence $\cW(\Pi_2)\simeq \tilde\cD_2^{12}[(1_{L_2[m_2]}+\tilde x\tilde y)^{-1}]$ instead, where the semifree dg category $\tilde\cD_2^{12}$ is given as follows:
	\begin{enumerate}[label = (\roman*)]
		\item {\em Objects:} $L_1[m_1],\quad L_2[m_2]$.
		\item {\em Generating morphisms:} $\tilde x\colon L_1[m_1]\to L_2[m_2],\quad \tilde y\colon L_2[m_2]\to L_1[m_1]$.
		\item {\em Degrees:} $|\tilde x|=m_1-m_2,\quad |\tilde y|=m_2-m_1$.
		\item {\em Differentials:} $d\tilde x=d\tilde y=0$.
	\end{enumerate}
	Then, the dg functors $\tilde\Phi$ and $\tilde\Psi$ are described as follows:
	\begin{gather*}
		\tilde \Phi\colon \cW(T^*S^1) \to\cW(\Pi_2)\hspace{3cm} \tilde \Psi\colon \cW(T^*S^1) \to\cW(\Pi_2)\\
		L \mapsto L_1[m_1],\quad
		z \mapsto 1_{L_1[m_1]}+\tilde y \tilde x,\hspace{2cm}
		L \mapsto L_2[m_2],\quad
		z \mapsto 1_{L_2[m_2]}+\tilde x\tilde y .
	\end{gather*}
\end{rmk}

\subsection{Proof of the computations of $\cW(\Pi_n)$}\label{sec:proof-plumbing-sector}

The goal of this subsection is to prove Theorem \ref{thm:wfuk-plumbing-3} and \ref{thm:wfuk-plumbing-2}. We start with the description of the dg category obtained by adding a cone of a morphism to a given dg category. It will be used to prove Lemma \ref{lem:generator-change}.

\begin{prp}\label{prp:extend-cone}
	Let $\cA$ be a dg category, $L_0,L_1\in\cA$, and $g\in\hom^0_{\cA}(L_0,L_1)$ with $dg=0$. Let $\hat\cA$ be the full dg subcategory of $\Tw(\cA)$ with the objects of $\cA$ and $L_2:=\Cone(g)$. Then, the dg category $\hat\cA$ is given as follows:
	\begin{enumerate}[label = (\roman*)]
		\item {\em Objects:} The objects of $\cA$, and $L_2:=\Cone(g)$.
		\item {\em Generating morphisms:} The morphisms in $\cA$, and the morphisms $i_0,i_1,p_0,p_1$ shown in blue below:
		\[\begin{tikzcd}
			& \textcolor{blue}{L_2}\ar[ld,"p_0"',bend right=20,blue]\ar[rd,"p_1",bend left=20,blue]\\
			L_0 \ar[rr,"g"']\ar[ru,"i_0"',blue]&
			& L_1\ar[lu,"i_1",blue]
		\end{tikzcd}\]
		\item {\em Degrees:} The degrees of the morphisms from $\cA$ are the same as in $\cA$, and
		\[|i_0|=-1,\quad |i_1|=0,\quad |p_0|=1,\quad |p_1|=0 .\]
		\item {\em Differentials:} The differentials of the morphisms from $\cA$ are the same as in $\cA$, and
		\[di_0=i_1 g,\quad
		di_1=0,\quad
		dp_0=0,\quad
		dp_1=-g p_0 .\]
		\item {\em Relations:} The relations between the morphisms from $\cA$ are the same as in $\cA$. The compositions involving $i_0,i_1,p_0,p_1$ are free, except
		\[p_0 i_0 =1_{L_0},\quad
		p_0 i_1 =0,\quad
		p_1 i_0 =0,\quad
		p_1 i_1 =1_{L_1},\quad
		i_0 p_0+i_1 p_1=1_{L_2} .\]
	\end{enumerate}
\end{prp}
We note that the dg category $\hat{\cA}$ in the above proposition is not semifree since it has nontrivial relations. 

\begin{proof}[Proof of Proposition \ref{prp:extend-cone}]
	Write $L_2:=\Cone(g)$ as the twisted complex
	\[L_2=\left(L_0[1]\oplus L_1,\mx{0 & 0 \\ 1_{0,1}\otimes g & 0}\right)\]
	where the second term denotes the differential of the twisted complex, and $1_{p,q}\colon k[q]\to  k[p]$ is the canonical isomorphism of degree $q-p$ (recall that $k$ is the coefficient ring).
	Then, we have the morphisms
	\begin{gather*}
		i_0:=\mx{1_{1,0}\otimes 1_{L_0} \\ 0}\colon L_0 \to L_2, \qquad
		i_1:=\mx{0 \\ 1_{L_1}}\colon L_1 \to L_2,\\
		p_0:=\mx{1_{0,1}\otimes 1_{L_0}  0}\colon L_2 \to L_0, \qquad
		p_1:=\mx{0  1_{L_1}}\colon L_2 \to L_1 .
	\end{gather*}
	Obviously, the morphisms in $\hat\cA$ are generated by the morphisms in $\cA$, and $i_0,i_1,p_0,p_1$. The compositions of the morphisms in $\cA$ with either of $i_0,i_1,p_0,p_1$ are free. Gradings, differentials, and compositions of $i_0,i_1,p_0,p_1$ follow from the definition of $\Tw(\cA)$, see \cite{seidel}. As an example, by seeing the object $L_0$ as the twisted complex $(L_0,0)$, we have
	\begin{align*}
		di_0=d\mx{1_{1,0} \otimes 1_{L_0} \\ 0}&=\mx{d(1_{1,0} \otimes 1_{L_0}) \\ 0}+\mx{0 & 0 \\ 1_{0,1}\otimes g & 0}\mx{1_{1,0} \otimes 1_{L_0} \\ 0} - (-1)^{|i_0|}\mx{1_{1,0} \otimes 1_{L_0} \\ 0} 0\\
		&=\mx{(d1_{1,0}) \otimes 1_{L_0}+(-1)^{|1_{1,0}|} 1_{1,0}\otimes (d1_{L_0}) \\ 0}+\mx{0 \\ (1_{0,1}\otimes g)\circ (1_{1,0} \otimes 1_{L_0})}\\
		&=\mx{0 \\ (-1)^{|g| |1_{1,0}|}(1_{0,1}\circ 1_{1,0}) \otimes (g\circ 1_{L_0})}=\mx{0 \\ g}=\mx{0 \\ 1_{L_1}} g\\
		&=i_1 g
	\end{align*}
	where the graded Leibniz rule and Koszul sign rule are used whenever needed.
\end{proof}

The following lemma is about changing generators of the triangulated closure of a particular dg category. It will be used in the proof of Theorem \ref{thm:wfuk-plumbing-3} and \ref{thm:wfuk-plumbing-2}.

\begin{lem}\label{lem:generator-change}
	For any $n\in\Z$, let $\cD_n^{01}$ be a semifree dg category given as follows:
	\begin{enumerate}[label = (\roman*)]
		\item {\em Objects:} $L_0,L_1$.
		\item {\em Generating morphisms:}
		$\begin{tikzcd}
			L_0\ar[r,"g"]\ar[r,"h"', bend right]& L_1\ar[loop right,"\alpha_1"]
		\end{tikzcd}.$
		\item {\em Degrees:} $|g|=0,\quad |\alpha_1|=2-n,\quad |h|=1-n$.
		\item {\em Differentials:} $dg=d\alpha_1=0,\quad dh=\alpha_1 g$.
	\end{enumerate}
	Then, the following hold:
	\begin{enumerate}
		\item The full dg subcategory of $\Tw(\cD_n^{01})$ with the objects $L_1$ and $L_2:=\Cone(g)$ is quasi-equivalent to the semifree dg category $\cD_n^{12}$ given as follows:
		\begin{enumerate}[label = (\roman*)]
			\item {\em Objects:} $L_1,L_2$.
			\item {\em Generating morphisms:}
			$\begin{tikzcd}
				L_1\ar[r,"x", bend left]& L_2\ar[l,"y",bend left]
			\end{tikzcd}.$
			\item {\em Degrees:} $|x|=0,\quad |y|=2-n$.
			\item {\em Differentials:} $dx=dy=0$.
		\end{enumerate}
		
		\item We have the pretriangulated equivalence $\cD_n^{12}\simeq\cD_n^{01}$, or equivalently, the quasi-equivalence $\Tw(\cD_n^{12})\simeq\Tw(\cD_n^{01})$ induced by
		\begin{align*}
			\cD_n^{12}&\to\Tw(\cD_n^{01})\\
			L_1\mapsto L_1,\quad L_2&\mapsto L_2,\quad x\mapsto i_1,\quad y\mapsto (-1)^n hp_0 + \alpha_1 p_1
		\end{align*}
		where the morphisms $i_0,i_1,p_0,p_1$ are as in Proposition \ref{prp:extend-cone}.
		
		\item The induced quasi-isomorphism
		\[\hom^*_{\Tw(\cD_n^{12})}(L_j,L_j)\to \hom^*_{\Tw(\cD_n^{01})}(L_j,L_j)\]
		sends $yx$ to $\alpha_1$ when $j=1$, and $xy$ to $\alpha_2:=i_1\alpha_1 p_1 + (-1)^n i_1hp_0$ when $j=2$.
	\end{enumerate}
\end{lem}

\begin{proof}
	First, we prove the first item. Let $\cD_n^{012}$ be the full dg subcategory of $\Tw(\cD_n^{01})$ with the objects $L_0,L_1,L_2=\Cone(g)$. Then by Proposition \ref{prp:extend-cone}, the morphisms in $\cD_n^{012}$ are generated by the morphisms below:
	\[\begin{tikzcd}
		& \textcolor{blue}{L_2}\ar[ld,"p_0"',bend right=20,blue]\ar[rd,"p_1",bend left=20,blue]\\
		L_0 \ar[rr,"g"]\ar[rr,"h",bend right]\ar[ru,"i_0"',blue]&
		& L_1\ar[lu,"i_1",blue]\ar[loop right,"\alpha_1"]
	\end{tikzcd}\]
	The compositions are free, except the ones given in Proposition \ref{prp:extend-cone}. The gradings and differentials of $i_0,i_1,p_0,p_1$ are given in Proposition \ref{prp:extend-cone}.
	
	Now, we want to consider the full dg subcategory $\cE_n^{12}$ of $\cD_n^{012}$ consisting of the objects $L_1$ and $L_2$. Note that since both $\{L_0,L_1\}$ and $\{L_1,L_2\}$ generate $\cD_n^{012}$ (because $L_2$ is the cone of $g\colon L_0\to L_1$), we have the quasi-equivalences
	\[\Tw(\cE_n^{12})\simeq\Tw(\cD_n^{012})\simeq\Tw(\cD_n^{01}) .\]
	By setting
	\[x:=i_1,\quad a:=p_1,\quad b:= g p_0,\quad c:=h p_0,\]
	the dg category $\cE_n^{12}$ is given as follows:
	\begin{enumerate}[label = (\roman*)]
		\item {\em Objects:} $L_1,L_2$.
		\item {\em Generating morphisms:}
		$\begin{tikzcd}
			L_1\ar[loop left,"\alpha_1"]\ar[r,"x",bend left]& L_2\ar[l,"{a,b,c}", bend left]
		\end{tikzcd}.$
		\item {\em Degrees:} $|\alpha_1|=2-n,\quad |x|=0,\quad |a|=0,\quad |b|=1,\quad |c|=2-n$.
		\item {\em Differentials:} $d\alpha_1=0,\quad dx=0,\quad da=-b,\quad db=0,\quad dc=\alpha_1 b$.
		\item {\em Relations:} $a x =1_{L_1},\quad b x =0,\quad c x =0$.
	\end{enumerate}
	We can simplify $\cE_n^{12}$ by defining $y:=(-1)^n c + \alpha_1 a$ to replace $c$. In that case, we have $y x= \alpha_1$, so we can remove $\alpha_1$ from the generating morphisms. Hence, we can express the dg category $\cE_n^{12}$ as follows:
	\begin{enumerate}[label = (\roman*)]
		\item {\em Objects:} $L_1,L_2$.
		\item {\em Generating morphisms:}
		$\begin{tikzcd}
			L_1\ar[r,"x", bend left]& L_2\ar[l,"{a,b,y}",bend left]
		\end{tikzcd}.$
		\item {\em Degrees:} $|x|=0,\quad |y|=2-n,\quad |a|=0,\quad |b|=1$.
		\item {\em Differentials:} $dx=0,\quad dy=0,\quad da=-b,\quad db=0$.
		\item {\em Relations:} $a x =1_{L_1},\quad b x =0$.
	\end{enumerate}
	Finally, we want to get rid of $a$ and $b$ to show that $\cE_n^{12}$ and $\cD_n^{12}$ are quasi-equivalent. For that, define the dg functor
	\begin{align*}
		F\colon\cD_n^{12}&\to\cE_n^{12}\\
		L_1\mapsto L_1,\quad L_2&\mapsto L_2,\quad x\mapsto x,\quad y\mapsto y .
	\end{align*}
	To show that $F$ is a quasi-equivalence, we need to show that the induced chain maps
	\[F_{ij}\colon \hom^*_{\cD_n^{12}}(L_i,L_j)\to\hom^*_{\cE_n^{12}}(L_i,L_j)\]
	are quasi-isomorphisms for any $i,j\in\{1,2\}$. Here, we will just consider the case of $i=2, j=1$. The other cases can be easily proven in a similar way.
	
	Note that $\hom^*_{\cE_n^{12}}(L_2,L_1)$ is additively generated by the morphisms
	\[(yx)^m y , \quad (yx)^m a, \quad (yx)^m b ,\]
	where $m\geq 0$. Since $d((yx)^m a)= -(-1)^{nm}(yx)^m b$ for all $m\geq 0$, we can get rid of $(yx)^m a, (yx)^m b$ in the cohomology. Hence, $\{(yx)^m y\vb m\geq 0\}$  additively generates the cohomology of $\hom^*_{\cE_n^{12}}(L_2,L_1)$. We also know that $\{(yx)^m y\vb m\geq 0\}$ additively generates the cohomology of $\hom^*_{\cD_n^{12}}(L_2,L_1)$, and $F_{21}$ sends $(yx)^m y$ to $(yx)^m y$. Hence, $F_{21}$ is a quasi-isomorphism.
	
	Therefore, the full dg subcategory $\cE_n^{12}$ of $\Tw(\cD_n^{01})$ with the objects $L_1$ and $L_2=\Cone(g)$ is quasi-equivalent to $\cD_n^{12}$. This proves the first item. In particular, we have the quasi-equivalences
	\[\Tw(\cD_n^{12})\simeq\Tw(\cE_n^{12})\simeq \Tw(\cD_n^{01})\]
	induced by
	\begin{alignat*}{2}
		\cD_n^{12}&\xrightarrow{F}\cE_n^{12}&&\mapsto \Tw(\cD_n^{01})\\
		L_1&\mapsto L_1&&\mapsto L_1\\
		L_2&\mapsto L_2&&\mapsto L_2\\
		x&\mapsto x&&\mapsto i_1\\
		y&\mapsto y&&\mapsto (-1)^n hp_0 + \alpha_1 p_1
	\end{alignat*}
	which proves the second item. The third item is easy to confirm from this map.
\end{proof}

Now, we are ready to prove Theorem \ref{thm:wfuk-plumbing-3}.

\begin{proof}[Proof of Theorem \ref{thm:wfuk-plumbing-3}]
	One can easily see that the Weinstein sector $\Pi_n$ also corresponds to the Weinstein pair $(D^*D^n,S^*_0D^n)$ by the stop-sector correspondence in \cite{gps1} (or by Proposition \ref{prp:pair-sector-correspondence}), hence we have the quasi-equivalence
	\[\cW(\Pi_n)\simeq\cW(D^*D^n,S^*_0D^n) .\]
	Also, by perturbing all Lagrangians with an appropriate Hamiltonian whose flow on $S^*D^n$ is the reverse Reeb flow, we have the quasi-equivalence
	\[\cW(D^*D^n,S^*_0D^n)\simeq\cW(D^*D^n,\Lambda)\]
	where $\Lambda$ is the perturbation of $S^*_0 D^n$ by the flow of the Hamiltonian. Hence, by an appropriate choice of such Hamiltonian, $\Lambda$ can be seen as the boundary of the outward conormal bundle of a sphere $S_{1/2}^{n-1}\subset D^n$ centered at the origin with radius $1/2$. Explicitly, it is given by
	\[\Lambda:=\left.\left\{\left(\frac{1}{\sqrt{2}}x_1,\ldots,\frac{1}{\sqrt{2}}x_n,x_1,\ldots,x_n\right)\, \right|\,  x_1^2+\ldots+x_n^2=1\right\}\subset S^*D^n .\]
	
	For $n=2$ (for illustration purposes), Figure \ref{fig:reeb-flow-n=2} depicts $\Lambda\subset S^*D^n$ as the boundary of the outward conormal bundle of the red circle $\pi(\Lambda)$ of radius $1/2$, where $\pi\colon S^*D^n\to D^n$ is the projection map, and a blue arrow at any point $p$ on the red circle $\pi(\Lambda)$ represents a point of $\Lambda$ by giving a codirection in $S^*_pD^n$.
	
	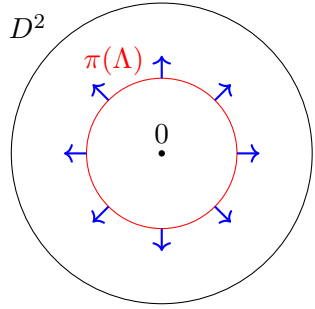
\begin{figure}[ht]
		\centering
		\begin{tikzpicture}
			\draw (0,0) circle (2cm) node[above left,xshift=-1.4cm,yshift=1.4cm]{$D^2$};
			\draw[red] (0,0) circle (1cm) node[above left,xshift=-0.1cm,yshift=0.9cm]{$\pi(\Lambda)$};
			\draw[fill=black] (0,0) circle (1 pt) node [above] {0};
			\draw[->,blue,thick] ({cos(0)},{sin(0)}) -- ({1.3*cos(0)},{1.3*sin(0)});
			\draw[->,blue,thick] ({cos(45)},{sin(45)}) -- ({1.3*cos(45)},{1.3*sin(45)});
			\draw[->,blue,thick] ({cos(90)},{sin(90)}) -- ({1.3*cos(90)},{1.3*sin(90)});
			\draw[->,blue,thick] ({cos(135)},{sin(135)}) -- ({1.3*cos(135)},{1.3*sin(135)});
			\draw[->,blue,thick] ({cos(180)},{sin(180)}) -- ({1.3*cos(180)},{1.3*sin(180)});
			\draw[->,blue,thick] ({cos(225)},{sin(225)}) -- ({1.3*cos(225)},{1.3*sin(225)});
			\draw[->,blue,thick] ({cos(270)},{sin(270)}) -- ({1.3*cos(270)},{1.3*sin(270)});
			\draw[->,blue,thick] ({cos(315)},{sin(315)}) -- ({1.3*cos(315)},{1.3*sin(315)});
		\end{tikzpicture}
		\caption{Perturbation $\Lambda\subset S^*D^n$ of $S^*_0D^n$ by the reverse Reeb flow for the case $n=2$}
		\label{fig:reeb-flow-n=2}
	\end{figure}
	
	From now on, we will consider every Weinstein pair as a Weinstein sector via stop-sector correspondence. The Weinstein pair $(D^*D^n,\Lambda)$ has a Weinstein sectorial covering by
	\[D^*D^n_{1/4}\quad\text{and}\quad(D^*(D^n\setminus \intr(D^n_{1/4})),\Lambda),\]
	where $D^n_{1/4}\subset D^n$ is the radius $1/4$ disk centered at the origin, and $\intr$ stands for interior. Hence, by \cite{gps2} (Theorem \ref{thm:gps}), we have the pretriangulated equivalence
	\[\cW(D^*D^n,\Lambda)\simeq\hocolim(\cW(D^*(D^n\setminus \intr(D^n_{1/4})),\Lambda)\leftarrow \cW(D^*S^{n-1})\rightarrow \cW(D^*D^n_{1/4})) .\]
	Moreover, since we have $(D^*D^1,\{(0,1)\})\times D^*S^{n-1}\simeq (D^*(D^n\setminus \intr(D^n_{1/4})),\Lambda)$, we have the K\"unneth embedding (a quasi-full and faithful functor) by \cite{gps2}
	\[\cW(D^*D^1,\{(0,1)\})\otimes \cW(D^*S^{n-1})\hookrightarrow\cW(D^*(D^n\setminus \intr(D^n_{1/4})),\Lambda)\]
	which is a pretriangulated equivalence since it hits the generators of $\cW(D^*(D^n\setminus \intr(D^n_{1/4})),\Lambda)$.
	
	Collecting all the results, we have the pretriangulated equivalence
	\begin{equation}\label{eq:wfuk-plumbing-sector-hocolim}
		\cW(\Pi_n)\simeq\hocolim(\cW(D^*D^1,\{(0,1)\})\otimes \cW(D^*S^{n-1})\leftarrow \cW(D^*S^{n-1})\rightarrow \cW(D^*D^n_{1/4})) .
	\end{equation}
	Also, we have the pretriangulated equivalences $\cW(D^*D^n_{1/4})\simeq\cA_1$ and $\cW(D^*D^1,\{(0,1)\})\simeq\cA_2$ by Proposition \ref{prp:msh-a2}, and $\cW(D^*S^{n-1})\simeq \cC_{n-1}$ by Proposition \ref{prp:wfuk-sphere} (since $n-1\geq 2$), where the semifree dg categories $\cA_1$, $\cA_2$, and $\cC_{n-1}$ are described in the relevant propositions. Hence, we get the pretriangulated equivalence
	\begin{gather}
		\label{eqn gps}
		\cW(\Pi_n)\simeq\hocolim(\Tw(\cA_2\otimes \cC_{n-1})\xleftarrow{F} \Tw(\cC_{n-1})\xrightarrow{G} \Tw(\cA_1)).
	\end{gather}
	
	Before going further, we note that the homotopy colimit diagram in \eqref{eqn gps} is similarly given in \cite[Lemma 6.2]{gps3}.
	
	Now, we will describe $\cA_2\otimes \cC_{n-1}$ explicitly. 
	We note that $\cA_2 \otimes \cC_{n-1}$ is {\em derived} tensor product, but since $\cA_2$ and $\cC_{n-1}$ are semifree, the derived tensor product is quasi-equivalent to the strict tensor product. 
	See, for example, \cite{dgcat}. 
	It is the dg category given as follows:
	\begin{enumerate}[label = (\roman*)]
		\item {\em Objects:} $K_0\otimes L,\quad K_1\otimes L$.
		\item {\em Generating morphisms:}
		\[\begin{tikzcd}[column sep=2cm]
			K_0\otimes L\ar[loop left, "1_{K_0}\otimes z"]\ar[r,"f\otimes 1_L"]& K_1\otimes L\ar[loop right,"1_{K_1}\otimes z"]
		\end{tikzcd}\]
		\item {\em Degrees:} $|1_{K_0}\otimes z|= |1_{K_1}\otimes z|=2-n,\quad |f\otimes 1_L|=0$.
		\item {\em Differentials:} $d(1_{K_0}\otimes z)=d(1_{K_1}\otimes z)=d(f\otimes 1_L)=0$.
		\item {\em Relations:} $(1_{K_1}\otimes z)\circ(f\otimes 1_L)=(f\otimes 1_L)\circ (1_{K_0}\otimes z)$.
	\end{enumerate}
	The relation above follows from the equations
	\begin{align*}
		(1_{K_1}\otimes z)\circ(f\otimes 1_L)&=(-1)^{|z| |f|}(1_{K_1}\circ f)\otimes (z\circ 1_L)=f\otimes z,\\
		(f\otimes 1_L)\circ (1_{K_0}\otimes z)&=(-1)^{|1_L| |1_{K_0}|}(f\circ 1_{K_0})\otimes (1_L\circ z)=f\otimes z,
	\end{align*}
	where we used the Koszul sign rule. (Recall that $|f|=0$ and $df=0$.)
	
	The dg functors $F$ and $G$ are induced from the inclusion maps from $D^*S^{n-1}$ to the corresponding boundary of $D^*(D^n\setminus \intr(D^n_{1/4}))$ and $D^*D^n_{1/4}$, respectively. Therefore, it is easy to describe $F$ and $G$ as they send cotangent fibers to cotangent fibers (up to shifts, which does not affect the result since the plumbing sector has a unique grading structure).
	
	Since a dg functor from $\Tw(\cC_{n-1})$ is determined by its restriction to $\cC_{n-1}$, the dg functors $F$ and $G$ are given by
	\begin{alignat*}{5}
		F\colon \cC_{n-1}&\xrightarrow{F'}\cA_2\otimes\cC_{n-1}&&\hookrightarrow \Tw(\cA_2\otimes\cC_{n-1}), \qquad G\colon &\cC_{n-1}&&\xrightarrow{G'}\cA_1 &\hookrightarrow\Tw(\cA_1)\\
		L&\mapsto\hspace{0.3cm} (K_0,L)&&\mapsto (K_0,L) &L&&\mapsto K&\mapsto K\\
		z&\mapsto\hspace{0.3cm} 1_{K_0}\otimes z && \mapsto 1_{K_0}\otimes z & z&&\mapsto 0&\mapsto 0,
	\end{alignat*}
	hence we have the pretriangulated equivalence
	\[\cW(\Pi_n)\simeq\hocolim(\cA_2\otimes \cC_{n-1}\xleftarrow{F'} \cC_{n-1}\xrightarrow{G'} \cA_1) .\]
	To replace $\cA_2\otimes\cC_{n-1}$ with a semifree dg category, define a semifree dg category $\cB_n^{01}$ as follows:
	\begin{enumerate}[label = (\roman*)]
		\item {\em Objects:} $L_0,L_1$.
		\item {\em Generating morphisms:}
		$\begin{tikzcd}
			L_0\ar[loop left, "\alpha_0"]\ar[r,"g"]\ar[r,"h"', bend right=20]& L_1\ar[loop right,"\alpha_1"]
		\end{tikzcd}$.
		\item {\em Degrees:} $|\alpha_0|=|\alpha_1|=2-n,\quad |g|=0,\quad |h|=1-n$.
		\item {\em Differentials:} $d\alpha_0=d\alpha_1=dg=0,\quad dh=\alpha_1 g - g \alpha_0$.
	\end{enumerate}
	Then, we have the quasi-equivalence
	\begin{align}\label{eq:functor-m}
		M\colon\cB_n^{01}&\xrightarrow{\sim}\cA_2\otimes\cC_{n-1},\\
		L_i\mapsto (K_i,L),\quad\alpha_i&\mapsto 1_{K_i}\otimes z, \quad g\mapsto f\otimes 1_L,\quad h\mapsto 0,\notag
	\end{align}
	for $i=0,1$. Note that the dg functor
	\begin{gather*}
		F''\colon \cC_{n-1}\to\cB_n^{01}\\
		L\mapsto L_0,\quad
		z\mapsto \alpha_0
	\end{gather*}
	satisfies $M\circ F''=F'$, hence we have the pretriangulated equivalence
	\[\cW(\Pi_n)\simeq\hocolim(\cB_n^{01}\xleftarrow{F''} \cC_{n-1}\xrightarrow{G'} \cA_1) .\]
	All the dg categories in the homotopy colimit diagram above are semifree, and the inclusion functor $F''$ is a semifree extension. Therefore, by Remark \ref{rmk colimit}, the homotopy colimit becomes a colimit. Taking the colimit results in setting $\alpha_0=0$ in $\cB_n^{01}$ by Proposition \ref{prp:colimit-dg}. Hence, we get the pretriangulated equivalence
	\[\cW(\Pi_n)\simeq \cD_n^{01}\]
	where $\cD_n^{01}$ is the semifree dg category given as follows:
	\begin{enumerate}[label = (\roman*)]
		\item {\em Objects:} $L_0,L_1$.
		\item {\em Generating morphisms:}
		$\begin{tikzcd}
			L_0\ar[r,"g"]\ar[r,"h"', bend right]& L_1\ar[loop right,"\alpha_1"]
		\end{tikzcd}$.
		\item {\em Degrees:} $|g|=0,\quad |\alpha_1|=2-n,\quad |h|=1-n$.
		\item {\em Differentials:} $dg=d\alpha_1=0,\quad dh=\alpha_1 g$.
	\end{enumerate}
	Then, by Lemma \ref{lem:generator-change}, we have the pretriangulated equivalence
	\[\cW(\Pi_n)\simeq\cD_n^{01}\simeq\cD_n^{12},\]
	which proves the first part of Theorem \ref{thm:wfuk-plumbing-3}.
	
	As for the second part, note that the inclusion $\Phi\colon D^*S^{n-1}\to\Pi_n$ can be decomposed as
	\[\Phi\colon D^*S^{n-1}\hookrightarrow(D^*D^1,\{(0,1)\})\times D^*S^{n-1}\xrightarrow{\sim}(D^*(D^n\setminus\intr(D^n_{1/4})),\Lambda)\hookrightarrow\Pi_n,\]
	which is induced by the composition of maps on respective skeleta
	\begin{alignat*}{3}
		\Phi\colon S^{n-1}&\hookrightarrow \hspace{1cm}A_2\times S^{n-1}&&\xrightarrow{\sim}(D^n\setminus\intr(D^n_{1/4}))\cup\Cone(\Lambda)&&\hookrightarrow (D^n\times\mathbf{0}_n)\cup(\mathbf{0}_n\times D^n),\\
		(x_1,\ldots,x_n)&\mapsto ((1,0),(x_1,\ldots,x_n))&&\mapsto \hspace{0.5cm}(x_1,\ldots,x_n,0,\ldots,0)&&\mapsto (x_1,\ldots,x_n,0,\ldots,0),
	\end{alignat*}
	where $A_2$ is the $A_2$-arboreal singularity as in Proposition \ref{prp:msh-a2}. This then induces the composition of functors
	\[\Phi\colon \cW(D^*S^{n-1}) \to\cW(D^*D^1,\{(0,1)\})\otimes \cW(D^*S^{n-1})\xrightarrow{\sim}\cW(D^*(D^n\setminus\intr(D^n_{1/4})),\Lambda)\to \cW(\Pi_n),\]
	which is given by the composition of functors
	\begin{alignat*}{3}
		\Phi\colon \Tw(\cC_{n-1}) &\to\Tw(\cA_2\otimes \cC_{n-1})&&\xrightarrow{\sim}\Tw(\cB_n^{01})&&\to \Tw(\cD_n^{12})\\
		L&\mapsto \hspace{1cm}(K_1,L)&&\mapsto \hspace{0.5cm}L_1&&\mapsto L_1\\
		z&\mapsto \hspace{1cm} 1_{K_1}\otimes z &&\mapsto \hspace{0.5cm}\alpha_1 &&\mapsto \Phi(z) .
	\end{alignat*}
	To determine $\Phi(z)$, we need to determine the last arrow. It is given by the composition
	\begin{alignat*}{2}
		\Tw(\cB_n^{01})&\to \Tw(\cD_n^{01})&&\xrightarrow{\sim}\Tw(\cD_n^{12})\\
		\alpha_1&\mapsto \hspace{0.5cm}\alpha_1 &&\mapsto yx,
	\end{alignat*}
	where the second arrow is given by Lemma \ref{lem:generator-change}.

	The harder part is determining the inclusion functor $\Psi\colon \cW(T^*S^{n-1})\to\cW(\Pi_n)$. The inclusion of skeleta $\Psi\colon S^{n-1}\to (D^n\times\mathbf{0}_n)\cup(\mathbf{0}_n\times D^n)$ can be decomposed as
	\begin{alignat*}{3}
		\Psi\colon S^{n-1}&\hookrightarrow \hspace{1cm}A_2\times S^{n-1}&&\xrightarrow{\sim}\hspace{0.3cm}(D^n\setminus\intr(D^n_{1/4}))\cup\Cone(\Lambda)&&\rightarrow (D^n\times\mathbf{0}_n)\cup(\mathbf{0}_n\times D^n)\\
		(x_1,\ldots,x_n)&\mapsto ((0,1),(x_1,\ldots,x_n))&&\mapsto \left(\frac{1}{\sqrt{2}} x_1,\ldots,\frac{1}{\sqrt{2}}x_n,x_1,\ldots,x_n\right)&&\mapsto (0,\ldots,0,x_1,\ldots,x_n) .
	\end{alignat*}
	This then induces the composition of functors
	\[\Psi\colon \cW(D^*S^{n-1}) \to\cW(D^*D^1,\{(0,1)\})\otimes \cW(D^*S^{n-1})\xrightarrow{\sim}\cW(D^*(D^n\setminus\intr(D^n_{1/4})),\Lambda)\to \cW(\Pi_n),\]
	which is given by the composition of functors
	\begin{alignat*}{3}
		\Psi\colon \Tw(\cC_{n-1}) &\to\Tw(\cA_2\otimes\cC_{n-1})&&\xrightarrow{\sim}\Tw(\cB_n^{01})&&\to \Tw(\cD_n^{12})\\
		L&\mapsto \hspace{1cm}(K_2,L)&&\mapsto \hspace{0.5cm} L_2&&\mapsto L_2\\
		z&\mapsto \hspace{1cm} 1_{K_2}\otimes z &&\mapsto \hspace{0.5cm} \alpha_2' &&\mapsto \Psi(z),
	\end{alignat*}
	where $K_2=\Cone(f)$, hence $(K_2,L)\simeq \Cone(f\otimes 1_L)$, and $L_2=\Cone(g)$. We need to determine $\alpha_2'$ first.
	For that, recall the quasi-equivalence given in \eqref{eq:functor-m}
	\[M\colon \Tw(\cB_n^{01})\xrightarrow{\sim}\Tw(\cA_2\otimes\cC_{n-1}) .\]
	Note that, using the Koszul sign rule, we have
	\begin{align*}
		1_{K_2}\otimes z &=(i_0\circ p_0+i_1\circ p_1)\otimes (1_L\circ z)\\
		&=(i_0\otimes 1_L)\circ (p_0\otimes z)+ (i_1\otimes 1_L)\circ (p_1\otimes z)\\
		&=(i_0\otimes 1_L)\circ ((1_{K_0}\circ p_0)\otimes (z\circ 1_L))+ (i_1\otimes 1_L)\circ ((1_{K_1}\circ p_1)\otimes (z\circ 1_L))\\
		&= (-1)^n (i_0\otimes 1_L)\circ (1_{K_0}\otimes z)\circ (p_0\otimes 1_L)+ (i_1\otimes 1_L)\circ (1_{K_1}\otimes z)\circ (p_1\otimes 1_L),
	\end{align*}
	where $i_0,i_1,p_0,p_1$ are as in Proposition \ref{prp:extend-cone}.
	Also, we can set $M(i_l)=i_l\otimes 1_L$ and $M(p_l)=p_l\otimes 1_L$ for $l=0,1$. Hence, we must have
	\[\alpha_2'=(-1)^n i_0\alpha_0 p_0  + i_1\alpha_1 p_1 + (-1)^n i_1hp_0,\]
	up to homotopy, since $d\alpha_2'=0$ and
	\begin{gather*}
		M(\alpha_2')=(-1)^n (i_0\otimes 1_L)\circ (1_{K_0}\otimes z)\circ (p_0\otimes 1_L)+ (i_1\otimes 1_L)\circ (1_{K_1}\otimes z)\circ (p_1\otimes 1_L)=1_{K_2}\otimes z .
	\end{gather*}
	Finally, to determine the image $\Psi(z)$ of $\alpha_2'$ under the functor $\Tw(\cB_n^{01})\to\Tw(\cD_n^{12})$, consider its decomposition
	\begin{alignat*}{2}
		\Tw(\cB_n^{01})&\to\Tw(\cD_n^{01})&&\xrightarrow{\sim}\Tw(\cD_n^{12})\\
		\alpha_2'&\mapsto \hspace{0.5cm}\alpha_2&&\mapsto \Psi(z),
	\end{alignat*}
	where $\alpha_2= i_1\alpha_1 p_1 + (-1)^n i_1hp_0$, since the first arrow sends $\alpha_0$ to $0$. By Lemma \ref{lem:generator-change}, the second arrow sends $\alpha_2$ to $xy$, hence $\Psi(z)=xy$.
\end{proof}

In the beginning of the proof, we used a Hamiltonian whose flow on $S^*D^n$ is the reverse Reeb flow to perturb Lagrangians. We could have also used a Hamiltonian whose flow on $S^*D^n$ is the Reeb flow. Then, instead of \eqref{eq:wfuk-plumbing-sector-hocolim}, we would get
\begin{align*}
	\cW(D^*D^n,S^*_0D^n)&\simeq\cW(D^*D^n,\Lambda')\\
	&\simeq\hocolim(\cW(D^*D^1,\{(0,-1)\})\otimes \cW(D^*S^{n-1})\leftarrow \cW(D^*S^{n-1})\rightarrow \cW(D^*D^n_{1/4})),
\end{align*}
where $\Lambda'$ is as $\Lambda$ in Figure \ref{fig:reeb-flow-n=2}, but the blue arrows are reversed, i.e.,
\[\Lambda':=\left.\left\{\left(\frac{1}{\sqrt{2}}x_1,\ldots,\frac{1}{\sqrt{2}}x_n,-x_1,\ldots,-x_n\right)\, \right|\,  x_1^2+\ldots+x_n^2=1\right\}\subset S^*D^n .\]
Note that the stop $\{(0,1)\}$ in $S^*D^1$ is changed as $\{(0,-1)\}$. This change would affect signs appearing in the proof. However, one can check that the end result would not be affected.

Finally, we present the proof of Theorem \ref{thm:wfuk-plumbing-2} by modifying the proof of Theorem \ref{thm:wfuk-plumbing-3}.

\begin{proof}[Proof of Theorem \ref{thm:wfuk-plumbing-2}]
	The proof is similar to the proof of Theorem \ref{thm:wfuk-plumbing-3}, except that we have the pretriangulated equivalence $\cW(D^*S^1)\simeq \cC_1[z^{-1}]$ as described in Proposition \ref{prp:wfuk-sphere}. Then we have the quasi-equivalences
	\[\cA_2\otimes (\cC_1[z^{-1}])\simeq (\cA_2\otimes\cC_1)[\{1_{K_0}\otimes z,1_{K_1}\otimes z\}^{-1}]\simeq\cB_2^{01}[\{\alpha_0,\alpha_1\}^{-1}]\]
	where $\cA_2$ is as in Proposition \ref{prp:msh-a2}, and $\cB_2^{01}$ is the semifree dg category given as follows:
	\begin{enumerate}[label = (\roman*)]
		\item {\em Objects:} $L_0,L_1$.
		\item {\em Generating morphisms:}$\begin{tikzcd}
			L_0\ar[loop left, "\alpha_0"]\ar[r,"g"]\ar[r,"h"', bend right=20]& L_1\ar[loop right,"\alpha_1"]
		\end{tikzcd}$.
		\item {\em Degrees:} $|\alpha_0|=|\alpha_1|=0,\quad |g|=0,\quad |h|=-1$.
		\item {\em Differentials:} $d\alpha_0=d\alpha_1=dg=0,\quad dh=\alpha_1 g - g \alpha_0$.
	\end{enumerate}
	Hence, we have the pretriangulated equivalence
	\[\cW(\Pi_2)\simeq\hocolim(\cB_2^{01}[\{\alpha_0,\alpha_1\}^{-1}]\xleftarrow{F''} \cC_{1}[z^{-1}]\xrightarrow{G'} \cA_1)\]
	where $\cA_1$ is as in Proposition \ref{prp:msh-a2}, and
	\begin{align*}
		F''&\colon \cC_1[z^{-1}]\to\cB_2^{01} \qquad &&G'\colon \cC_1[z^{-1}]\xrightarrow{G'}\cA_1\\
		L&\mapsto L_0, z\mapsto \alpha_0 &&L\mapsto K, z\mapsto 1_K.
	\end{align*}
	Note that $G'$ sends $z$ to $1_K$, not to $0$, which is the main difference from the proof of Theorem \ref{thm:wfuk-plumbing-3}. Then, by \cite{hocolim, Karabas-Lee24} (or Theorem \ref{thm:hocolim-functor-dg}), we have the pretriangulated equivalence
	\[\cW(\Pi_2)\simeq\hocolim(\cB_2^{01}\xleftarrow{F''} \cC_{1}\xrightarrow{G'} \cA_1)[\{\alpha_0,\alpha_1\}^{-1}] .\]
	Note that all the dg categories in the homotopy colimit diagram above are semifree, and the inclusion functor $F''$ is a semifree extension. 
	Therefore, thanks to Remark \ref{rmk colimit}, the homotopy colimit becomes a colimit. Taking the colimit results in setting $\alpha_0=1_{L_0}$ in $\cB_2^{01}$ by Proposition \ref{prp:colimit-dg}, therefore we get the pretriangulated equivalence
	\[\cW(\Pi_2)\simeq\cD_2^{01} [\alpha_1^{-1}],\]
	where $\cD_2^{01}$ is the semifree dg category given as follows:
	\begin{enumerate}[label = (\roman*)]
		\item {\em Objects:} $L_0,L_1$.
		\item {\em Generating morphisms:}
		$\begin{tikzcd}
			L_0\ar[r,"g"]\ar[r,"h"', bend right]& L_1\ar[loop right,"\alpha_1"]
		\end{tikzcd}$.
		\item {\em Degrees:} $|g|=0,\quad |\alpha_1|=0,\quad |h|=-1$.
		\item {\em Differentials:} $dg=d\alpha_1=0,\quad dh=\alpha_1 g-g=\alpha_1' g$ (where we define $\alpha_1':=\alpha_1-1_{L_1}$).
	\end{enumerate}
	Then, by Lemma \ref{lem:generator-change} (note that $\alpha_1'$ here corresponds to $\alpha_1$ in the lemma), we have the pretriangulated equivalence
	\[\cW(\Pi_2)\simeq\cD_2^{01}[\alpha_1^{-1}]=\cD_2^{01}[(1_{L_1}+\alpha_1')^{-1}]\simeq\cD_2^{12}[(1_{L_1}+yx)^{-1}],\]
	since the quasi-isomorphism
	\[\hom^*_{\Tw(\cD_2^{01})}(L_1,L_1)\to \hom^*_{\Tw(\cD_2^{12})}(L_1,L_1)\]
	sends $\alpha_1=1_{L_1}+\alpha_1'$ to $1_{L_1}+yx$ by Lemma \ref{lem:generator-change}. This almost proves the first part of Theorem \ref{thm:wfuk-plumbing-2}, except that we have $\cW(\Pi_2)\simeq\cD_2^{12}[(1_{L_2}+xy)^{-1}]$ there. We will achieve this pretriangulated equivalence at the end of the proof.
	
	As for the inclusion $\Phi$ in the second part, the proof is almost the same as the proof of Theorem \ref{thm:wfuk-plumbing-3}, except that $\Phi\colon \Tw(\cC_1[z^{-1}])\to\Tw(\cD_2^{12})[(1_{L_1}+yx)^{-1}]$ can be decomposed as
	\begin{alignat*}{2}
		\Tw(\cC_1[z^{-1}])&\to \Tw(\cD_2^{01}[\alpha_1^{-1}])&&\xrightarrow{\sim}\Tw(\cD_2^{12}[(1_{L_1}+yx)^{-1}])\\
		z&\mapsto \alpha_1=1_{L_1}+\alpha_1' &&\mapsto 1_{L_1}+yx,
	\end{alignat*}
	where the second arrow is given by Lemma \ref{lem:generator-change} (note that $\alpha_1'$ here corresponds to $\alpha_1$ in the lemma).
	
	Finally, we can determine the inclusion functor $\Psi$ as similar to the proof of Theorem \ref{thm:wfuk-plumbing-3}. The functor $\Psi\colon \Tw(\cC_1[z^{-1}])\to\Tw(\cD_2^{12})[(1_{L_1}+yx)^{-1}]$ can be decomposed as
	\begin{alignat*}{3}
		\Tw(\cC_1[z^{-1}])&\to\Tw(\cB_2^{01}[\{\alpha_0,\alpha_1\}^{-1}])&&\to \Tw(\cD_2^{01}[\alpha_1^{-1}])&&\xrightarrow{\sim}\Tw(\cD_2^{12}[(1_{L_1}+yx)^{-1}])\\
		z&\mapsto \hspace{1.5cm}\alpha_2'&&\mapsto \hspace{1cm}\alpha_2 &&\mapsto \Psi(z),
	\end{alignat*}
	where
	$\alpha_2'=i_0\alpha_0 p_0  + i_1\alpha_1 p_1 + i_1hp_0$.
	We need to determine $\alpha_2$ first. The second arrow sends $\alpha_0$ to $1_{L_0}$, hence
	\[\alpha_2=i_0p_0  + i_1\alpha_1 p_1 + i_1hp_0=i_0p_0  + i_1(1_{L_1}+\alpha_1') p_1 + i_1hp_0=1_{L_2}+i_1\alpha_1'p_1+i_1hp_0 .\]
	The third arrow sends $i_1\alpha'_1 p_1 + i_1hp_0$ to $xy$ by Lemma \ref{lem:generator-change} (recall that $\alpha_1'$ here corresponds to $\alpha_1$ in the lemma), hence $\Psi(z)=1_{L_2}+xy$.
	
	We conclude with the following observation: Since $z$ is invertible up to homotopy in $\cC_1[z^{-1}]$, having $\Psi(z)=1_{L_2}+xy$ shows that $1_{L_2}+xy$ is invertible up to homotopy in $\cD_2^{12}[(1_{L_1}+yx)^{-1}]$, and by the symmetry, $1_{L_1}+yx$ is invertible up to homotopy in $\cD_2^{12}[(1_{L_2}+xy)^{-1}]$, and hence we have the pretriangulated equivalence
	\[\cW(\Pi_2)\simeq\cD_2^{12}[(1_{L_1}+yx)^{-1}]\simeq \cD_2^{12}[(1_{L_2}+xy)^{-1}],\]
	which concludes the proof of the first part of Theorem \ref{thm:wfuk-plumbing-2} (and it also proves Remark \ref{rmk:wfuk-plumb-2-inv}).
\end{proof}

\section{Wrapped Fukaya category of cotangent bundles of punctured $n$-spheres and surfaces}
\label{section wrapped Fukaya categories of Weinstein sectors}
Let us recall our strategy as stated in Section \ref{subsubsection strategy}: To provide a formula for the wrapped Fukaya category of any plumbing space, we will first compute the wrapped Fukaya category of plumbings of cotangent bundles of disks using Equation \eqref{eqn homotopy colimit formula}. To do that, we need, in particular, the wrapped Fukaya category of cotangent bundles of punctured disks, or in other words, punctured spheres, which we will compute in Section \ref{subsection the cotangent bundles of spheres with punctures}.

On the other hand, after establishing a formula for the wrapped Fukaya category of plumbings of cotangent bundles of disks, we will give an explicit computation for the examples mentioned in Section \ref{subsubsection strategy}:
\begin{itemize}
	\item Plumbings of $T^*S^n$ along any quivers, with or without negative intersections, 
	\item Plumbings of cotangent bundles of closed, oriented surfaces along any quivers, with or without negative intersections. 
\end{itemize}
According to Equation \eqref{eqn homotopy colimit formula alternative}, to provide computations for the plumbing spaces above, we will need the wrapped Fukaya category of cotangent bundles of once-punctured $n$-spheres (which are disks, hence trivial) and once-punctured oriented surfaces. In Section \ref{subsection the cotangent bundles of oriented surfaces with punctures}, we will compute the wrapped Fukaya category of cotangent bundles of oriented surfaces with any number of punctures, although once-punctured oriented surfaces are enough for our purposes.

\subsection{The cotangent bundles of spheres with punctures}
\label{subsection the cotangent bundles of spheres with punctures}
The wrapped Fukaya category of $T^*S^n$ is given in Proposition \ref{prp:wfuk-sphere}. 
In Section \ref{subsection the cotangent bundles of spheres with punctures}, we compute wrapped Fukaya categories of cotangent bundles of $n$-spheres with punctures, and describe functors from $\cW(T^*S^{n-1})$ to these wrapped Fukaya categories, which are induced by the inclusions of the boundaries of the neighborhoods of the punctures.

Given integers $n\geq 2$ and $m\geq 1$, we let $S^n_m$ denote the $n$-dimensional sphere with $m$-many punctures.
Or equivalently, 
\[S^n_m := S^n \setminus \{U_i\vb i=1,\ldots,m\},\]
where $\{U_i\}$ is a disjoint collection of small open disks in $S^n$, which we call punctures following our convention in Remark \ref{rmk:puncture}. We describe the inclusion maps of $m$-many connected boundary components of $S^n_m$ into $S^n_m$ as follows:
First, we fix an orientation on $S^{n-1}$ and $S^n$.
For any $i=1,\ldots,m$, we let $\partial U_i$ have the boundary orientation coming from $U_i\subset S^n$.
Then, we have an inclusion map
\begin{equation}\label{eq:punctued-sphere-inclusion}
	F^n_i\colon S^{n-1}\xrightarrow{\sim}\partial U_i\hookrightarrow S^n_m,
\end{equation}
such that the first arrow is an orientation preserving diffeomorphism.

Given these notations, we will compute the wrapped Fukaya category $\cW(T^*S^n_m)$. First, note that by Remark \ref{rmk:grading-pin}, $\cW(T^*S^n_m)$ can be given $\Z$-grading. Also by Remark \ref{rmk:grading-pin}, the definition of $\cW(T^*S^n_m)$ depends on the grading structure $\eta\in H^1(T^*S^n_m;\Z)$ and the background class $b\in H^2(T^*S^n_m;\Z/2)$. They are uniquely determined for $T^*S^n_m$ for $n\geq 4$. For $T^*S^2_m$ and $T^*S^3_m$, we choose the standard grading structure and the background class as in \cite[Section 5.3.1]{nadler-zaslow}. 
Especially, by \cite[Proposition 5.3.1.]{nadler-zaslow}, the bicanonical bundle of a cotangent bundle is canonically trivial.
For the standard grading structure, we are using the canonical trivialization of the bicanonical bundle. 

We will comment on the nonstandard choices in Remark \ref{rmk:punctured-sphere}\eqref{item:punctured-sphere-grading-1} and \ref{rmk:punctured-sphere}\eqref{item:punctured-sphere-grading-2}.

\begin{prp}
	\label{prp sphere with punctures} 
	Fix a pair of integers $\left(n \geq 2, m \geq 1\right)$.
	\begin{enumerate}
		\item\label{item:wfuk-punctured-sphere} Let $S^n_m$ denote the $n$-dimensional sphere with $m$-many punctures $U_1,\ldots,U_m$. Then, up to pretriangulated equivalence, we have
		\[\cW(T^*S^n_m) \simeq
		\begin{cases}
			\cS_{2,m}[\{a_1,\ldots,a_m\}^{-1}] & \text{if }n=2,\\
			\cS_{n,m} & \text{if }n\geq 3,
		\end{cases}\]
		where {\em $\cS_{n,m}$} is a semifree dg category given as follows:
		\begin{enumerate}[label = (\roman*)]
			\item {\em Objects:} $L$ (representing a cotangent fiber).
			\item {\em Generating morphisms:} $a_1, \ldots, a_m,h\in\hom^*(L,L)$.
			\item {\em Degrees:} $|a_i| = 2-n$ for all $i =1, \ldots, m$, and $|h|=1-n$.
			\item {\em Differentials:} $d a_i = 0$ for all $i = 1, \ldots, m$, and
			\[dh=\begin{cases}
				\left(\prod_{i=1}^m a_i \right)- 1_L & \text{if }n=2,\\
				\sum_{i=1}^m a_i & \text{if }n\geq 3 ,
			\end{cases}\]
			where the product is read from right to left, i.e., $\prod_{i=1}^ma_i=a_m\circ\ldots\circ a_1$.
		\end{enumerate}
		
		\item\label{item:punctured-sphere-inclusion}  Let $F^n_i\colon T^*S^{n-1}\xrightarrow{\sim}T^*\partial U_i\to T^*S^n_m$ be the inclusion of Liouville sectors coming from the orientation preserving inclusion $F^n_i\colon S^{n-1}\xrightarrow{\sim}\partial U_i\hookrightarrow S^n_m$ as in \eqref{eq:punctued-sphere-inclusion} for any $i=1,\ldots,m$.
		Note that $\cW(T^*S^{n-1})$ is given as in Proposition \ref{prp:wfuk-sphere}. Then, we have the induced dg functor
		\begin{gather*}
			F^n_i\colon \cW(T^*S^{n-1})\to \cW(T^*S^n_m)\\
			L\mapsto L, z\mapsto a_i.
		\end{gather*}
	\end{enumerate}
\end{prp}

\begin{rmk}\label{rmk:punctured-sphere}\mbox{}
	\begin{enumerate}
		\item We could omit $a_m$ and $h$ in the description of $\cS_{n,m}$ by simplification for any $n\geq 2$. However, keeping them will be beneficial for the computations in this paper.
		
		\item \label{item:punctured-sphere-invertibility} When describing $\cW(T^*S^2_m)$, we do not actually need to invert $a_m$ as its invertibility is implied by the differential of $h$.
		
		\item Proposition \ref{prp sphere with punctures}\eqref{item:punctured-sphere-inclusion} effectively assigns each puncture $U_i$ of $S^n_m$ a generating morphism $a_i$. On the other hand, if $n=2$, the description of $\cW(T^*S^2_m)$ in Proposition \ref{prp sphere with punctures}\eqref{item:wfuk-punctured-sphere} distinguishes between different $a_i$'s as the differential of $h$ includes a term with an ordered product of $a_i$'s. Hence, when $n=2$, we must first fix an ordering of punctures $U_i$ of $S^2_m$ before talking about the inclusion functors $F^2_i$'s in Proposition \ref{prp sphere with punctures}\eqref{item:punctured-sphere-inclusion}.
		
		\item \label{item:sphere-reorder-punctures} In the description of $\mathcal{S}_{n,m}$ in Proposition \ref{prp sphere with punctures}\eqref{item:wfuk-punctured-sphere}, if $n=2$, it may appear that the order of punctures $\left\{ U_1, \dots, U_m \right\}$ affects the resulting category, differently from the case of $n \geq 3$. However, up to quasi-equivalence, the choice of order does not impact the resulting category. Instead, the choice affects the geometric meaning of the morphism $h$.
		
		\item \label{item:punctured-sphere-grading-1} For $T^*S^2_m$, there are $\Z^{m-1}$-many grading structures to define $\cW(T^*S^2_m)$. To capture the nonstandard ones, one needs to let $|a_i|=d_i$ for $d_i\in\Z$ satisfying $d_1+\ldots+d_m=0$.
		\item  \label{item:punctured-sphere-grading-2} For $T^*S^3_m$, there are another background classes to define $\cW(T^*S^3_m)$. To capture the nonstandard ones, one needs to replace some (or all) $da_i=1_L-1_L=0$ with $da_i=1_L+1_L=2\cdot 1_L$ for $i=1,\ldots,m-1$. ($da_m$ is uniquely determined.)
	\end{enumerate}
\end{rmk}

\begin{proof}[Proof of Proposition \ref{prp sphere with punctures}]
	First, we note that we will not use Theorem \ref{thm:gps} in the computation, since the sectorial coverings of $T^*S_m^n$ we will use in the proof contain Liouville sectors with corners. To deal with such coverings, an improvement of Theorem \ref{thm:gps} is needed, see \cite{gps2}. Instead, we will determine $\cW(T^*S^n_m)$ in a different way.
	
	By Theorem \ref{thm:cotangent-generation}, we have a pretriangulated equivalence
	\[\cP(S^n_m)\xrightarrow{\sim}\cW(T^*S^n_m)\]
	sending a point in $S^n_m$ to the cotangent fiber of $T^*S^n_m$ at that point, where $\cP(S^n_m)$ is the Pontryagin category of $S^n_m$ defined in Definition \ref{dfn:pontryagin-category}. By Remark \ref{rmk:loop-space}, $\cP(S^n_m)$ is just a dg category with a single object $p\in S^n_m$ whose morphism space is given by chains on the based loop space of $S^n_m$, i.e., $C_{-*}(\Omega_p S^n_m)$ since $S^n_m$ is path connected.
	
	Moreover, since $F^n_i\colon \cW(T^*S^{n-1})\to \cW(T^*S^n_m)$ sends a cotangent fiber to a cotangent fiber, it can be seen as (up to natural equivalence) a functor
	\[F^n_i\colon\cP(S^{n-1})\to\cP(S^n_m).\]
	
	To determine $\cP(S^n_m)$, without loss of generality, we assume that $S^n$ is the unit sphere in $\mathbb{R}^{n+1}$. 
	Thus, $S^n_m = S^n \setminus \{U_i\vb i=1,\ldots,m\}$ is also a subset of $\R^{n+1}$, where $\{U_i\}$ is a disjoint collection of small open disks in $S^n$.
	We can further assume that the centers of $U_i$ are located in $\mathbb{R}^n \times \{0\} \subset \mathbb{R}^{n+1}$. 
	Then, we define $S^{n,\pm}_m$ as follows:
	\[S^{n,-}_m := S^n_m \cap \left(\mathbb{R}^n \times (-\infty, 0]\right)\quad \text{  and  }\quad S^{n,+}_m := S^n_m \cap \left(\mathbb{R}^n \times [0,\infty)\right).\]
	From the covering $\{S^{n,-}_m, S^{n,+}_m\}$ of $S^n_m$ (which can be seen as an open covering after thickening), we have the following quasi-equivalence by Theorem \ref{thm:seifert-van-kampen}:
	\begin{equation}
		\label{eqn sphere with punctures}
		\cP(S^n_m) \simeq \hocolim\left(\cP(S^{n,-}_m) \xleftarrow{\alpha} \cP(S^{n,-}_m \cap S^{n,+}_m) \xrightarrow{\beta} \cP(S^{n,+}_m)\right).
	\end{equation}
	We note that $\alpha$ and $\beta$ are induced from the inclusions of $S^{n,-}_m \cap S^{n,+}_m$ into $S^{n,\pm}_m$. 
	
	One can easily observe the following facts:
	\begin{itemize}
		\item Since $S^{n,\pm}_m$ is contractible, $\cP(S^{n,\pm}_m)$ is quasi-equivalent to $\cA_1$, where $\cA_1$ is a semifree dg category defined as follows:
		\begin{enumerate}[label = (\roman*)]
			\item {\em Objects:} $K$ (representing a point in $S^{n,\pm}_m$, or a cotangent fiber of $T^*S^{n,\pm}_m$ at that point).
			\item {\em Generating morphisms:} No generating morphisms.
		\end{enumerate}
		\item By the definition of $S^{n,\pm}_m$, if $n \geq 3$, $S^{n,-}_m \cap S^{n,+}_m \subset \mathbb{R}^n \times \{0\} \subset \mathbb{R}^{n+1}$ is an $(n-1)$-dimensional sphere with $m$-many punctures, i.e., $S^{n-1}_m \subset \mathbb{R}^n$.
		Hence, if $n\geq 3$, we have
		\begin{equation}
			\label{eqn sphere with punctures 2}
			\cP(S^n_m) \simeq \hocolim\left(\cP(S^{n,-}_m) \xleftarrow{\alpha} \cP(S^{n-1}_m) \xrightarrow{\beta} \cP(S^{n,+}_m)\right).
		\end{equation}
		If $n =2$, $S^{2,-}_m \cap S^{2,+}_m$ is a disjoint union of $m$-many intervals. 
	\end{itemize}
	
	From these facts, we see that $\cP(S^n_m)$ can be computed by an induction on dimension $n$, which will prove Proposition \ref{prp sphere with punctures}\eqref{item:wfuk-punctured-sphere}.
	Each step of the induction can be proven by applying Theorem \ref{thm:hocolim-functor-dg}. A detailed induction argument will appear after the following comment on Proposition \ref{prp sphere with punctures}\eqref{item:punctured-sphere-inclusion}.
	
	As for proving Proposition \ref{prp sphere with punctures}\eqref{item:punctured-sphere-inclusion}, observe that $i^{\text{th}}$-component of the boundary of $S^n_m$ can be decomposed as $S^{n-1}\simeq\partial U_i=\partial U_i^-\cup \partial U_i^+$, where $\partial U_i^- := \partial U_i \cap S^{n,-}_m$ and $\partial U_i^+:= \partial U_i \cap S^{n,+}_m$.
	Then, we have the homotopy colimit of the morphism of diagrams as follows:
	\begin{equation}\label{eq:punctured-spheres-hocolim-diagram}
		\hocolim\left(
		\begin{tikzcd}
			\cP(\partial U_i^-) \ar[d,"F^{n,-}_i",blue] & \cP(\partial U_i^-\cap \partial U_i^+) \ar[l,"\alpha'"']\ar[r,"\beta'"]\ar[d,"F^{n-1}_i",blue] & \cP(\partial U_i^+) \ar[d,"F^{n,+}_i",blue]\\
			\cP(S^{n,-}_m) & \cP(S^{n,-}_m \cap S^{n,+}_m) \ar[l,"\alpha"']\ar[r,"\beta"] & \cP(S^{n,+}_m)
		\end{tikzcd}\right)
		\simeq
		\begin{tikzcd}
			\cP(S^{n-1})\ar[d,"F^n_i",blue]\\
			\cP(S^n_m)
		\end{tikzcd} .
	\end{equation}
	Note that $\partial U^{\pm}_i$ and $S^{n,\pm}_m$ are contractible, $\partial U_i^-\cap \partial U_i^+= S^{n-2}$, and $S^{n,-}_m \cap S^{n,+}_m=S^{n-1}_m$. The functors $F^{n,\pm}$ are uniquely determined (up to natural equivalence), hence the functors $F^n_i$ can be determined by induction on $n$, which will prove Proposition \ref{prp sphere with punctures}\eqref{item:punctured-sphere-inclusion}.
	Each step of the induction can be proven by applying Theorem \ref{thm:hocolim-functor-dg}.
		
	\noindent {\em The case of $n =2$}:
	This is the base step of the induction. From the definition of $\cP(S^2_m)$, Proposition \ref{prp sphere with punctures} directly follows. Alternatively, one can get $\cP(S^2_m)$ by computing the homotopy colimit in Equation \eqref{eqn sphere with punctures} using Theorem \ref{thm:hocolim-functor-dg} along with the fact that $\cP(S^{2,-}_m \cap S^{2,+}_m)) \simeq \coprod_{i=1}^m \cA_1$. Also, one can determine the functor $F^2_i\colon\cP(S^1)\to\cP(S^2_m)$ by computing the homotopy colimit in Equation \eqref{eq:punctured-spheres-hocolim-diagram} using Theorem \ref{thm:hocolim-functor-dg}.
		
	\noindent {\em The case of $n=3$}:
	From Equation \eqref{eqn sphere with punctures 2} and the base step of the induction, one obtains
	\[\cP(S^3_m) \simeq \hocolim\left(\cA_1 \xleftarrow{\alpha} \cS_{2,m}[\{a_1,\ldots,a_m\}^{-1}] \xrightarrow{\beta} \cA_1\right)\simeq\hocolim\left(\cA_1 \xleftarrow{\alpha} \cS_{2,m} \xrightarrow{\beta} \cA_1\right) ,\]
	where the second equivalence is by Theorem \ref{thm:hocolim-functor-dg} (2). Note that
	\begin{itemize}
		\item $\alpha$ and $\beta$ send each invertible morphism $a_i$ necessarily to an invertible element in $\cA_1$, which is the identity of $\cA_1$ (assuming the standard background class),
		\item $\alpha$ and $\beta$ send $h$ necessarily to $0$ for degree reasons.
	\end{itemize} 
	The last homotopy colimit diagram in the above equation, i.e., $\hocolim\left(\cA_1 \xleftarrow{\alpha} \cS_{2,m} \xrightarrow{\beta} \cA_1\right)$, can be explicitly written by applying Theorem \ref{thm:hocolim-functor-dg} (1). 
	The resulting category is $\cS_{3,m}'[t_L^{-1}]$, where $\cS_{3,m}'$ is a semifree dg category given as follows:
	\begin{enumerate}[label = (\roman*)]
		\item {\em Objects:} $K_1,K_2$.
		\item {\em Generating morphisms:} $t_L,t_{a_1},\ldots,t_{a_m},t_h\in\hom^*(K_1,K_2)$.
		\item {\em Degrees:} $|t_L| = 0$,\quad $|t_{a_i}|=-1$ for all $i =1, \ldots, m$, \quad $|t_h|=-2$.
		\item {\em Differentials:} $d t_L = 0$,\quad $dt_{a_i}=0$ for all $i = 1, \ldots, m$,\quad $dt_h=\sum_{i=1}^mt_{a_i}$.
	\end{enumerate}
	Finally, we have the quasi-equivalence
	\begin{align}
		\label{eq:s3m-quasi-equivalence}\cS_{3,m}&\xrightarrow{\sim}\cS'_{3,m}[t_L^{-1}]\\
		\notag L\mapsto K_1,\quad a_i&\mapsto t_L'\circ t_{a_i}\text{ for $i=1,\ldots,m$},\quad h\mapsto t_L'\circ t_h
	\end{align}
	where $t'_L$ is the inverse of $t_L$ up to homotopy. This proves Proposition \ref{prp sphere with punctures}\eqref{item:wfuk-punctured-sphere} for $n=3$.
	
	As for Proposition \ref{prp sphere with punctures}\eqref{item:punctured-sphere-inclusion}, when $n=3$, Equation \ref{eq:punctured-spheres-hocolim-diagram} becomes
	\[\hocolim\left(\begin{tikzcd}
		\cA_1 \ar[d,"1",blue] & \cC_1[z^{-1}] \ar[l,"\alpha'"']\ar[r,"\beta'"]\ar[d,"F^2_i",blue] & \cA_1\ar[d,"1",blue]\\
		\cA_1 &\cS_{2,m}[\{a_1,\ldots,a_m\}^{-1}] \ar[l,"\alpha"']\ar[r,"\beta"] & \cA_1
	\end{tikzcd}\right)
	\simeq
	\begin{tikzcd}
		\cC_2\ar[d,"F^3_i",blue]\\
		\cS_{3,m}
	\end{tikzcd}
	\]
	where $\cP(S^1)\simeq\cC_1[z^{-1}]$ and $\cP(S^2)\simeq\cC_2$ as in Proposition \ref{prp:wfuk-sphere}. We apply Theorem \ref{thm:hocolim-functor-dg} to compute the image of the diagram under the homotopy colimit functor above. As a result, we get the dg functor
	\begin{align*}
		\cC_2'[t_L^{-1}]&\to \cS'_{3,m}[t_L^{-1}]\\
		K_1\mapsto K_1,\quad K_2&\mapsto K_2,\quad t_L\mapsto t_L,\quad t_x\mapsto t_{a_i},
	\end{align*}
	where $\cC_2'$ is the semifree dg category given as follows:
	\begin{enumerate}[label = (\roman*)]
		\item {\em Objects:} $K_1,K_2$.
		\item {\em Generating morphisms:} $t_{L},t_x\in\hom^*(K_1,K_2)$.
		\item {\em Degrees:} $|t_L|=0,\quad |t_x|=-1$.
		\item {\em Differentials:} $dt_{L}=0,\quad dt_x=t_L-t_L=0$.
	\end{enumerate}
	By reinterpreting the domain and codomain of the functor above using quasi-equivalences, we get
	\begin{alignat*}{3}
		F^3_i\colon\cC_2&\xrightarrow{\sim}\cC_2'[t_L^{-1}]&&\to \cS'_{3,m}[t_L^{-1}]&&\xrightarrow{\sim}\cS_{3,m}\\
		L&\mapsto \hspace{0.3cm}K_1&&\mapsto \hspace{0.5cm}K_1&&\mapsto L\\
		z&\mapsto \hspace{0.1cm}t_L'\circ t_x&&\mapsto \hspace{0.2cm}t_L'\circ t_{a_i}&&\mapsto a_i.
	\end{alignat*}
	We note that the first quasi-equivalence is trivial from Propositions \ref{prp:variable-change} and the second quasi-equivalence is a quasi-inverse of \eqref{eq:s3m-quasi-equivalence}. This proves Proposition \ref{prp sphere with punctures}\eqref{item:punctured-sphere-inclusion} for $n=3$.
	
	\noindent {\em The case of $n \geq 4$}:
	Assuming the induction hypotheses $\cP(S^{n-1}_m)\simeq \cS_{n-1,m}$, $F^{n-1}_i(L)=L$, and $F^{n-1}_i(z)=a_i$, proofs of Proposition \ref{prp sphere with punctures}\eqref{item:wfuk-punctured-sphere} and \ref{prp sphere with punctures}\eqref{item:punctured-sphere-inclusion} for $n\geq 4$ are similar to the case of $n=3$.
\end{proof}

We conclude this subsection with an alternative description of $\cW(T^*S^n_m)$ that is more suitable for describing functors $\cW(T^*S^{n-1})\to\cW(T^*S^n_m)$ induced by the inclusion maps which are not necessarily orientation preserving. More precisely, recall that for any $i=1,\ldots,m$, we have the inclusion map \eqref{eq:punctued-sphere-inclusion}
\begin{equation*}
	F^n_i\colon S^{n-1}\xrightarrow{\sim}\partial U_i\hookrightarrow S^n_m
\end{equation*}
such that the first arrow is an orientation preserving diffeomorphism. We can consider another inclusion map that reverses orientation as follows: Let
\begin{equation}\label{eq:punctued-sphere-inclusion-reverse}
	G^n_i\colon S^{n-1}\xrightarrow{\sim}\partial U_i\hookrightarrow S^n_m
\end{equation}
such that the first arrow is an orientation reversing diffeomorphism. Then we have the following proposition:

\begin{prp}
	\label{prp sphere with punctures general} 
	Fix a pair of integers $\left(n \geq 2, m\geq 1\right)$, and another pair of integers $\left(m_+,m_-\right)$ satisfying $m_+ + m_- =m$.
	\begin{enumerate}
		\item\label{item:wfuk-punctured-sphere-general} Let $S^n_m$ denote the $n$-dimensional sphere with $m$-many punctures $U_1,\dots,U_{m_+},V_1,\ldots,V_{m_-}$. Then, up to pretriangulated equivalence, we have
		\[\cW(T^*S^n_m) \simeq
		\begin{cases}
			\cS_{2,m_+,m_-}[\{a_1,\ldots,a_{m_+},b_1,\ldots,b_{m_-}\}^{-1}] & \text{if }n=2,\\
			\cS_{n,m_+,m_-} & \text{if }n\geq 3,
		\end{cases}\]
		where {\em $\cS_{n,m_+,m_-}$} is a semifree dg category given as follows:
		\begin{enumerate}[label = (\roman*)]
			\item {\em Objects:} $L$ (representing a cotangent fiber).
			\item {\em Generating morphisms:} $a_1, \ldots, a_{m_+},b_1,\ldots,b_{m_-}, h\in\hom^*(L,L)$.
			\item {\em Degrees:} $|a_i| =2-n$ for all $i=1,\ldots,m_+$, $|b_i|= 2-n$ for all $i =1, \ldots, m_-$, and $|h|=1-n$.
			\item {\em Differentials:} $d a_i =0$ for all $i=1,\ldots,m_+$, $db_i=0$ for all $i = 1, \ldots, m_-$, and
			\[dh=\begin{cases}
				\prod_{i=1}^{m_+} a_i - \prod_{i=1}^{m_-} b_i & \text{if }n=2,\\
				\sum_{i=1}^{m_+} a_i - \sum_{i=1}^{m_-} b_i & \text{if }n\geq 3 .
			\end{cases}\]
		\end{enumerate}
		
		\item\label{item:punctured-sphere-inclusion-general} Let $F^n_i\colon T^*S^{n-1}\xrightarrow{\sim}T^*\partial U_i\to T^*S^n_m$ (resp.\ $G^n_i\colon T^*S^{n-1}\xrightarrow{\sim}T^*\partial V_i\to T^*S^n_m$) be the inclusion of Liouville sectors coming from the orientation preserving inclusion $F^n_i\colon S^{n-1}\xrightarrow{\sim}\partial U_i\hookrightarrow S^n_m$  as in \eqref{eq:punctued-sphere-inclusion} (resp.\ orientation reversing inclusion $G^n_i\colon S^{n-1}\xrightarrow{\sim}\partial V_i\hookrightarrow S^n_m$ as in \eqref{eq:punctued-sphere-inclusion-reverse}) for any $i=1,\ldots,m_{+}$ (resp.\ $i=1,\ldots,m_-$). Then, we have the induced dg functors
		\begin{align*}
			F^n_i&\colon \cW(T^*S^{n-1})\to \cW(T^*S^n_m) \qquad &&G^n_i\colon \cW(T^*S^{n-1})\to \cW(T^*S^n_m) \\
			L &\mapsto L, z\mapsto a_i &&L\mapsto L, z \mapsto b_i.
		\end{align*}
	\end{enumerate}
\end{prp}
Note that with this new notation, $\cS_{n,m}$ in Proposition \ref{prp sphere with punctures} is the same as $\cS_{n,m,0}$.

\begin{proof}[Proof of Proposition \ref{prp sphere with punctures general}]
	Recall, as in the proof of Proposition \ref{prp sphere with punctures}, that we have a pretriangulated equivalence $\cP(S^n_m)\xrightarrow{\sim}\cW(T^*S^n_m)$,
	where $\cP(S^n_m)$ is the Pontryagin category of $S^n_m$. Then for $n=2$, Proposition \ref{prp sphere with punctures general} follows from the definition of $\cP(S^2_m)$. Hence, we assume $n\geq 3$ from here onwards.
	
	Consider $S^n_{m_++1}$ (resp.\ $S^n_{m_-+1}$) with punctures $U_1,\dots,U_{m_++1}$ (resp.\ $V_1,\dots,V_{m_-+1}$). The category $\cP(S^n_{m_++1})$ is given as in Proposition \ref{prp sphere with punctures}, with the dg functor $\tilde F^n_i\colon \cP(S^{n-1})\to \cP(S^n_{m_++1})$ sending $z$ to $a_i$ that is induced by an orientation preserving inclusion $\tilde F^n_i\colon S^{n-1}\xrightarrow{\sim}\partial U_i\hookrightarrow S^n_{m_++1}$ for each $i=1,\ldots,m_++1$.
	
	Similarly, $\cP(S^n_{m_-+1})$ is given as in Proposition \ref{prp sphere with punctures} (here, we replace each $a_i$ in the definition of $\cS_{n,m_-+1}$ with $b_i$ for notational purposes), with the dg functor $\tilde G^n_i\colon \cP(S^{n-1})\to \cP(S^n_{m_-+1})$ sending $z$ to $b_i$ that is induced by an orientation preserving inclusion $\tilde G^n_i\colon S^{n-1}\xrightarrow{\sim}\partial V_i\hookrightarrow S^n_{m_-+1}$ for each $i=1,\ldots,m_-+1$.
	
	Next, observe that we can glue $S^n_{m_++1}$ and $S^n_{m_-+1}$ along $S^{n-1}$ using the orientation preserving inclusions $\tilde F^n_{m_++1}$ and $\tilde G^n_{m_-+1}$, i.e., via the identification
	\[S^n_{m_++1}\hookleftarrow \partial U_{m_++1}\xleftarrow{\sim}S^{n-1}\xrightarrow{\sim}\partial V_{m_-+1}\hookrightarrow S^n_{m_-+1} .\]
	As a result of the gluing, we get $S^n_m$ with punctures $U_1,\dots,U_{m_+},V_1,\dots,V_{m_-}$. Note that this gluing, without loss of generality, keeps the orientation of $S^n_{m_++1}\subset S^n_m$ same and changes the orientation of $S^n_{m_-+1}\subset S^n_m$. Hence, the inclusion
	\[F^n_i\colon S^{n-1}\xrightarrow{\sim}\partial U_i\hookrightarrow S^n_{m_++1}\hookrightarrow S^n_m\] 
	induced by $\tilde F^n_i$ is still orientation preserving for each $i=1,\ldots,m_+$, however, the inclusion
	\[G^n_i\colon S^{n-1}\xrightarrow{\sim}\partial V_i\hookrightarrow S^n_{m_-+1}\hookrightarrow S^n_m\]
	induced by $\tilde G^n_i$ is orientation reserving for $i=1,\ldots,m_-$.
	
	Finally, by Theorem \ref{thm:seifert-van-kampen}, we have a quasi-equivalence
	\[\cP(S^n_m)\simeq\hocolim\left(\cP(S^n_{m_++1})\leftarrow\cP(S^{n-1})\rightarrow\cP(S^n_{m_-+1})\right) .\]
	Then applying Theorem \ref{thm:hocolim-functor-dg} proves the proposition (after a simplification of resulting categories). 
\end{proof}

\subsection{The cotangent bundles of oriented surfaces with punctures}
\label{subsection the cotangent bundles of oriented surfaces with punctures}
In Section \ref{subsection the cotangent bundles of oriented surfaces with punctures}, we compute the wrapped Fukaya category of cotangent bundles of oriented surfaces with punctures. 
To do that, we fix a pair of positive integers $g$ and $m$, and let $\Sigma_{g,m}$ be an oriented surface with genus $g$ and with $m$-many punctures. Or equivalently,
\[\Sigma_{g,m}:=\Sigma_g\setminus\{U_i\vb i=1,\ldots,m\},\]
where $\Sigma_g$ is a closed orientable surface of genus $g$, and $\{U_i\}$ is a disjoint collection of small open disks in $\Sigma_g$, which we call punctures following our convention in Remark \ref{rmk:puncture}.

We describe the inclusion maps of $m$-many connected boundary components of $\Sigma_{g,m}$ into $\Sigma_{g,m}$ as follows:
First, we fix an orientation on $S^1$ and $\Sigma_g$.
For any $i=1,\ldots,m$, we let $\partial U_i$ have the boundary orientation coming from $U_i\subset\Sigma_g$.
Then, we have an inclusion map
\begin{equation}\label{eq:punctued-surface-inclusion}
	F_i\colon S^1\xrightarrow{\sim}\partial U_i\hookrightarrow \Sigma_{g,m}
\end{equation}
such that the first arrow is an orientation preserving diffeomorphism.

Given these notations, we will compute the wrapped Fukaya category $\cW(T^*\Sigma_{g,m})$. First, note that by Remark \ref{rmk:grading-pin}, the wrapped Fukaya category $\cW(T^*\Sigma_{g,m})$ can be given $\Z$-grading. Also by Remark \ref{rmk:grading-pin}, the definition of $\cW(T^*\Sigma_{g,m})$ depends on the grading structure $\eta\in H^1(T^*\Sigma_{g,m};\Z)=\Z^{2g+m-1}$ and the background class $b\in H^2(T^*\Sigma_{g,m};\Z/2)=0$. 
We will work with the standard grading structure and the background class (the latter is unique) as in \cite{nadler-zaslow}. We will comment on the nonstandard grading structures in Remark \ref{rmk:punctured-surface-grading}.

\begin{prp}
	\label{prp orientable surface with puctures}
	Fix a pair of positive integers $g$ and $m$. 
	\begin{enumerate}
		\item\label{item:wfuk-punctured-surface} Let $\Sigma_{g,m}$ be an oriented surface with genus $g$ and with $m$ punctures $U_1,\ldots,U_m$. 
		Then, up to pretriangulated equivalence, we have
		\[\cW(T^*\Sigma_{g,m}) \simeq \cM_{g,m}[\{\alpha_j,\beta_j,a_i\vb i=1,\ldots,m;\, j=1,\ldots,g\}^{-1}],\]
		where {\em $\cM_{g,m}$} is a semifree dg category given as follows:
		\begin{enumerate}[label = (\roman*)]
			\item {\em Objects:} $L$ (representing a cotangent fiber).
			\item {\em Generating morphisms:} $\alpha_j, \beta_j, \gamma_j, \delta_j, a_i, h\in\hom^*(L,L)$ for $i = 1, \ldots, m$ and $j=1, \ldots, g$.
			\item {\em Degrees:} For all $i=1, \dots, m$ and for all $j=1, \dots, g$,
			\[|\alpha_j| = |\beta_j|= |\delta_j| = |a_i| =0,\quad |\gamma_j|=|h| = -1.\]
			\item {\em Differentials:} For all $i=1, \dots, m$ and for all $j=1, \dots, g$,
			\[d \alpha_j = d \beta_j = d \delta_j = d a_i =0,\quad d\gamma_j = \alpha_j \beta_j- \beta_j\alpha_j\delta_j ,\quad d h = \prod_{i=1}^m a_i -  \prod_{j=1}^g\delta_j  ,\]  
			where the product is read from right to left, e.g., $\prod_{i=1}^ma_i=a_m\circ\ldots\circ a_1$.
		\end{enumerate}
		\item\label{item:punctured-surface-inclusion}  Let $F_i\colon T^*S^1\xrightarrow{\sim}T^*\partial U_i\to T^*\Sigma_{g,m}$ be the inclusion of Liouville sectors coming from the orientation preserving inclusion $F_i\colon S^1\xrightarrow{\sim}\partial U_i\hookrightarrow \Sigma_{g,m}$ as in \eqref{eq:punctued-surface-inclusion} for any $i=1,\ldots,m$. Note that $\cW(T^*S^1)$ is computed in Proposition \ref{prp:wfuk-sphere}. Then, we have the induced dg functor
		\begin{gather*}
			F_i\colon \cW(T^*S^1)\to \cW(T^*\Sigma_{g,m})\\
			L\mapsto L, z\mapsto a_i.
		\end{gather*}
	\end{enumerate}
\end{prp}

\begin{rmk}\label{rmk:punctured-surface-grading}
	\mbox{}
	\begin{enumerate}
		\item There are $\Z^{2g+m-1}$-many grading structures to define $\cW(T^*\Sigma_{g,m})$. To capture the nonstandard ones, one needs to let $|\alpha_j|=p_j$, $|\beta_j|=q_j$, $|\gamma_j|=p_j+q_j-1$, $|a_i|=r_i$ for $i=1,\ldots,m$, $j=1,\ldots,g$, and $p_j,q_j,r_i\in\Z$ satisfying $r_1+\ldots+r_m=0$.
		\item We would like to thank an anonymous referee for suggesting an alternative way to prove Proposition \ref{prp orientable surface with puctures} \eqref{item:wfuk-punctured-surface}. 
		His suggestion is to use the fact that $\Sigma_{g,m}$ is homotopic to a punctured sphere.
		Then, $\cW(T^*\Sigma_{g,m})$ should be equivalent to the wrapped Fukaya category of cotangent bundle of the punctured sphere, since wrapped Fukaya categories of cotangent bundles are determined by the homotopic class of the zero sections, by Abouzaid \cite{Abouzaid12}.
		Then, Proposition \ref{prp sphere with punctures general} would imply Proposition \ref{prp orientable surface with puctures}. 
	\end{enumerate}
\end{rmk}

\begin{proof}[Proof of Proposition \ref{prp orientable surface with puctures}]
Consider the decomposition $\{X,Y\}$ of $\Sigma_{g,m}$ given in Figure \ref{figure new_decomposition}, where $X$ is a genus $g$ orientable surface with one puncture, i.e., $\Sigma_{g,1}$, and $Y$ is a disk with $m$-many punctures, or equivalently, $S^2_{m+1}$ using the notation in Section \ref{subsection the cotangent bundles of spheres with punctures}.
\begin{figure}[ht]
	\centering
\begingroup%
  \makeatletter%
  \providecommand\color[2][]{%
    \errmessage{(Inkscape) Color is used for the text in Inkscape, but the package 'color.sty' is not loaded}%
    \renewcommand\color[2][]{}%
  }%
  \providecommand\transparent[1]{%
    \errmessage{(Inkscape) Transparency is used (non-zero) for the text in Inkscape, but the package 'transparent.sty' is not loaded}%
    \renewcommand\transparent[1]{}%
  }%
  \providecommand\rotatebox[2]{#2}%
  \newcommand*\fsize{\dimexpr\f@size pt\relax}%
  \newcommand*\lineheight[1]{\fontsize{\fsize}{#1\fsize}\selectfont}%
  \ifx\svgwidth\undefined%
    \setlength{\unitlength}{340.15748031bp}%
    \ifx\svgscale\undefined%
      \relax%
    \else%
      \setlength{\unitlength}{\unitlength * \real{\svgscale}}%
    \fi%
  \else%
    \setlength{\unitlength}{\svgwidth}%
  \fi%
  \global\let\svgwidth\undefined%
  \global\let\svgscale\undefined%
  \makeatother%
  \begin{picture}(1,0.39166667)%
    \lineheight{1}%
    \setlength\tabcolsep{0pt}%
    \put(0,0){\includegraphics[width=\unitlength,page=1]{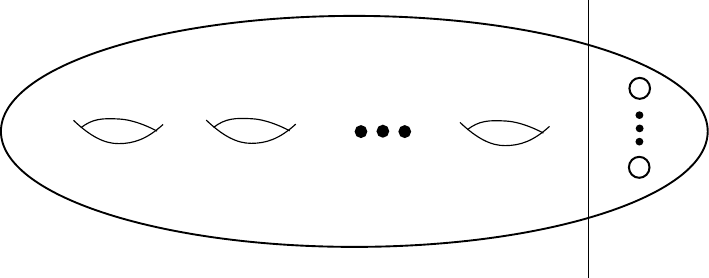}}%
    \put(0.46400531,0.01034012){\color[rgb]{0,0,0}\makebox(0,0)[lt]{\lineheight{1.25}\smash{\begin{tabular}[t]{l}$X=\Sigma_{g,1}$\end{tabular}}}}%
    \put(0.86342417,0.01234012){\color[rgb]{0,0,0}\makebox(0,0)[lt]{\lineheight{1.25}\smash{\begin{tabular}[t]{l}$Y=S^2_{m+1}$\end{tabular}}}}%
  \end{picture}%
\endgroup%
		
	\caption{A decomposition of $\Sigma_{g,m}$ as a union of two subsets $X$ and $Y$.}
	\label{figure new_decomposition}
\end{figure}

The strategy for computing $\cW(T^*\Sigma_{g,m})$ is to glue $\cW(T^*X)$ and $\cW(T^*Y)$.
Since $\cW(T^*Y)=\cW(T^*S^2_{m+1})$ is already computed in Section \ref{subsection the cotangent bundles of spheres with punctures}, our next task is to compute $\cW(T^*X)=\cW(T^*\Sigma_{g,1})$.
For that, we decompose $X=\Sigma_{g,1}$ as a union of subsets $X=X_1\cup X_2\cup X_3\cup X_4$, as described in Figure \ref{figure X_1_and_X_2}. 
\begin{figure}[ht]
	\centering
	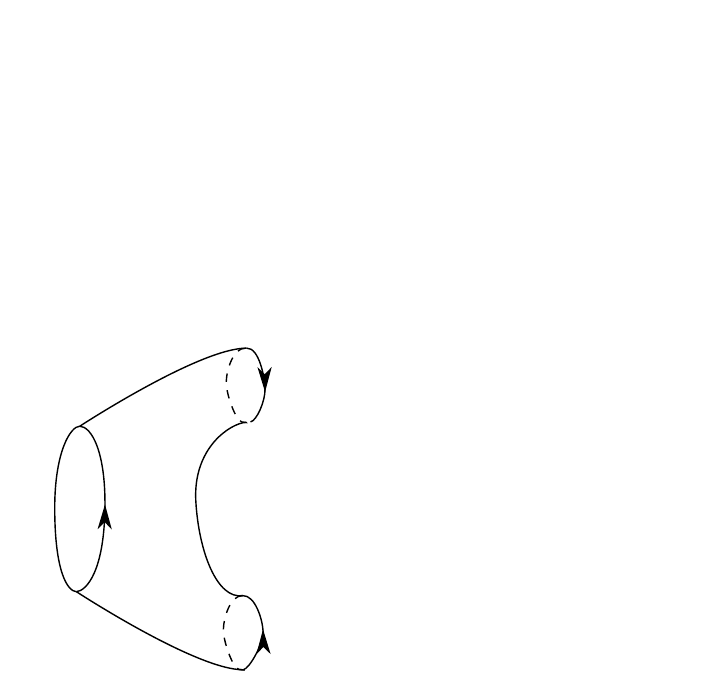		
	\caption{The upper picture is the decomposition of $X=\Sigma_{g,1}$ into four subsets $X_1, \dots, X_4$. The lower picture is describing $j^{\text{th}}$ connected component of $X_1$ and $X_2$.}
	\label{figure X_1_and_X_2}
\end{figure}
Each of subsets $X_1$ and $X_2$ is a disjoint union of $g$-many pairs of pants, or equivalently, $S^2$ with three punctures, i.e., $S^2_3$. 
Moreover, $X_3$ is a disjoint union of $g$-many disks, and $X_4$ is a $S^2$ with $(g+1)$-many punctures, i.e., $S^2_{g+1}$.
We will compute $\cW(T^*X)$ by gluing $\cW(T^*X_i)$. 

The first step is to compute $\cW(T^*X_1)$ and $\cW(T^*X_2)$. 
Since each connected component of $X_1$ (resp.\ $X_2$) is $S^2_3$, we can use Proposition \ref{prp sphere with punctures general} (with $m_+=2$ and $m_-=1$ for computational convenience) to get, up to pretriangulated equivalence
\[\cW(T^*X_i)=\cW\left(T^*\coprod_{j=1}^g S^2_3\right)\simeq \coprod_{j=1}^g\cS_{2,2,1}^{i,j}[\{a_1^{i,j},a_2^{i,j},b_1^{i,j}\}^{-1}]\]
for $i=1,2$, where we relabelled the objects and generating morphisms of $\cS_{2,2.1}$ in Proposition \ref{prp sphere with punctures general} as $L^{i,j}$ and $a_1^{i,j},a_2^{i,j},b_1^{i,j},h^{i,j}$, respectively, to define $\cS_{2,2,1}^{i,j}$.

We note that for each $i=1,2$, via Proposition \ref{prp sphere with punctures general}\eqref{item:punctured-sphere-inclusion-general}, each of the morphisms $a_1^{i,j},a_2^{i,j}$ (resp.\ $b_1^{i,j}$) corresponds to a boundary component of $X_i$ oriented with the opposite (resp.\ same) boundary orientation coming from $X_i$. The correspondence is drawn in Figure \ref{figure X_1_and_X_2}.  

Since $X_1 \cap X_2$ is a disjoint union of $2g$-many $S^1$, Theorem \ref{thm:gps} gives a pretriangulated equivalence
\[\cW\left(T^*(X_1 \cup X_2)\right) \simeq \hocolim\left(\cW(T^*X_1) \xleftarrow{T_1} \coprod_{j=1}^{g} \cW(T^*S^1)\amalg\cW(T^*S^1) \xrightarrow{T_2} \cW(T^*X_2) \right)\]
where $\cW(T^*S^1)$ is described by Proposition \ref{prp:wfuk-sphere}. For each $j=1,\ldots,g$, denote the generating objects (resp.\ generating morphisms) of $\cW(T^*S^1)\amalg\cW(T^*S^1)$ by $L_{1,j}$ and $L_{2,j}$ (resp.\ $z_{1,j}$ and $z_{2,j}$). From Figure \ref{figure X_1_and_X_2}, it is easy to observe that $T_1$ (resp.\ $T_2$) sends
\[L_{1,j},L_{2,j}\mapsto L^{1,j}, z_{1,j}\mapsto a_2^{1,j}, z_{2,j}\mapsto b_1^{1,j} (\text{resp.\  } L_{1,j},L_{2,j}\mapsto L^{2,j}, z_{1,j}\mapsto b_1^{2,j}, z_{2,j}\mapsto a_2^{2,j}).\]
Then, applying Theorem \ref{thm:hocolim-functor-dg} identifies (after simplification) $L^{1,j}=L^{2,j}$ (denote it by $L_j$) and $a_2^{1,j}=b_1^{2,j}$, and introduces new generating morphism $\beta_j,\gamma_j\colon L_j\to L_j$ such that $\beta_j$ is invertible up to homotopy, $|\gamma_j|=-1$, and $d\gamma_j= a_2^{2,j}\beta_j - \beta_j b_1^{1,j}$. Note that we also have the following equations:
\begin{gather*}
	dh^{1,j}=a_2^{1,j}a_1^{1,j}- b_1^{1,j} ,
	dh^{2,j}=a_2^{2,j}a_1^{2,j} - b_1^{2,j} .
\end{gather*}
After relabelling
\[\alpha_j:=a_2^{2,j},\quad \delta_j:=a_1^{1,j},\quad \epsilon_j:=a_1^{2,j} ,\]
and cancelling $h^{1,j}$ with $b_1^{1,j}$ and $h^{2,j}$ with $b_1^{2,j}$ using Propositions \ref{prp:variable-change} and \ref{prp:destabilisation}, we see that
\[\cW\left(T^*(X_1 \cup X_2)\right) \simeq \coprod_{j=1}^g \mathcal{G}_j[\{\alpha_j,\beta_j,\delta_j,\epsilon_j\}^{-1}]\]
up to pretriangulated equivalence, where $\mathcal{G}_j$ is a semifree dg category defined as follows:
\begin{enumerate}[label = (\roman*)]
	\item {\em Objects:} $L_j$ (representing a cotangent fiber).
	\item {\em Generating morphisms:} $\alpha_j, \beta_j, \gamma_j,\delta_j,\epsilon_j\in\hom^*(L_j,L_j)$.
	\item {\em Degrees:} $|\alpha_j|=|\beta_j|=|\delta_j|=|\epsilon_j|=0,\quad |\gamma_j|=-1$.
	\item {\em Differentials:} $d\alpha_j = d\beta_j =d\delta_j=d\epsilon_j=0,\quad d\gamma_j = \alpha_j \beta_j- \beta_j\alpha_j \epsilon_j\delta_j$. 
\end{enumerate}

The next step is to glue $\cW\left(T^*(X_1 \cup X_2)\right)$ and $\cW(T^*X_3)$. 
Since $X_3$ is a disjoint union of disks, it is easy to glue those two by Theorem \ref{thm:hocolim-functor-dg}.
The effect of the gluing is to identify $\epsilon_j$ and the identity morphism for each $j=1,\ldots,g$. Note that as a result, the invertibility (up to homotopy) of $\alpha_j$ and $\beta_j$ implies the invertibility of $\delta_j$ via the differential relation for $\gamma_j$. Hence, we no longer need to invert $\delta_j$.
Thus, we have
\[\cW\left(T^*(X_1 \cup X_2 \cup X_3)\right) \simeq \coprod_{j=1}^g \mathcal{H}_j[\{\alpha_j,\beta_j\}^{-1}]\]
up to pretriangulated equivalence, where $\mathcal{H}_j$ is a semifree dg category defined as follows:
\begin{enumerate}[label = (\roman*)]
	\item {\em Objects:} $L_j$ (representing a cotangent fiber).
	\item {\em Generating morphisms:} $\alpha_j, \beta_j, \gamma_j,\delta_j\in\hom^*(L_j,L_j)$.
	\item {\em Degrees:} $|\alpha_j|=|\beta_j|=|\delta_j|=0,\quad |\gamma_j|=-1$.
	\item {\em Differentials:} $d\alpha_j = d\beta_j =d\delta_j=0,\quad d\gamma_j = \alpha_j \beta_j- \beta_j\alpha_j\delta_j$. 
\end{enumerate}

We note that $X_1 \cup X_2 \cup X_3$ has $g$-many connected components, and each of them is a punctured torus.
Moreover, each of the morphisms $\alpha_j,\beta_j,\delta_j$ corresponds to an oriented loop in the $j^{th}$ component of $X_1 \cup X_2 \cup X_3$ as shown in Figure \ref{figure X_1 cup X_2 cup X_3}.

\begin{figure}[ht]
	\centering
	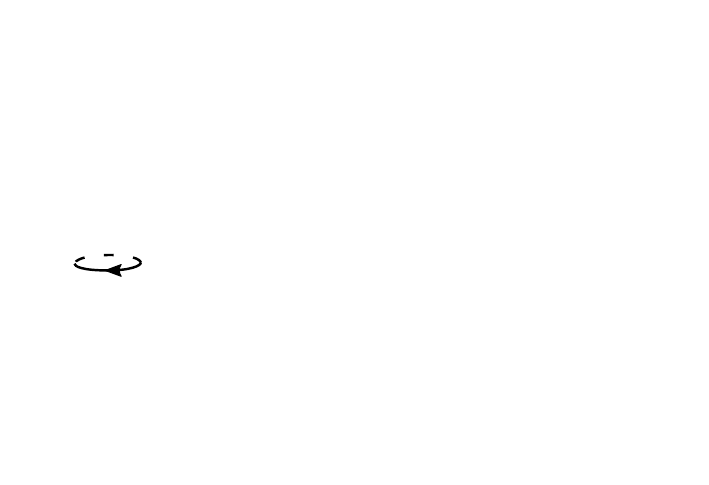		
	\caption{Each of the punctured tori in the upper part of the picture is a connected component of $X_1 \cup X_2 \cup X_3$. The oriented loops in the picture are labeled by the corresponding generating morphisms in $\cH_j$. The lower part corresponds to $X_4$ and the loops correspond to the generating morphisms in $\cS_{2,1,g}$.}
	\label{figure X_1 cup X_2 cup X_3}
\end{figure}

Next, recall that $X_4$ is a sphere with $g+1$ punctures, i.e., $S^2_{g+1}$. By Proposition \ref{prp sphere with punctures general} (with $m_+=1$ and $m_-=g$), its wrapped Fukaya category is given by
\[\cW(S^2_{g+1})\simeq\cS_{2,1,g}[\{a_1,b_1,\ldots,b_g\}^{-1}]\]
up to pretriangulated equivalence.
Moreover, by Proposition \ref{prp sphere with punctures general}\eqref{item:punctured-sphere-inclusion-general}, the morphism $a_1$ (resp.\ each of the morphisms $b_1,\ldots,b_g$) corresponds to a boundary component of $X_4$ oriented with the opposite (resp.\ same) boundary orientation coming from $X_4$. The correspondence is drawn in Figure \ref{figure X_1 cup X_2 cup X_3}.

Since $X_1\cup X_2\cup X_3$ and $X_4$ intersect at a disjoint union of $g$-many $S^1$'s, by Theorem \ref{thm:gps}, we have a pretriangulated equivalence
\[\cW(T^*X)=\cW(T^*\Sigma_{g,1}) \simeq \hocolim\left(\cW(T^*(X_1\cup X_2\cup X_3)) \xleftarrow{Q_1} \coprod_{j=1}^{g} \cW(T^*S^1)\xrightarrow{Q_2} \cW(T^*X_4) \right)\]
where $\cW(T^*S^1)$ is described by Proposition \ref{prp:wfuk-sphere}. For each $j=1,\ldots,g$, denote the generating object (resp.\ generating morphism) of $\cW(T^*S^1)$ by $L_j$ (resp.\ $z_j$). From Figure \ref{figure X_1 cup X_2 cup X_3}, it is easy to observe that $Q_1$ (resp.\ $Q_2$) sends
\[L_j\mapsto L_j, z_j\mapsto \delta_j  (\text{resp.\  } L_j\mapsto L, z_j\mapsto b_j).\]
Then, Theorem \ref{thm:hocolim-functor-dg} identifies (after simplification) $L=L_j$ and $\delta_j=b_j$ for any $j=1,\ldots,g$, and we conclude the proof of Proposition \ref{prp orientable surface with puctures}\eqref{item:wfuk-punctured-surface} for $m=1$. Since the boundary component of $\Sigma_{g,1}$ is the same as the boundary component of $X_4=S^2_{g+1}\subset \Sigma_{g,1}$ corresponding to the morphism $a_1$, Proposition \ref{prp orientable surface with puctures}\eqref{item:punctured-surface-inclusion} for $m=1$ follows from Proposition \ref{prp sphere with punctures general}\eqref{item:punctured-sphere-inclusion-general} and Theorem \ref{thm:hocolim-functor-dg}.

To prove Proposition \ref{prp orientable surface with puctures} for an arbitrary $m$, we need to glue $\cW(T^*X)=\cW(T^*\Sigma_{g,1})$ and $\cW(T^*Y)=\cW(T^*S^2_{m+1})$ as we remarked before. This is straightforward using Theorem \ref{thm:gps} and Theorem \ref{thm:hocolim-functor-dg}. Alternatively, one can replace $X_4=S^2_{g+1}$ above with $S^2_{g+m}$, and the wrapped Fukaya category $\cW(T^*X_4)\simeq \cS_{2,1,g}[\{a_1,b_1,\ldots,b_g\}^{-1}]$ with $\cS_{2,m,g}[\{a_1,\ldots,a_m,b_1,\ldots,b_g\}^{-1}]$ to directly prove Proposition \ref{prp orientable surface with puctures} for an arbitrary $m$ by gluing $\cW\left(T^*(X_1 \cup X_2\cup X_3)\right)$ and $\cW(T^*X_4)$.
\end{proof}

\begin{rmk}
	\label{rmk simplified version of M_g,m with relations}
	One can simplify $\cM_{g,m}[\{\alpha_j,\beta_j,a_i\}^{-1}]$ in Proposition \ref{prp orientable surface with puctures} by setting $\gamma_j=h=0$ and identifying $\delta_j$ with $[\alpha_j,\beta_j]$ for $j=1,\ldots,g$, where $[\alpha_j,\beta_j] := \alpha_j^{-1} \beta_j^{-1} \alpha_j \beta_j$, and $\alpha_j^{-1},\beta_j^{-1}$ are inverses of $\alpha_j,\beta_j$, respectively. As a result, we get generating morphisms $\{\alpha_j, \beta_j, a_i\}$ satisfying 
	\[d \alpha_j = d \beta_j = da_i = 0,\quad \prod_{i=1}^m a_i = \prod_{j=1}^g [\alpha_j,\beta_j] .\]
	Note that the result is not a semifree dg category anymore.
\end{rmk}

\section{Wrapped Fukaya category of plumbings}
\label{section wrapped Fukaya categories of plumbings}
The goal of Section \ref{section wrapped Fukaya categories of plumbings} is to provide a formula for the wrapped Fukaya category of plumbings of cotangent bundles. To achieve this goal, we will follow the strategy outlined in Section \ref{subsubsection strategy}.

In Section \ref{subsection plumbings disks}, we present a formula for the wrapped Fukaya category of plumbings of cotangent bundles. This formula is applicable to cotangent bundles of any collection of connected oriented $n$-manifolds $\{M(v)\vb v\in V(Q)\}$ for $n\geq 2$ along any quiver $Q$ with or without negative intersections. The formula depends on having the dg algebra $C_{-*}(\Omega_p(M(v)\setminus\text{pt}))$ available at each $v\in V(Q)$. In Section \ref{subsection plumbing grading}, we discuss the choices of grading structures when defining the wrapped Fukaya category of plumbings. Then, we present a general formula that considers different grading structures.

As an application, we explicitly present:
\begin{itemize}
	\item the wrapped Fukaya categories of $T^*S^n$'s for $n\geq 3$ as dg categories and compare them to Ginzburg dg categories in Section \ref{subsection plumbings in (a)},
	
	\item the wrapped Fukaya categories of cotangent bundles of oriented closed surfaces as dg categories and compare them to derived multiplicative preprojective algebras in Section \ref{subsection plumbings in (b)}.
\end{itemize}

\subsection{A formula for general plumbing spaces}
\label{subsection plumbings disks}
In this subsection, we give a formula for the wrapped Fukaya category $\cW(P)$ of any plumbing space $P$ in Theorem \ref{thm:wfuk-plumbing-general}.

Before we start, we note that for any plumbing space $P=P(Q,M,\s)$, as per Remark \ref{rmk:grading-pin}, $\cW(P)$ can be endowed with a $\Z$-grading. Furthermore, according to Remark \ref{rmk:grading-pin}, the definition of $\cW(P)$ relies on the grading structure $\eta\in H^1(P;\Z)$ and the background class $b\in H^2(P;\Z/2)$. In this paper, we select the standard background class
\[b\in H^2(P;\Z/2)=\bigoplus_{v\in V(Q)} H^2(M(v);\Z/2) ,\]
where $b$ restricts to the standard background class of $T^*M(v)$ in $H^2(M(v);\Z/2)$ (as explained in \cite{nadler-zaslow}) for each $v\in V(Q)$. Regarding the grading structure, we choose
\[\eta\in H^1(P;\Z)=H^1(Q;\Z)\oplus\bigoplus_{v\in V(Q)} H^1(M(v);\Z)\]
in a manner that $\eta$ restricts to the standard grading structure of $T^*M(v)$ in $H^1(M(v);\Z)$ (as explained in \cite{nadler-zaslow}[Proposition 5.3.1] and also in \cite{Abouzaid12}[Theorem 1.1]) for each $v\in V(Q)$. Additionally, $\eta$ restricts to $0\in H^1(Q;\Z)$, a point that will be elaborated in Section \ref{subsection plumbing grading}. Here, we just note that the choice $0\in H^1(Q;\Z)$ can be understood through our selection of the functors $\Phi$ and $\Psi$ in the homotopy colimit diagram \eqref{eqn homotopy colimit formula} as specified in Theorem \ref{thm:wfuk-plumbing-3} and Theorem \ref{thm:wfuk-plumbing-2}. This choice contrasts with the alternatives outlined in Remark \ref{rmk:shifting-generators} and \ref{rmk:shifting-generators-2}.

For a more in-depth discussion on grading structures, see Section \ref{subsection plumbing grading}. We refer the specific grading structure $\eta$ mentioned above as {\em the standard grading structure on $P$}, and $\cW(P)$ denotes the wrapped Fukaya category of $P$ with this standard grading structure $\eta$. We will consider the wrapped Fukaya category of $P$ with other grading structures in Section \ref{subsection plumbing grading}.

Through this subsection, we fix an integer $n \geq 2$. 
Moreover, we assume that there is a given plumbing data $(Q,M,\s)$ such that $Q$ is any quiver, $M(v)$ is a connected oriented $n$-manifold (with or without boundary) for all $v \in V(Q)$, and $\s$ is any map $E(Q)\to\{\pm 1\}$. Let $P(Q,M,\s)$ be plumbing space corresponding to the plumbing data $(Q,M,\s)$. See Section \ref{section plumbing space} for more details.

We will follow the strategy outlined in Section \ref{subsubsection strategy} to compute the wrapped Fukaya category of $P(Q,M,\s)$. Hence, we should consider another plumbing data $(Q,M',\s)$ such that $Q$ and $\s$ are as above, and $M'(v)$ is the closed disk $\D^n$ for all $v \in V(Q)$.
We will first compute the wrapped Fukaya category of $P(Q,M',\s)$ in Lemma \ref{lem:wfuk-plumbing-of-disks}.

Before stating Lemma \ref{lem:wfuk-plumbing-of-disks}, let us discuss our strategy for computing $\cW(P(Q,M',\s))$. As remarked in Section \ref{subsubsection strategy}, we can compute $\cW(P(Q,M',\s))$ using the homotopy colimit diagram
\begin{equation}\label{eq:hocolim-plumbing-disks}
	\cW(P(Q,M',\s)) \simeq \textup{hocolim}\left(
	\begin{tikzcd}[column sep=0em]
		\coprod_{v \in V(Q)} \cW(T^*M'_v) & & \coprod_{e \in E(Q)} \cW(\Pi_n)\\
		& \hspace*{-5em} \coprod_{e \in E(Q)} \left(\cW(T^*S^{n-1}) \amalg \cW(T^*S^{n-1})\right) \hspace*{-5em} \ar[lu,"\coprod_{e \in E(Q)} \left(F_e \amalg G_e\right)"]\ar[ru,"\coprod_{e \in E(Q)} \left( \Phi_e \amalg \Psi_e \right)"']
	\end{tikzcd}\right)
\end{equation}
up the pretriangulated equivalence. The wrapped Fukaya category $\cW(T^*S^{n-1})$ is described in Proposition \ref{prp:wfuk-sphere}. The wrapped Fukaya category $\cW(\Pi_n)$ of the plumbing sector $\Pi_n$ and the functors $\Phi_e$ and $\Psi_e$ are described in Theorem \ref{thm:wfuk-plumbing-2} (for $n=2$) and Theorem \ref{thm:wfuk-plumbing-3} (for $n\geq 3$). We need to explain $\cW(T^*M'_v)$ for each $v\in V(Q)$ and the functors $F_e$ and $G_e$ for each $e\in E(Q)$.

For each $v \in V(Q)$, recall from Definition \ref{dfn M_v} that
\[M'_v =\D^n\setminus \left(\bigcup_{e= v \to \bullet} U_{p(e)} \cup \bigcup_{e =\bullet \to v} U_{q(e)}\right)\]
where $p(e)$ (resp.\ $q(e)$) is a point in $\D^n$ for each edge $e$ starting (resp.\ ending) at $v$, and $U_{p(e)}$ (resp.\ $U_{q(e)}$) is a sufficiently small open disk in $\D^n$ centered at $p(e)$ (resp.\ $q(e)$) such that 
\[\left\{ U_{p(e)}\vb e=v\to\bullet\right\}\cup \left\{ U_{q(e)} \vb e = \bullet \to v \right\}\] is a mutually disjoint collection of open disks in $\D^n$. Hence, $M'_v$ can be equivalently seen as an $n$-sphere with punctures
\[M'_v =S^n\setminus \left(U_{r(v)}\cup\bigcup_{e= v \to \bullet} U_{p(e)} \cup \bigcup_{e =\bullet \to v} U_{q(e)}\right)\]
where $r(v)$ is a point in $S^n$ with a small open neighborhood $U_{r(v)}$ such that $S^n\setminus U_{r(v)}\simeq\D^n=M'(v)$. We call $U_{r(v)},U_{p(e)}, U_{q(e)}$ punctures of $M'_v$ following the convention in Remark \ref{rmk:puncture}. Then, the number of punctures in the punctured sphere $M'_v$ is $|v|+1$, where $|v|$ denotes the valency of the vertex $v$.

Fix an orientation on $S^{n-1}$ and $S^n$. Give the subspace orientation to $U_{r(v)},U_{p(e)}, U_{q(e)}\subset S^n$. Then, give the boundary orientation to $\partial U_{r(v)},\partial U_{p(e)}, \partial U_{q(e)}$ by considering them as boundaries of $U_{r(v)},U_{p(e)}, U_{q(e)}$, respectively. Recall from Section \ref{subsubsection gluing information} that for each $e\in E(Q)$,
\begin{itemize}
	\item $F_e\colon \cW(T^*S^{n-1})\to\cW(T^*M'_v)$ is induced by the map $F_e\colon S^{n-1}\xrightarrow{\sim}\partial U_{p(e)}\hookrightarrow M'_v$, where the first arrow is an orientation preserving diffeomorphism,
	\item $G_e\colon \cW(T^*S^{n-1})\to\cW(T^*M'_v)$ is induced by the map $G_e\colon S^{n-1}\xrightarrow{\sim}\partial U_{q(e)}\hookrightarrow M'_v$, where the first arrow is 
	\[\begin{cases}
		\text{  an orientation preserving diffeomorphism, if  } (-1)^{*_e}=1,\\
		\text{  an orientation reversing diffeomorphism, if  } (-1)^{*_e}=-1,
	\end{cases}\]
	where $(-1)^{*_e}=(-1)^{n(n-1)/2}\s(e)$.
\end{itemize}
We also have the map $\mu_v\colon S^{n-1}\xrightarrow{\sim}\partial U_{r(v)}\hookrightarrow M'_v$, where the first arrow is an orientation preserving diffeomorphism, which can be also interpreted as the inclusion of the boundary
\begin{equation}\label{eq:inclusion-disk-boundary}
	\mu_v\colon S^{n-1}\xrightarrow{\sim}\partial \D^n\hookrightarrow \D^n=M'(v) ,
\end{equation}
where $\partial \D^n$ has the boundary orientation coming from $\D^n$, and the first arrow is an orientation reversing diffeomorphism. It induces an inclusion of Liouville sectors
\[\mu_v\colon T^*S^{n-1}\hookrightarrow T^*\D^n=T^*M'(v)\hookrightarrow P(Q,M',\s) ,\]
which induces a functor
\[\mu_v\colon\cW(T^*S^{n-1})\to\cW(P(Q,M',\s)) .\]

Motivated by the conventions above, let us call a puncture $U$ of $M'_v$ \textit{positively (resp.\ negatively) oriented} if the relevant diffeomorphism $S^{n-1}\xrightarrow{\sim} \partial U$ above is orientation preserving (resp.\ reversing). In particular, for any $e \in E(Q)$ and $v \in V(Q)$, the punctures $U_{p(e)}$ and $U_{r(v)}$ are positively oriented.

We can use Proposition \ref{prp sphere with punctures general} to describe $\cW(T^*M'_v)$. To make the description compatible with the functors $F_e$ and $G_e$, let us first define $m_+$ (resp.\ $m_-$) to be the number of positively (resp.\ negatively) oriented punctures of $M'_v$. Explicitly, they are given as follows:
\begin{itemize}
	\item If $\tfrac{1}{2}(n-1)n$ is an even integer,
	\begin{align*}
		m_+&:= |\{e=v\to \bullet \vb e\in E(Q)\}| + |\{e=\bullet\to v\vb e \in E(Q),\, \s(e)=1\}| +1,\\
		m_-&:= |\{e=\bullet\to v\vb e\in E(Q),\, \s(e)=-1\}| .
	\end{align*}
	\item If $\tfrac{1}{2}(n-1)n$ is an odd integer,
	\begin{align*}
		m_+&:= |\{e=v\to \bullet \vb e\in E(Q)\}| + |\{e=\bullet\to v\vb e \in E(Q),\, \s(e)=-1\}| + 1 ,\\
		m_-&:= |\{e=\bullet\to v\vb e\in E(Q),\, \s(e)=1\}| .
	\end{align*}
\end{itemize}
Then, Proposition \ref{prp sphere with punctures general} with the above numbers $m_+$ and $m_-$ gives a pretriangulated equivalence
\begin{equation}\label{eq:wfuk-m-v}
	\cW(T^*M'_v)\simeq \begin{cases}
		\cS^v_{2,m_+,m_-}[\{a_1^v,\ldots,a_{m_+}^v,b_1^v,\ldots,b_{m_-}^v\}^{-1}] & \text{if }n=2,\\
		\cS^v_{n,m_+,m_-} & \text{if }n\geq 3 ,
	\end{cases}
\end{equation}
where for each $v\in V(Q)$, we define $\cS^v_{n,m_+,m_-}$ by relabeling the unique object $L$ and the generating morphisms $a_1,\ldots, a_{m_+}$, $b_1,\ldots,b_{m_-}$, $h$ of $\cS_{n,m_+,m_-}$ given in Proposition \ref{prp sphere with punctures general} by $L_v$ and $a_1^v,\ldots,a_{m_+}^v$, $b_1^v,\ldots,b_{m_-}^v$, $h_v$.

Now, give an ordering to, i.e., relabel the positively oriented punctures of $M'_v$ by $U_1,\ldots, U_{m_+}$ and the negatively oriented punctures of $M'_v$ by $V_1,\dots,V_{m_-}$ in such a way that
\begin{itemize}
	\item given an edge $e$ starting at $v$, $U_{p(e)}=U_i$ for some $i\in\{1,\ldots,m_+'\}$,\\
	
	\item given an edge $e$ ending at $v$,
	$U_{q(e)}=\begin{cases}
		U_j\text{ for some $j\in\{m_+'+1,\ldots,m_+-1\}$} & \text{if $(-1)^{*_e}=1$}\\
		V_l\text{ for some $l\in\{1,\ldots,m_-\}$} & \text{if $(-1)^{*_e}=-1$}
	\end{cases}$,\\
	\item $U_{r(v)}=U_{m_+}$,
\end{itemize}
where $m_+':= |\{e=v\to \bullet \vb e\in E(Q)\}|$, and $(-1)^{*_e}=(-1)^{n(n-1)/2}\s(e)$.

Then, by Proposition \ref{prp sphere with punctures general}\eqref{item:punctured-sphere-inclusion-general}, given $e=v\to\bullet$, the map $F_e\colon S^{n-1}\xrightarrow{\sim}\partial U_{p(e)}\hookrightarrow M'_v$ induces the dg functor
\begin{gather}\label{eq:m-v-inclusion-a}
	F_e\colon\cW(T^*S^{n-1})\to\cW(T^*M'_v)\\
	L\mapsto L_v,\quad z\mapsto a_i^v ,\notag
\end{gather}
and given $e=\bullet\to v$, the map $G_e\colon S^{n-1}\xrightarrow{\sim}\partial U_{q(e)}\hookrightarrow M'_v$ induces the dg functor
\begin{gather}\label{eq:m-v-inclusion-b}
	G_e\colon\cW(T^*S^{n-1})\to\cW(T^*M'_v)\\
	\notag L\mapsto L_v ,\quad
	z\mapsto\begin{cases}
		a_j^v & \text{if $(-1)^{*_e}=1$,}\\
		b_l^v & \text{if $(-1)^{*_e}=-1$,}
	\end{cases}
\end{gather}
where $L$ (resp.\ $z$) is the generating object (resp.\ generating morphism) of $\cW(T^*S^{n-1})$ as described in Proposition \ref{prp:wfuk-sphere}.

Having all this information, we are ready to present the following lemma describing the wrapped Fukaya category of plumbings of cotangent bundles of disks:

\begin{lem}
	\label{lem:wfuk-plumbing-of-disks}
	Fix $n\geq 2$. Let $P'$ be the plumbing of $T^*\D^n$'s along any quiver $Q$ with or without negative intersections. 
	In other words, there is a plumbing data $(Q,M',\s)$ such that $M'(v) = \D^n$ for all $v \in V(Q)$, so that $P' = P(Q,M',\s)$. 
	\begin{enumerate}
		\item\label{item:wfuk-plumb-disks} The wrapped Fukaya category of $P'$ is, up to pretriangulated equivalence, given by
		\[\cW(P')\simeq\begin{cases}
			\cP'_2[\{1+y_ex_e\vb e\in E(Q)\}^{-1}] & \text{if $n=2$,}\\
			\cP'_n & \text{if $n\geq 3$,}
		\end{cases}\]
		where $\cP'_n$ is a semifree dg category given as follows:
		\begin{enumerate}[label = (\roman*)]
			\item {\em Objects:} $L_v$ (representing a Lagrangian cocore dual to $M'(v)$) for any $v \in V(Q)$.
			\item {\em Generating morphisms:}  
			\[\left\{m_v\colon L_v\to L_v, h_v\colon L_v\to L_v, x_e\colon L_v \to L_w, y_e\colon L_w \to L_v | v \in V(Q), e = v \to w \in E(Q)\right\}.\]
			\item {\em Degrees:} $|m_v|=2-n,\quad |h_v|=1-n,\quad |x_e|=0,\quad |y_e|=2-n$. 
			\item {\em Differentials:} $dm_v=d x_e = d y_e= 0$, and
			\[dh_v=\begin{cases}
				\displaystyle m_v\prod_{\substack{e= \bullet \to v\\ \s(e)=-1}} (1+x_ey_e)\prod_{e= v\to \bullet} (1+y_ex_e) - \prod_{\substack{e= \bullet \to v\\ \s(e)=1}}(1+x_ey_e) &\text{if $n=2$,}\\
				\displaystyle m_v+\sum_{e= v\to \bullet} y_ex_e + \sum_{e= \bullet \to v}(-1)^{n(n-1)/2}\s(e)\, x_ey_e &\text{if $n\geq 3$.}
			\end{cases}\]
		\end{enumerate}
		
		\item\label{item:inclusion-plumb-disks} When $\cW(T^*S^{n-1})$ is as in Proposition \ref{prp:wfuk-sphere}, for each $v\in V(Q)$, the inclusion of the boundary $\mu_v\colon S^{n-1}\xrightarrow{\sim}\partial \D^n\hookrightarrow \D^n=M'(v)$ in \eqref{eq:inclusion-disk-boundary} induces the dg functor
		\begin{gather*}
			\mu_v\colon\cW(T^*S^{n-1}) \to\cW(P')\\
			L\mapsto L_v, 
			z\mapsto m_v.
		\end{gather*}
	\end{enumerate}
\end{lem}

\begin{rmk}\label{rmk:plumbing-disks}\mbox{}
	\begin{enumerate}
		\item \label{item:wfuk-disk-inv} When $n=2$, the invertibility (up to homotopy) of $1+y_ex_e$ implies the invertibility of $1+x_ey_e$ for each $e\in E(Q)$ by Remark \ref{rmk:wfuk-plumb-2-inv}, and then the differential of $h_v$ implies the invertibility of $m_v$ by Remark \ref{rmk:punctured-sphere}\eqref{item:punctured-sphere-invertibility}.
	
		\item \label{item:reordering-disk} When $n=2$, for each $v\in V(Q)$, the order of the terms in the product
		\begin{equation}\label{eq:dhv-left}
			m_v\prod_{\substack{e= \bullet \to v\\ \s(e)=-1}} (1+x_ey_e)\prod_{e= v\to \bullet} (1+y_ex_e)
		\end{equation}
		appearing in the differential of $h_v$ can be permuted freely, since the order comes from the auxiliary ordering of the positively oriented punctures of $M'_v$ in the paragraph after \eqref{eq:wfuk-m-v}. Similarly, the order of the terms in the product
		\begin{equation}\label{eq:dhv-right}
			\prod_{\substack{e= \bullet \to v\\ \s(e)=1}}(1+x_ey_e)
		\end{equation}
		appearing in the differential of $h_v$ can be permuted freely, since the order comes from the auxiliary ordering of the negatively oriented punctures of $M'_v$. Finally, the term $m_v$ can be carried from \eqref{eq:dhv-left} to \eqref{eq:dhv-right} as we could have declared the puncture $U_{r(v)}$ of $M'_v$, whose monodromy corresponds to $m_v$, as negatively oriented instead of positively oriented. Any of these choices gives us the same dg category $\cP'_2$ up to quasi-equivalence. This is similar to Remark \ref{rmk:punctured-sphere}\eqref{item:sphere-reorder-punctures}.
	\end{enumerate}
\end{rmk}

\begin{proof}[Proof of Lemma \ref{lem:wfuk-plumbing-of-disks}]
	As remarked before, $\cW(P')$ can be computed by the homotopy colimit diagram \eqref{eq:hocolim-plumbing-disks} up to pretriangulated equivalence. We need to determine the categories and the functors in the diagram: 
	\begin{itemize}
		\item Proposition \ref{prp:wfuk-sphere} gives us $\cW(T^*S^{n-1})$. For each $e\in E(Q)$, we relabel the generating objects (resp.\ generating morphisms) of $\cW(T^*S^{n-1})\amalg\cW(T^*S^{n-1})$ in the diagram by $L_1^e$ and $L_2^e$ (resp.\ $z_1^e\colon L_1^e\to L_1^e$ and $z_2^e\colon L_2^e\to L_2^e$).
		
		\item Theorem \ref{thm:wfuk-plumbing-2} (if $n=2$) and Theorem \ref{thm:wfuk-plumbing-3} (if $n\geq 3$) give us $\cW(\Pi_n)$. For each $e\in E(Q)$, we relabel the generating objects (resp.\ generating morphisms) of $\cW(\Pi_n)$ by $L_1^e$ and $L_2^e$ (resp.\ $x_e\colon L_1^e\to L_2^e$ and $y_e\colon L_2^e\to L_1^e$).
		
		\item With these relabellings, and by Theorem \ref{thm:wfuk-plumbing-2}\eqref{item:inclusion-plumbing-2} (if $n=2$) and Theorem \ref{thm:wfuk-plumbing-3}\eqref{item:inclusion-plumbing-3} (if $n\geq 3$), we have
		\begin{gather*}
			\Phi_e\amalg \Psi_e\colon \cW(T^*S^{n-1})\amalg\cW(T^*S^{n-1})\to \cW(\Pi_n)\\
			L_1^e \mapsto L_1^e,\quad L_2^e \mapsto L_2^e, \quad z_1^e\mapsto
			\begin{cases}
				1 + y_e x_e & \text{if $n=2$}\\
				y_e x_e & \text{if $n\geq 3$}
			\end{cases}
				, \quad z_2^e\mapsto
				\begin{cases}
					1 + x_e y_e & \text{if $n=2$}\\
					x_e y_e & \text{if $n\geq 3$}
				\end{cases} .
		\end{gather*}
		
		\item For each $v\in V(Q)$, $\cW(T^*M'_v)$ is given by the equivalence in \eqref{eq:wfuk-m-v}.
		
		\item With these relabellings, and by the expression of $F_e$ and $G_e$ given in \eqref{eq:m-v-inclusion-a} and \eqref{eq:m-v-inclusion-b}, we have
		\begin{gather*}
			F_e\amalg G_e\colon \cW(T^*S^{n-1})\amalg\cW(T^*S^{n-1})\to\cW(T^*M'_v)\amalg\cW(T^*M'_w)\\
			L_1^e\mapsto L_v,\quad L_2^e\mapsto L_w,\quad z_1^e\mapsto a_i^v,\quad
			z_2^e\mapsto\begin{cases}
				a_j^w & \text{if $(-1)^{*_e}=1$,}\\
				b_l^w & \text{if $(-1)^{*_e}=-1$,}
			\end{cases}
		\end{gather*}
		for each $e=v\to w$, where $(-1)^{*_e}=(-1)^{n(n-1)/2}\s(e)$, and $i,j,k$ are determined as explained in the paragraph after \eqref{eq:wfuk-m-v}.
	\end{itemize}
	
	In the diagram \eqref{eq:hocolim-plumbing-disks}, all the dg categories are semifree, and the dg functor $\coprod_{e\in E(Q)}(\Phi_e\amalg \Psi_e)$ is a semifree extension. Therefore, Remark \ref{rmk colimit} and \cite{dwyer-spalinski} imply that the homotopy colimit in \eqref{eq:hocolim-plumbing-disks} becomes a usual colimit. Then, Proposition \ref{prp:colimit-dg} implies that $\cW(P')$ is pretriangulated equivalent to
	\[\left(\coprod_{v\in V(Q)}\cW(T^*M'_v)\right) \amalg \left(\coprod_{e\in E(Q)}\cW(\Pi_n)\right)\]
	with the identifications
	$F_e(L_1^e)=\Phi_e(L_1^e)$, $G_e(L_2^e)=\Psi_e(L_2^e)$, $F_e(z_1^e)=\Phi_e(z_1^e)$, $G_e(z_2^e)=\Psi_e(z_2^e)$
	for each $e= v\to w\in E(Q)$, or explicitly,
	\[L_v=L_1^e ,\quad L_w=L_2^e, \quad a_i^v=
		\begin{cases}
			1+y_ex_e & \text{if $n=2$} \\
			y_ex_e & \text{if $n\geq 3$}
		\end{cases},
	\]
	along with the identification
	\[a_j^w =\begin{cases}
		1+x_ey_e & \text{if $n=2$}\\
		x_ey_e & \text{if $n\geq 3$}
	\end{cases}\]
	if $(-1)^{*_e}=1$, or
	\[b_l^w=\begin{cases}
		1+x_ey_e & \text{if $n=2$}\\
		x_ey_e & \text{if $n\geq 3$}
	\end{cases}\]
	if $(-1)^{*_e}=-1$.
	
	Using these identifications, we can assume that $\cW(P')$ is, up to pretriangulated equivalence, given by the objects $L_v$ for each $v\in V(Q)$, and the generating morphisms
	\begin{itemize}
		\item $m_v:=a_{m_+}^v$ and $h_v$ for each $v\in V(Q)$,
		
		\item $x_e,y_e$ for each $e\in E(Q)$,
	\end{itemize}
	satisfying the conditions in Lemma \ref{lem:wfuk-plumbing-of-disks}, which proves the first part.
	
	For the second part, the map $\mu_v\colon S^{n-1}\xrightarrow{\sim}\partial \D^n\hookrightarrow \D^n=M'(v)$ can be equivalently considered as the map
	\[\mu_v\colon S^{n-1}\xrightarrow{\sim}\partial U_{r(v)}\hookrightarrow M'_v\]
	as discussed before \eqref{eq:inclusion-disk-boundary}. Also, by definition, we have $U_{r(v)}=U_{m_+}$. Then, Proposition \ref{prp sphere with punctures general}\eqref{item:punctured-sphere-inclusion-general} gives the induced dg functor
	\begin{gather*}
		\mu_v\colon\cW(T^*S^{n-1})\to\cW(T^*M'_v)\\
		L\mapsto L_v,\quad z\mapsto a_{m_+}^v=m_v ,
	\end{gather*}
	where $L$ (resp.\ $z$) is the generating object (resp.\ generating morphism) of $\cW(T^*S^{n-1})$ as in Proposition \ref{prp:wfuk-sphere}. By Theorem \ref{thm:hocolim-functor-dg}, this implies that there is an induced dg functor
	\begin{gather*}
		\mu_v\colon\cW(T^*S^{n-1})\to\cW(P')\\
		L\mapsto L_v,\quad z\mapsto m_v ,
	\end{gather*}
	which proves Lemma \ref{lem:wfuk-plumbing-of-disks}\eqref{item:inclusion-plumb-disks}.
\end{proof}

Before proceeding, let us recall our notation from Section \ref{subsubsection hocolim diagram B}: For each $v\in V(Q)$, given a connected oriented $n$-manifold $M(v)$ (with or without boundary), we write $M(v)^*$ for $M(v)\setminus U_v$, where $U_v$ is a small open disk in $M(v)$. So, $M(v)^*$ is homotopically a once-punctured $M(v)$, i.e., $M(v)\setminus\text{pt}$, and we call $U_v$ a puncture of $M(v)^*$ by abuse of notation.

Recall that by Proposition \ref{thm:cotangent-generation} and Remark \ref{rmk:loop-space}, $\cW(T^*M(v)^*)$ is (up to pretriangulated equivalence) given by a single object $K_v$ (representing a cotangent fiber) with the endomorphism (dg) algebra
\begin{equation}\label{eq:hom-punctured-mv}
	\hom^*(K_v,K_v):=C_{-*}(\Omega_pM(v)^*) ,
\end{equation}
where $C_{-*}(\Omega_pM(v)^*)$ is the cochain complex given by the chains on the based loop space $\Omega_pM(v)^*$ of $M(v)^*$ at $p\in M(v)^*$.

Next, we give an orientation on $S^{n-1}$, and orient $\partial U_v$ as a boundary of $M(v)^*\subset M(v)$. It is important to highlight that before this, the boundary $\partial U$ of any puncture $U$ was oriented as a boundary of $U$, but here, we are orienting it as a boundary of $M(v)^*$. Following this adjustment, we consider an inclusion map
\begin{equation}\label{eq:eta-2}
	\eta_v\colon S^{n-1}\xrightarrow{\sim}\partial U_v\hookrightarrow M(v)^* ,
\end{equation}
where the first arrow is orientation preserving diffeomorphism. Then, we have the induced dg functor
\begin{gather}\label{eq:eta-3}
	\eta_v\colon \cW(T^*S^{n-1})\to\cW(T^*M(v)^*)\\
	L\mapsto K_v, \quad z\mapsto \eta_v(z)\notag ,
\end{gather}
where $\cW(T^*S^{n-1})$ is given in Proposition \ref{prp:wfuk-sphere}. We can describe $\eta_v(z)$ equivalently as follows: The inclusion \eqref{eq:eta-2} induces a chain map
\begin{align}\label{eq:eta-inclusion}
	\eta_v\colon C_{-*}(\Omega_p S^{n-1})&\to C_{-*}(\Omega_pM(v)^*)\\
	z&\mapsto \eta_v(z)\notag ,
\end{align}
where $C_{-*}(\Omega_p S^{n-1})$ is given by $\hom^*(L,L)$ in $\cW(T^*S^{n-1})$ in Proposition \ref{prp:wfuk-sphere}, which can be seen by Proposition \ref{thm:cotangent-generation} and Remark \ref{rmk:loop-space}. Informally, one can think $\eta_v(z)$ as representing the (higher if $n\geq 3$) monodromy of degree $2-n$ coming from the puncture $U_v$ of $M(v)^*$.

We are now ready to present one of the main theorem of the paper, which is a formula for the wrapped Fukaya category of any plumbing space.

\begin{thm}
	\label{thm:wfuk-plumbing-general}
	Fix $n\geq 2$. Let $P$ be a plumbing of cotangent bundles of connected, oriented $n$-manifolds (with or without boundary) along any quiver with or without negative intersections.
	In other words, there is an arbitrary plumbing data $(Q,M,\s)$ such that $P = P(Q,M,\s)$. 
	Then, the wrapped Fukaya category of $P$ is, up to pretriangulated equivalence, given by
		\[\cW(P)\simeq\begin{cases}
			\cP_2[\{1+y_ex_e\vb e\in E(Q)\}^{-1}] & \text{if $n=2$,}\\
			\cP_n & \text{if $n\geq 3$,}
		\end{cases}\]
		where $\cP_n$ is a semifree dg category given as follows:
		\begin{enumerate}[label = (\roman*)]
			\item {\em Objects:} $L_v$ (representing a Lagrangian cocore dual to $M(v)$) for any $v \in V(Q)$.
			\item {\em Generating morphisms:} There are three types of generating morphisms: 
			\begin{itemize}
				\item For any $v \in V(Q)$, $h_v\colon L_v\to L_v$. 
				\item For any $v\in V(Q)$, the generating morphisms of $ C_{-*}(\Omega_p (M(v)\setminus\text{pt}))$, where \\ $C_{-*}(\Omega_p (M(v)\setminus\text{pt}))$ is considered as a semifree dg algebra, see Remark \ref{rmk:wfuk-plumbing}\eqref{item:based-loop-1} and \eqref{item:based-loop-2}.
				\item For any $e=v\to w\in E(Q)$, $x_e\colon L_v \to L_w,\quad y_e\colon L_w \to L_v$.
			\end{itemize}
			\item {\em Degrees:} $|h_v|=1-n,\quad |x_e|=0,\quad |y_e|=2-n$. 
			\item {\em Differentials:} $d x_e = d y_e= 0$, and
			\[dh_v=\begin{cases}
				\displaystyle \eta_v(z)\prod_{\substack{e= \bullet \to v\\ \s(e)=-1}} (1+x_ey_e)\prod_{e= v\to \bullet} (1+y_ex_e) - \prod_{\substack{e= \bullet \to v\\ \s(e)=1}}(1+x_ey_e) &\text{if $n=2$,}\\
				\displaystyle \eta_v(z)+\sum_{e= v\to \bullet} y_ex_e + \sum_{e= \bullet \to v}(-1)^{n(n-1)/2}\s(e)\, x_ey_e &\text{if $n\geq 3$,}
			\end{cases}\]
		\end{enumerate}
		where $\eta_v(z)\in C_{-*}(\Omega_p (M(v)\setminus\text{pt}))$ is as in \eqref{eq:eta-inclusion}.
\end{thm}

\begin{rmk}\label{rmk:wfuk-plumbing}\mbox{}
	\begin{enumerate}
		\item\label{item:based-loop-1} We can always assume that $C_{-*}(\Omega_p(M(v)\setminus \text{pt}))$ is a semifree dg algebra by taking its semifree resolution, hence it has generating morphisms.
		
		\item\label{item:based-loop-2} We can compute $C_{-*}(\Omega_p(M(v)\setminus \text{pt}))$ using Theorem \ref{thm:seifert-van-kampen} since
		\[\cP(M(v)\setminus\text{pt})\simeq C_{-*}(\Omega_p(M(v)\setminus \text{pt}))\]
		up to quasi-equivalence by Remark \ref{rmk:loop-space}, when the latter is seen as a dg category with a single object whose endomorphism (dg) algebra is $C_{-*}(\Omega_p(M(v)\setminus \text{pt}))$.
		We can consider an open covering of $M(v)\setminus\text{pt}$ consisting of balls to apply Theorem \ref{thm:seifert-van-kampen}, where the homotopy colimit can be computed by our Theorem \ref{thm:hocolim-functor-dg}, which gives a semifree dg category as a result.
	
		\item When $n=2$, the invertibility (up to homotopy) of $1+y_ex_e$ implies the invertibility of $1+x_ey_e$ for each $e\in E(Q)$ as in Remark \ref{rmk:plumbing-disks}\eqref{item:wfuk-disk-inv}.
		
		\item\label{item:reordering-general} When $n=2$, for each $v\in V(Q)$, the order of the terms in the differential of $h_v$ can be rearranged as explained in Remark \ref{rmk:plumbing-disks}\eqref{item:reordering-disk} (after setting $m_v=\eta_v(z)$). Any of these rearrangements gives us the same dg category $\cP_2$ up to quasi-equivalence.
	\end{enumerate}
\end{rmk}

\begin{proof}[Proof of Theorem \ref{thm:wfuk-plumbing-general}]
	Let $M(v)^*$ denote $M(v)\setminus\text{pt}$. As discussed in Section \ref{subsubsection hocolim diagram B}, we have the homotopy colimit diagram
	\begin{equation}\label{eq:hocolim-plumbing-general}
		\cW(P) \simeq \textup{hocolim}\left(
		\begin{tikzcd}[column sep=-3.5em]
			\cW(P(Q,M',\s)) & & \coprod_{v \in V(Q)} \cW(T^*M(v)^*)\\
			& \coprod_{v \in V(Q)} \cW(T^*S^{n-1}) \ar[lu,"\coprod_{v \in V(Q)} \mu_v"]\ar[ru,"\coprod_{v \in V(Q)} \eta_v"']
		\end{tikzcd}\right) ,
	\end{equation}
	up to pretriangulated equivalence, where $(Q,M',\s)$ is another plumbing data with the same quiver $Q$ and the same map $\s$ but $M'(v)=\mathbb{D}^n$ for all $v\in V(Q)$. Then,
	\begin{itemize}
		\item Proposition \ref{prp:wfuk-sphere} gives us $\cW(T^*S^{n-1})$ such that it has the generating object $L$ and generating morphism $z$.
		
		\item Lemma \ref{lem:wfuk-plumbing-of-disks} describes $\cW(P(Q,M',\s))$ and the dg functors $\mu_v$ for all $v\in V(Q)$.
		
		\item As discussed before, for each $v\in V(Q)$, the category $\cW(T^*M(v)^*)$ is (up to pretriangulated equivalence) given by a single object $K_v$ (representing a cotangent fiber) with the endomorphism (dg) algebra $\hom^*(K_v,K_v)=C_{-*}(\Omega_pM(v)^*)$ as in \eqref{eq:hom-punctured-mv}.
		
		\item The dg functors $\eta_v$ are as in \eqref{eq:eta-3} for all $v\in V(Q)$.
	\end{itemize}
	In the diagram \eqref{eq:hocolim-plumbing-general}, all the dg categories are semifree, and the dg functor $\coprod_{v\in V(Q)}\mu_v$ is a semifree extension. It concludes that, by Remark \ref{rmk colimit} and \cite{dwyer-spalinski}, the homotopy colimit in \eqref{eq:hocolim-plumbing-general} becomes a usual colimit. Then, Proposition \ref{prp:colimit-dg} implies that $\cW(P)$ is pretriangulated equivalent to
	\[\cW(P(Q,M',\s)) \amalg \left(\coprod_{v \in V(Q)} \cW(T^*M(v)^*)\right)\]
	with the identifications
	\[\mu_v(L)=\eta_v(L),\quad \mu_v(z)=\eta_v(z) ,\]
	for each $v\in V(Q)$, or explicitly,
	\[L_v=K_v ,\quad m_v=\eta_v(z) .\]
	Using these identifications, we can assume that $\cW(P)$ is, up to pretriangulated equivalence, given by the objects $L_v$ for each $v\in V(Q)$, and the generating morphisms
	\begin{itemize}
		\item $h_v$ and the generating morphisms of $C_{-*}(\Omega_p(M(v)\setminus \text{pt}))$ for each $v\in V(Q)$,
		
		\item $x_e,y_e$ for each $e\in E(Q)$,
	\end{itemize}
	satisfying the conditions in Theorem \ref{thm:wfuk-plumbing-general}, which concludes the proof.
\end{proof}

\subsection{Plumbing spaces with an arbitrary grading structure}
\label{subsection plumbing grading}
In this subsection, we improve Theorem \ref{thm:wfuk-plumbing-general}, and give a formula for the wrapped Fukaya category $\cW(P)$ of any plumbing space $P$ with an arbitrary grading structure in Theorem \ref{thm:wfuk-plumbing-grading}.

Through this subsection, we again fix an integer $n \geq 2$. Let $P=P(Q,M,\s)$ be a plumbing space for an arbitrary plumbing data $(Q,M,\s)$ where $Q$ is any (finite) quiver, $M(v)$ is a connected oriented $n$-manifold (with or without boundary) for each $v\in V(Q)$, and $\s\colon E(Q)\to\{\pm 1\}$ is any function.

By Remark \ref{rmk:grading-pin}, the wrapped Fukaya category $\cW(P)$ of $P$ depends on the choice of a grading structure $\eta$ on $P$. Since there are effectively $H^1(P;\Z)$-many grading structures on $P$ (see \cite{seidel}), we write $\eta\in H^1(P;\Z)$ meaning that $\eta$ corresponds to one of those grading structures. Let us denote the wrapped Fukaya category of $P$ with a grading structure $\eta$ by $\cW(P;\eta)$ to distinguish different choices.

We can compute $\cW(P;\eta)$ for different choices of $\eta\in H^1(P;\Z)$ via the homotopy colimit diagram \eqref{eq:hocolim-plumbing-general} by
\begin{itemize}
	\item replacing $\cW(T^*M(v)^*)$ (recall that $M(v)^*:=M(v)\setminus\text{pt})$ with $\cW(T^*M(v)^*;\xi_v)$ for some grading structure $\xi_v\in H^1(M(v)^*;\Z)$ on $T^*M(v)^*$ for each $v\in V(Q)$, and
	\item replacing $\cW(P(Q,M',\s))$ with $\cW(P(Q,M',\s);\eta')$ for some grading structure
	\[\eta'\in H^1(P(Q,M',\s);\Z)=H^1(Q;\Z),\]
	where $M'(v)$ is an $n$-dimensional closed disk for each $v\in V(Q)$.
\end{itemize}
Indeed, these choices make up all the grading structure choices $\eta\in H^1(P;\Z)$ on $P$ since for $n\geq 2$,
\begin{equation}\label{eq:grading-plumbing}
	H^1(P;\Z)=H^1(Q;\Z)\oplus \bigoplus_{v\in V(Q)} H^1(M(v)^*;\Z) .
\end{equation}

In this subsection, we choose the standard grading structures on $T^*M(v)^*$ as explained in \cite{nadler-zaslow}, although it is also possible to deal with different grading structures. Hence, we are only interested in the grading structures $\eta$ on $P$ such that $\eta\in H^1(Q;\Z)$ via the identification \eqref{eq:grading-plumbing}.

We can compute $\cW(P(Q,M',\s);\eta')$ for a general $\eta' \in H^1(Q;\mathbb{Z})$ using the homotopy colimit diagram \eqref{eq:hocolim-plumbing-disks} with the same functors $\Psi_e$, $F_e$, $G_e$ as before, and a possibly different functor $\Phi_e$: First, note that for each $e=v\to w\in E(Q)$, the functors $\Phi_e$, $\Psi_e$, $F_e$, $G_e$ in \eqref{eq:hocolim-plumbing-disks} are induced by the maps \eqref{eq:plumbing-inclusion-1} and \eqref{eq:F-inclusion}, each of which sends a cotangent fiber $L$ of $T^*S^{n-1}$ to a Lagrangian. However, when determining the functors on the relevant Fukaya categories induced by those maps, the image of $L$ is determined only up to a shift. This gives $\Z$-many choices when determining each of the functors $\Phi_e$, $\Psi_e$, $F_e$, and $G_e$.

Moreover, it is easy to see that by replacing the categories $\cW(\Pi_n)$ and $\cW(T^*M'_v)$ in \eqref{eq:hocolim-plumbing-disks} by their images under a shift functor (which is an autoequivalence) if necessary, we can assume that the functors $\Psi_e$, $F_e$, $G_e$ are given as in Lemma \ref{lem:wfuk-plumbing-of-disks}, and $\Phi_e$ is determined up to $\Z$-many choices.

For that reason, we choose $d_e\in\Z$ for each $e\in E(Q)$ representing the different possible choices for $\Phi_e$ in \eqref{eq:hocolim-plumbing-disks}, where such choices determine a grading structure $\eta'$ for $P(Q,M',\s)$, or $\eta$ for $P$. Then, $\cW(P(Q,M',\s);\eta')$ can be computed by modifying the proof of Lemma \ref{lem:wfuk-plumbing-of-disks} in the following way: For any edge $e\in E(Q)$ and a choice $d_e\in\Z$, we replace the category $\cD^{12}_n$ by $\tilde \cD^{12}_n$ (which is used when describing $\cW(\Pi_n)$), and the functor $\Phi_e$ by $\tilde \Phi_e$ in \eqref{eq:hocolim-plumbing-disks}, as explained in Remark \ref{rmk:shifting-generators} (if $n\geq 3$) and Remark \ref{rmk:shifting-generators-2} (if $n=2$) with the choices $m_1=d_e$ and $m_2=0$. Note that $\Psi_e$ will be automatically replaced by $\tilde\Psi_e$ as a result by these remarks. Finally, $\cW(P;\eta)$ can be computed by modifying the proof of Theorem \ref{thm:wfuk-plumbing-general} by replacing $\cW(P(Q,M',\s))$ with $\cW(P(Q,M',\s);\eta')$.

To see how the choices of $d_e$'s are related to the grading structures on $P$ coming from $H^1(Q;\Z)$, define a map
\begin{align}\label{eq:grading-sigma}
	\sigma\colon\Z^{|E(Q)|}&\to H^1(Q;\Z)\\
	\langle d_e\in\Z \vb e\in E(Q)\rangle&\mapsto \sum_{i=1}^m c_i q_i\notag
\end{align}
where $\{q_i\vb i=1,\ldots,m\}$ is a generating set for $H^1(Q;\Z)$ (which is finite, since $Q$ is a finite quiver), and each $c_i\in\Z$ is determined as follows: Each $q_i$ can be seen as an oriented loop in the graph $Q$. Hence, $q_i$ is a concatenation of $e_1,\dots,e_r$ for some oriented edges $e_j$'s in the graph $Q$, whose orientations do not need to match with the orientation of the quiver $Q$. Then, we define
\[c_i:=\sum_{j=1}^r(-1)^{\dagger_{e_j}} d_{e_j}\]
where $\dagger_{e_j}$ is even if $e_j$ has the same orientation as the quiver $Q$, and odd otherwise. With this definition, the map $\sigma$ is well-defined and surjective.

Moreover, it is not hard to see (by replacing the categories in \eqref{eq:hocolim-plumbing-disks} by their images under a shift functor if necessary) that if $\langle d_e\rangle$ and $\langle d_e'\rangle$ have the same images under $\sigma$, then they give the same wrapped Fukaya category for $P$. Hence, the map $\sigma$ can be lifted to the map
\begin{align*}
	\tilde \sigma\colon\Z^{|E(Q)|}&\to \{\cW(P;\eta)\vb \eta\in H^1(Q;\Z)\}\\
	\langle d_e\in\Z \vb e\in E(Q)\rangle&\mapsto \cW(P;\sigma(\langle d_e\rangle))\notag
\end{align*}
where $\cW(P;\sigma(\langle d_e\rangle))$ can be computed as explained in the paragraph above \eqref{eq:grading-sigma}. Note that $\tilde\sigma$ is surjective.

Hence, we get the following generalization of Theorem \ref{thm:wfuk-plumbing-general}:

\begin{thm}
	\label{thm:wfuk-plumbing-grading}
	Fix $n\geq 2$. Let $P$ be a plumbing of cotangent bundles of connected, oriented $n$-manifolds (with or without boundary) along any quiver with or without negative intersections.
	In other words, there is an arbitrary plumbing data $(Q,M,\s)$ such that $P = P(Q,M,\s)$. 
	Then, the wrapped Fukaya category of $P$ equipped with a grading structure $\eta\in H^1(Q;\Z)$ as in \eqref{eq:grading-plumbing} is, up to pretriangulated equivalence, given by
	\[\cW(P;\eta)\simeq\begin{cases}
		\cP_{2,\eta}[\{1+y_ex_e\vb e\in E(Q)\}^{-1}] & \text{if $n=2$,}\\
		\cP_{n,\eta} & \text{if $n\geq 3$,}
	\end{cases}\]
	where $\cP_{n,\eta}$ is a semifree dg category given as follows:
	\begin{enumerate}[label = (\roman*)]
		\item {\em Objects:} $L_v$ (representing a Lagrangian cocore dual to $M(v)$) for any $v \in V(Q)$.
		\item {\em Generating morphisms:} There are three types of generating morphisms: 
		\begin{itemize}
			\item For any $v \in V(Q)$, $h_v\colon L_v\to L_v$. 
			\item For any $v\in V(Q)$, the generating morphisms of $ C_{-*}(\Omega_p (M(v)\setminus\text{pt}))$, where \\$C_{-*}(\Omega_p (M(v)\setminus\text{pt}))$ is considered as a semifree dg algebra, see Remark \ref{rmk:wfuk-plumbing}\eqref{item:based-loop-1} and \eqref{item:based-loop-2}.
			\item For any $e=v\to w\in E(Q)$, $x_e\colon L_v \to L_w,\quad y_e\colon L_w \to L_v$.
		\end{itemize}
		\item {\em Degrees:} $|h_v|=1-n,\quad |x_e|=d_e,\quad |y_e|=2-n-d_e$. 
		\item {\em Differentials:} $d x_e = d y_e= 0$, and
		\[dh_v=\begin{cases}
			\displaystyle \eta_v(z)\prod_{\substack{e= \bullet \to v\\ \s(e)=-1}} (1+x_ey_e)\prod_{e= v\to \bullet} (1+y_ex_e) - \prod_{\substack{e= \bullet \to v\\ \s(e)=1}}(1+x_ey_e) &\text{if $n=2$,}\\
			\displaystyle \eta_v(z)+\sum_{e= v\to \bullet} (-1)^{nd_e}y_ex_e + \sum_{e= \bullet \to v}(-1)^{n(n-1)/2}\s(e)\, x_ey_e &\text{if $n\geq 3$,}
		\end{cases}\]
	\end{enumerate}
	where $\eta_v(z)\in C_{-*}(\Omega_p (M(v)\setminus\text{pt}))$ is as in \eqref{eq:eta-inclusion}, and $\langle d_e\in\Z\vb e\in E(Q)\rangle$ is any collection that satisfies $\eta=\sigma(\langle d_e\rangle)$, where $\sigma$ is given in \eqref{eq:grading-sigma}.
\end{thm}

\begin{rmk}\label{rmk:wfuk-plumbing-grading}\mbox{}
	\begin{enumerate}
		\item When $n=2$, the invertibility (up to homotopy) of $1+y_ex_e$ implies the invertibility of $1+x_ey_e$ for each $e\in E(Q)$ as in Remark \ref{rmk:plumbing-disks}\eqref{item:wfuk-disk-inv}.
		
		\item\label{item:reordering-grading} When $n=2$, for each $v\in V(Q)$, the order of the terms in the differential of $h_v$ can be rearranged as explained in Remark \ref{rmk:plumbing-disks}\eqref{item:reordering-disk} (after setting $m_v=\eta_v(z)$). Any of these rearrangements gives us the same dg category $\cP_{2,\eta}$ up to quasi-equivalence.
	\end{enumerate}
\end{rmk}

As an illustration of the importance of grading structures in computations, we consider the following setting: Fix $n\geq 2$. Let $(Q,M,\s)$ be an arbitrary plumbing data where $M(v)$ is a connected oriented $n$-manifold for each $v\in V(Q)$. Choose an arrow $u\in E(Q)$ and define another plumbing data $(Q',M,\s')$ where
\begin{itemize}
	\item $Q'$ is the same quiver as $Q$ except the arrow $u$ is reversed and renamed as $\bar u$,
	
	\item $\s'$ is the same as $\s$ except $\s'(\bar u):=(-1)^n\s(u)$.
\end{itemize}
Recall from Proposition \ref{prp property 1} (if $n$ is odd) and Proposition \ref{prp property 2} (if $n$ is even) that the associated plumbing spaces $P(Q,M,\s)$ and $P(Q',M,\s')$ are equivalent. However, their wrapped Fukaya categories coming from Theorem \ref{thm:wfuk-plumbing-general} are not equivalent in general. The reason is that we should also take the grading structures into consideration. Using Theorem \ref{thm:wfuk-plumbing-grading} instead, we get the following proposition:

\begin{prp}\label{prp:wfuk-equivalent-1}
	Fix $n\geq 2$. Let $(Q,M,\s)$ and $(Q',M,\s')$ be plumbing data as above. Then, for any $\langle d_e\in\Z\vb e\in E(Q)\rangle$, we have a quasi-equivalence
	\[\cW(P(Q,M,\s);\sigma(\langle d_e\rangle))\simeq\cW(P(Q',M,\s');\sigma(\langle d'_e\rangle))\]
	where $\sigma$ is given in \eqref{eq:grading-sigma}, and $\langle d'_e\vb e\in E(Q)\rangle$ be such that $d'_e:=d_e$ for all $e\in E(Q)$ except $d'_{\bar u}:=2-n - d_u$.
\end{prp}

\begin{proof}
	Both categories can be computed via Theorem \ref{thm:wfuk-plumbing-grading}. Then, if $n\geq 3$, the desired equivalence is given by the dg functor
	\begin{equation}\label{eq:flipping-edge}\cW(P(Q,M,\s);\sigma(\langle d_e\rangle))\to\cW(P(Q',M,\s');\sigma(\langle d'_e\rangle))\end{equation}
	which is identity on generating objects and morphisms, except
	\[x_u\mapsto y_{\bar u},\quad
	y_u\mapsto (-1)^{n+nd_u+n(n-1)/2}\s(u)\, x_{\bar u} .\]
	
	If $n=2$ and $\s(u)=-1$, the desired equivalence is given again by the dg functor \eqref{eq:flipping-edge} after reordering terms as in Remark \ref{rmk:wfuk-plumbing-grading}\eqref{item:reordering-grading} if necessary. If $n=2$ and $\s(u)=1$, the proof is a bit more complicated. When expressing $\cW(P(Q',M,\s');\sigma(\langle d'_e\rangle))$, we should modify Theorem \ref{thm:wfuk-plumbing-grading} by modifying the construction of the plumbing space $P(Q',M,\s')$: In Section \ref{subsubsection gluing information}, we should replace the orientation preserving $F_{\bar u}$ with an orientation reversing $F_{\bar u}$, and the orientation reversing $G_{\bar u}$ with an orientation preserving $G_{\bar u}$. The resulting space will not change as a result of the simultaneous modification. Then, the dg functor \eqref{eq:flipping-edge} gives the desired equivalence, except we set
	\[x_u\mapsto y_{\bar u},\quad y_u\mapsto x_{\bar u} .\]
\end{proof}

Finally, consider plumbing data $(Q,M,\s)$ and $(Q,M,\s')$ with $\s$ and $\s'$ related as in Proposition \ref{prp property 3} or Corollary \ref{cor property 4}. In this case, the plumbing spaces $P(Q,M,\s)$ and $P(Q,M,\s')$ are equivalent. Then, we have:

\begin{prp}\label{prp:wfuk-equivalent-2}
	Fix $n\geq 2$. Let $(Q,M,\s)$ and $(Q,M,\s')$ be plumbing data as above. Then, we have a quasi-equivalence
	\[\cW(P(Q,M,\s);\eta)\simeq\cW(P(Q,M,\s');\eta)\]
	where $\eta\in H^1(Q;\Z)$ is a grading structure on $P(Q,M,\s)$ and $P(Q,M,\s')$ given as in \eqref{eq:grading-plumbing}.
\end{prp}

\begin{proof}
	Since the plumbing spaces $P(Q,M,\s)$ and $P(Q,M,\s')$ share the same underlying quiver, their wrapped Fukaya categories are equivalent with the same grading structure. This equivalence can also be directly demonstrated, following a similar approach to the proof of Proposition \ref{prp:wfuk-equivalent-1}.
\end{proof}

\subsection{Plumbings of $T^*S^n$'s for $n \geq 3$ with an arbitrary grading structure}
\label{subsection plumbings in (a)}

In this subsection, we compute wrapped Fukaya categories of plumbings of $T^*S^n$'s for $n\geq 3$. A detailed statement can be found in Corollary \ref{cor homotopy colimit for sphere plumbing}, which is a corollary of Theorem \ref{thm:wfuk-plumbing-grading}. Additionally, we compare these computed wrapped Fukaya categories to Ginzburg dg categories in Corollary \ref{cor:ginzburg-equal-wfuk}.

\begin{cor}
	\label{cor homotopy colimit for sphere plumbing}
	Fix $n\geq 3$. Let $P_{S^n}$ be a plumbing of $T^*S^n$'s along any quiver $Q$ with or without negative intersections, where $S^n$ is an $n$-sphere.
	In other words, there is a plumbing data $(Q,M,\s)$ with $M(v)=S^n$ for any $v\in V(Q)$ such that $P_{S^n} = P(Q,M,\s)$. 
	Then, the wrapped Fukaya category of $P_{S^n}$ equipped with a grading structure $\eta\in H^1(Q;\Z)$ as in \eqref{eq:grading-plumbing} is, up to pretriangulated equivalence, given by
	\[\cW(P_{S^n};\eta)\simeq\cP^{S^n}_{n,\eta} ,\]
	where $\cP^{S^n}_{n,\eta}$ is a semifree dg category given as follows:
	\begin{enumerate}[label = (\roman*)]
		\item {\em Objects:} $L_v$ (representing a Lagrangian cocore dual to $M(v)=S^n$) for any $v \in V(Q)$.
		\item {\em Generating morphisms:} For any $v\in V(Q)$,
		\[h_v\colon L_v\to L_v ,\]
		and for any $e=v\to w\in E(Q)$,
		\[x_e\colon L_v \to L_w,\quad y_e\colon L_w \to L_v .\]
		\item {\em Degrees:} $|h_v|=1-n,\quad |x_e|=d_e,\quad |y_e|=2-n-d_e$. 
		\item {\em Differentials:} $d x_e = d y_e= 0$, and
		\[dh_v=\sum_{e= v\to \bullet} (-1)^{nd_e}y_ex_e + \sum_{e= \bullet \to v}(-1)^{n(n-1)/2}\s(e)\, x_ey_e ,\]
	\end{enumerate}
	where $\langle d_e\in\Z\vb e\in E(Q)\rangle$ is any collection that satisfies $\eta=\sigma(\langle d_e\rangle)$, and $\sigma$ is given in \eqref{eq:grading-sigma}.
\end{cor}

\begin{proof}
	It is a direct corollary of Theorem \ref{thm:wfuk-plumbing-grading} by setting $M(v)=S^n$, which gives
	\[C_{-*}(\Omega_p (M(v)\setminus\text{pt}))=C_{-*}(\Omega_p \D^n)\simeq k\]
	for each $v\in V(Q)$, where $k$ is the coefficient ring representing the multiples of the identity of $L_v$. Then, for each $v\in V(Q)$, $\eta_v(z)=0\in C_{-*}(\Omega_p \D^n)\simeq k$ for degree reasons since $|z|=2-n\neq 0$.
\end{proof}

To get the wrapped Fukaya category $\cW(P_{S^n})$ of a plumbing $P_{S^n}$ of $T^*S^n$'s with the {\em standard grading structure} as in Section \ref{subsection plumbings disks}, we set $d_e=0$ for each $e\in E(Q)$ in Corollary \ref{cor homotopy colimit for sphere plumbing}. Alternatively, $\cW(P_{S^n})$ can be obtained as a direct corollary of Theorem \ref{thm:wfuk-plumbing-general}.

Next, we will compare wrapped Fukaya categories $\cW(P_{S^n};\eta)$ to Ginzburg dg categories without potential, originally introduced in \cite{ginzburg-calabi-yau} (when $n=3$), and then in \cite{keller-calabi-yau} in more general context.

\begin{dfn}
	Fix $n\in\Z$. Let $Q$ be a graded quiver, i.e., each of its arrows $e$ is associated with an integer value $q_e$. We define the {\em Ginzburg dg category} $\cG_n(Q)$ of $Q$ (without potential) as a semifree dg category defined as follows:
	\begin{enumerate}[label = (\roman*)]
		\item {\em Objects:} $v \in V(Q)$.
		\item {\em Generating morphisms:} For any $v\in V(Q)$ and for any $e=v\to w\in E(Q)$,
		\[t_v\colon v\to v, e\colon v \to w,\quad e^*\colon w \to v .\]
		\item {\em Degrees:} $|t_v|=1-n,\quad |e|=q_e,\quad |e^*|=2-n-q_e$. 
		\item {\em Differentials:} $d e = d e^*= 0$, and
		\[dt_v=\sum_{e= v\to \bullet} e^*e - \sum_{e= \bullet \to v}(-1)^{|e| |e^*|} e e^* .\]
	\end{enumerate}
\end{dfn}

\begin{cor}\label{cor:ginzburg-equal-wfuk}
	Fix $n\geq 3$. Let $Q$ be a graded quiver, and denote the grading of each arrow $e\in E(Q)$ by $q_e\in\Z$. Define a plumbing space $P_{S^n}:=P(Q,M,\s)$ with a grading structure $\eta \in H^1(P_{S^n};\Z)$ by
	\begin{itemize}
		\item $M(v):=S^n$ for each $v\in V(Q)$,
		
		\item $\displaystyle \s(e):=-(-1)^{q_e+\frac{n(n-1)}{2}}$ for each $e\in E(Q)$,
		
		\item $\eta:=\sigma(\langle q_e\rangle)$, where $\sigma$ is given in \eqref{eq:grading-sigma}.
	\end{itemize}
	Then, the Ginzburg dg category $\cG_n(Q)$ of $Q$ is quasi-equivalent to the full subcategory of $\cW(P_{S^n};\eta)$ consisting of Lagrangian cocores $L_v$ for each $v\in V(Q)$ chosen as in Corollary \ref{cor homotopy colimit for sphere plumbing} with $d_e:=q_e$. More precisely,
	\[\cG_n(Q)=\cP^{S^n}_{n,\eta} .\]
	In particular, we have a pretriangulated equivalence
	\[\cG_n(Q)\simeq \cW(P_{S^n};\eta) .\]
\end{cor}

\begin{proof}
	First, we relabel the objects and generating morphisms of $\cG_n(Q)$ as follows:
	\begin{align*}
		v\mapsto L_v&,\quad t_v\mapsto h_v &&\text{ for each }v\in V(Q),\\
		e\mapsto x_e&,\quad e^*\mapsto (-1)^{nq_e}y_e &&\text{for each }e\in E(Q) .
	\end{align*}
	Then, we see that $\cG_n(Q)=\cP^{S^n}_{n,\eta}$. Rest of the corollary directly follows.
\end{proof}

The wrapped Fukaya category of any $2n$-dimensional Weinstein manifold is homologically smooth and $n$-Calabi-Yau, hence Corollary \ref{cor:ginzburg-equal-wfuk} implies a result of \cite{keller-calabi-yau} when $n\geq 3$:

\begin{cor}\label{cor:ginzburg-smooth-cy}
	Given $n\geq 3$ and any graded quiver $Q$, the Ginzburg dg category $\cG_n(Q)$ is homologically smooth and $n$-Calabi-Yau.
\end{cor}
\begin{proof}
	It follows Corollary \ref{cor:ginzburg-equal-wfuk} and \cite[Theorem 1.3]{Ganatra12}.
\end{proof}

\begin{rmk}
	A special case of Corollary \ref{cor homotopy colimit for sphere plumbing} is given by \cite{lekili-ueda} when $Q$ is a Dynkin quiver and by \cite{Asplund21} when $Q$ is any quiver with $\s(e)=-(-1)^{n(n-1)/2}$ for each $e\in E(Q)$ (implicitly chosen), and with a fixed grading structure given by $d_e=0$ for all $e\in E(Q)$, where the coefficient ring $k$ is a field. In comparison, the advantage of our computation in Corollary \ref{cor homotopy colimit for sphere plumbing} lies in its applicability to:
	\begin{itemize}
		\item any quiver $Q$,
		
		\item any sign data on $Q$, i.e., $\s(e)$ is arbitrary for each $e\in E(Q)$,
		
		\item any grading structure given by $H^1(Q;\Z)$,
		
		\item any commutative ring as a coefficient ring $k$.
	\end{itemize}
	As a result, we can realize the Ginzburg dg category of any graded quiver (without potential) as a wrapped Fukaya category of a plumbing space in Corollary \ref{cor:ginzburg-equal-wfuk}.
\end{rmk}

We conclude this subsection with a characterization of some pretriangulated (and hence, Morita) equivalent Ginzburg dg categories, observed through their realization as wrapped Fukaya categories:

\begin{cor}
	Fix $n\in \Z$. Let $Q$ be a graded quiver, and $\langle q_e\in\Z\vb e\in E(Q)\rangle$ be the grading of its edges. Let $Q'$ be another graded quiver with the same underlying quiver and with the grading $\langle q_e'\in\Z\vb e\in E(Q')\rangle$. If
	\[\sigma(\langle q_e\rangle)=\sigma(\langle q_e'\rangle)\in H^1(Q;\Z) ,\]
	where $\sigma$ is given in \eqref{eq:grading-sigma}, then there is a pretriangulated equivalence
	\[\cG_n(Q)\simeq\cG_n(Q') .\]
\end{cor}

\begin{proof}
	First, assume $n\geq 3$. Let $\eta:=\sigma(\langle q_e\rangle)=\sigma(\langle q_e'\rangle)\in H^1(Q;\Z)$. By Corollary \ref{cor:ginzburg-equal-wfuk}, we have $\cG_n(Q)\simeq\cW(P(Q,M,\s);\eta)$, where $M(v):=S^n$ for all $v\in V(Q)$ and $\s(e):=-(-1)^{q_e+\frac{n(n-1)}{2}}$ for all $e\in E(Q)$. We also have $\cG_n(Q')\simeq\cW(P(Q',M,\s');\eta)$, where $\s'(e):=-(-1)^{q'_e+\frac{n(n-1)}{2}}$ for all $e\in E(Q')$. Then, it is not hard to show that
	\[\cW(P(Q,M,\s);\eta)\simeq\cW(P(Q',M,\s');\eta)\]
	using Proposition \ref{prp:wfuk-equivalent-2} along with the facts that $Q$ and $Q'$ have the same underlying quiver and $\sigma(\langle q_e\rangle)=\sigma(\langle q_e'\rangle)\in H^1(Q;\Z)$, which proves the corollary for $n\geq 3$. Since the statement only depends on the parity of $n$, it holds for any $n\in\Z$.
\end{proof}

\subsection{Plumbings of cotangent bundles of oriented closed surfaces with an arbitrary grading structure}
\label{subsection plumbings in (b)}

In this subsection, we compute wrapped Fukaya categories of plumbings of cotangent bundles of oriented closed surfaces. A detailed statement can be found in Corollary \ref{cor:wfuk-plumbing-surface}, which is a corollary of Theorem \ref{thm:wfuk-plumbing-grading}. Additionally, we compare these computed wrapped Fukaya categories to derived multiplicative projective algebras in Corollary \ref{cor:multiplicative-preprojective-algebra}.

\begin{cor}
	\label{cor:wfuk-plumbing-surface}
	Let $P_{\Sigma}$ be a plumbing of $T^*\Sigma_{g_v}$'s along any quiver $Q$ with or without negative intersections, where $\Sigma_{g_v}$ is an oriented closed surface of genus $g_v\geq 0$ for any $v\in V(Q)$.
	In other words, there is a plumbing data $(Q,M,\s)$ with $M(v)=\Sigma_{g_v}$ for any $v\in V(Q)$ such that $P_{\Sigma} = P(Q,M,\s)$. 
	Then, the wrapped Fukaya category of $P_{\Sigma}$ equipped with a grading structure $\eta\in H^1(Q;\Z)$ as in \eqref{eq:grading-plumbing} is, up to pretriangulated equivalence, given by
	\[\cW(P_{\Sigma};\eta)\simeq\cP^{\Sigma}_{2,\eta}[\{1+y_ex_e,\alpha_j^v,\beta_j^v\vb e\in E(Q), v\in V(Q), j\in\{1,\ldots,g_v\}\}^{-1}]\]
	where $\cP^{\Sigma}_{2,\eta}$ is a semifree dg category given as follows:
	\begin{enumerate}[label = (\roman*)]
		\item {\em Objects:} $L_v$ (representing a Lagrangian cocore dual to $M(v)=\Sigma_{g_v}$) for any $v \in V(Q)$.
		\item {\em Generating morphisms:} For any $v\in V(Q)$,
		\[h_v\colon L_v\to L_v ,\quad \alpha_j^v, \beta_j^v, \gamma_j^v, \delta_j^v\colon L_v\to L_v\quad\text{for }j=1, \ldots, g_v,\]
		and for any $e=v\to w\in E(Q)$,
		\[x_e\colon L_v \to L_w,\quad y_e\colon L_w \to L_v .\]
		\item {\em Degrees:} $|h_v|=-1,\quad |\alpha_j^v| = |\beta_j^v|= |\delta_j^v| =0,\quad |\gamma_j^v|= -1,\quad |x_e|=d_e,\quad |y_e|=-d_e$. 
		\item {\em Differentials:} $d \alpha_j^v = d \beta_j^v = d \delta_j^v =0,\quad d\gamma_j^v = \alpha_j^v \beta_j^v- \beta_j^v\alpha_j^v\delta_j^v, \quad d x_e = d y_e= 0$, and
		\[dh_v=\prod_{j=1}^{g_v}\delta_j^v\prod_{\substack{e= \bullet \to v\\ \s(e)=-1}} (1+x_ey_e)\prod_{e= v\to \bullet} (1+y_ex_e) - \prod_{\substack{e= \bullet \to v\\ \s(e)=1}}(1+x_ey_e)\]
	\end{enumerate}
	where $\langle d_e\in\Z\vb e\in E(Q)\rangle$ is any collection that satisfies $\eta=\sigma(\langle d_e\rangle)$, where $\sigma$ is given in \eqref{eq:grading-sigma}.
\end{cor}

\begin{rmk}\label{rmk:wfuk-plumbing-surface}\mbox{}
	\begin{enumerate}
		\item In Corollary \ref{cor:wfuk-plumbing-surface}, when $g_v=0$, $\Sigma_{g_v}$ is $S^2$ and $\prod_{j=1}^{g_v}\delta_j^v$ is the identity.
		
		\item The invertibility (up to homotopy) of $1+y_ex_e$ implies the invertibility of $1+x_ey_e$ for each $e\in E(Q)$ as in Remark \ref{rmk:plumbing-disks}\eqref{item:wfuk-disk-inv}.
		
		\item\label{item:reordering-surface} For each $v\in V(Q)$, the order of the terms in the differential of $h_v$ can be rearranged as explained in Remark \ref{rmk:plumbing-disks}\eqref{item:reordering-disk} (after setting $m_v=\prod_{j=1}^{g_v}\delta_j^v$). Any of these rearrangements gives us the same dg category $\cP_{2,\eta}^{\Sigma}$ up to quasi-equivalence.
	\end{enumerate}
\end{rmk}

\begin{proof}[Proof of Corollary \ref{cor:wfuk-plumbing-surface}]
	It is a direct corollary of Theorem \ref{thm:wfuk-plumbing-grading} by setting $M(v)=\Sigma_{g_v}$ for each $v\in V(Q)$. In more detail, for each $v\in V(Q)$, we have
	\[C_{-*}(\Omega_p (M(v)\setminus\text{pt}))=C_{-*}(\Omega_p \Sigma_{g_v,1})\]
	which is given by $\hom^*(L,L)$ in $\cW(T^*\Sigma_{g_v,1})$ by Proposition \ref{thm:cotangent-generation} and Remark \ref{rmk:loop-space}, where $L$ is a cotangent fiber of $T^*\Sigma_{g_v,1}$. Moreover, $\cW(T^*\Sigma_{g_v,1})$ is described in Proposition \ref{prp orientable surface with puctures} (setting $m=1$), where we relabel the generating morphisms by $\alpha_j^v, \beta_j^v, \gamma_j^v, \delta_j^v, a_1^v, h^v$ for $j=1, \ldots, g_v$ here. Then, for each $v\in V(Q)$, $\eta_v(z)=F_1(z)=a_1^v$ by Proposition \ref{prp orientable surface with puctures}\eqref{item:punctured-surface-inclusion}, after reversing the orientation of $\Sigma_{g_v,1}$ and making $F_1$ an orientation reversing inclusion. Note that we can set $h^v=0$ and identify $a_1^v$ with $\prod_{j=1}^{g_v} \delta_j^v$, which proves the corollary.
\end{proof}

\begin{rmk}\mbox{}
	\begin{enumerate}
		\item To get the wrapped Fukaya category $\cW(P_{\Sigma})$ of a plumbing $P_{\Sigma}$ with the {\em standard grading structure} as in Section \ref{subsection plumbings disks}, we set $d_e=0$ for each $e\in E(Q)$ in Corollary \ref{cor:wfuk-plumbing-surface}. Alternatively, $\cW(P_{\Sigma})$ can be obtained as a direct corollary of Theorem \ref{thm:wfuk-plumbing-general}.

		\item A special case of Corollary \ref{cor:wfuk-plumbing-surface} corresponds to \cite[Theorem 1 and Remark 19]{Etgu-Lekili19}, up to quasi-equivalence, when the grading structure is given by $d_e=0$ for all $e\in E(Q)$, and the coefficient ring $k$ is a field. See Remark \ref{rmk equivalence with the conventional plumbing} and Remark \ref{rmk simplified version of M_g,m with relations}.
	\end{enumerate}
\end{rmk}

We conclude this subsection with a comparison of wrapped Fukaya categories $\cW(P_{\Sigma};\eta)$ to derived multiplicative preprojective algebras, a result also presented in \cite{Etgu-Lekili19}:

\begin{cor}\label{cor:multiplicative-preprojective-algebra}
	Let $P_{\Sigma}$ be a plumbing of $T^*\Sigma_{g_v}$'s along any quiver $Q$ without negative intersections, where $\Sigma_{g_v}$ is an oriented closed surface of genus $g_v\geq 0$ for any $v\in V(Q)$. Then, the wrapped Fukaya category $\cW(P_{\Sigma};\eta)$ of $P_{\Sigma}$ with the grading structure $\eta=0\in H^1(Q;\Z)$ is Morita equivalent to a derived multiplicative preprojective algebra defined in \cite{crawley-boevey-shaw} when $g_v=0$ for all $v\in V(Q)$ and in \cite{bezrukavnikov-kapranov} for a general choice of $g_v$'s.
\end{cor}

\begin{proof}
	It is true by definition that if we define $L=\bigoplus_{v\in V(Q)} L_v$ in $\cW(P_{\Sigma};\eta)$, then $\hom^*(L,L)$ is a derived multiplicative projective algebra. Hence, we get the desired Morita equivalence.
\end{proof}

\begin{rmk}
	The converse of Corollary \ref{cor:multiplicative-preprojective-algebra} is also true, meaning that any derived multiplicative preprojective algebra can be expressed as $\cW(P_{\Sigma};\eta)$. Since we can reorder the terms in Corollary \ref{cor:wfuk-plumbing-surface} as explained in Remark \ref{rmk:wfuk-plumbing-surface}\eqref{item:reordering-surface}, we can also reorder the terms in derived multiplicative projective algebras. Additionally, by Proposition \ref{prp:wfuk-equivalent-1}, derived multiplicative preprojective algebras only depend on the underlying graph of the quiver $Q$. All these observations recover \cite[Theorem 1.4]{crawley-boevey-shaw}.
\end{rmk}

\bibliographystyle{amsalpha}
\bibliography{plumbing_topology_part_new_version}

\end{document}